\documentclass[12pt]{article}

\usepackage{amsmath}
\usepackage{amssymb}
\usepackage{amsthm}
\usepackage{amscd}
\usepackage{amsfonts}
\usepackage{stmaryrd}
\usepackage{bbding}
\usepackage{wasysym}
\usepackage{graphicx}
\usepackage{fancyhdr}
\usepackage{psfrag}

\newcommand{\Reflect}{\operatorname{Reflect}}
\newcommand{\FL}{\operatorname{FL}}
\newcommand{\PFL}{\operatorname{PFL}}
\newcommand{\Diam}{\operatorname{Diam}}
\newcommand{\Conv}{\operatorname{Conv}}
\newcommand{\Bd}{\operatorname{Bd}}
\newcommand{\Cl}{\operatorname{Cl}}
\newcommand{\Int}{\operatorname{Int}}
\renewcommand{\Re}{\operatorname{Re}}
\renewcommand{\Im}{\operatorname{Im}}
\newtheorem{thm}{Theorem}[section]
\newtheorem{cor}[thm]{Corollary}
\newtheorem{lem}[thm]{Lemma}
\newtheorem{Def}[thm]{Definition}
\newtheorem{prop}[thm]{Proposition}
\newtheorem{conj}[thm]{Conjecture}
\newtheorem{openq}[thm]{Open Question}

\title{The Dynamics of Twisted Tent Maps}
\author{Stephen Chamblee}		
\date{\today}					

\begin{document}
\maketitle	

\begin{abstract}
This paper is a study of the dynamics of a new family of maps from the complex plane to itself, which we call twisted tent maps.  A twisted tent map is a complex generalization of a real tent map.  The action of this map can be visualized as the complex scaling of the plane followed by folding the plane once.  Most of the time, scaling by a complex number will ``twist'' the plane, hence the name.  The ``folding'' both breaks analyticity (and even smoothness) and leads to interesting dynamics ranging from easily understood and highly geometric behavior to chaotic behavior and fractals.
\end{abstract}

\begin{section}{Introduction}

Real tent maps are piecewise-linear maps from the real line $\mathbb{R}$\label{symbol:R} to itself and have been extensively studied.  Real tent maps can be visualized as a real scaling followed by a folding.  Working on the complex plane $\mathbb{C}$\label{symbol:C} instead of $\mathbb{R}$ and replacing the real scaling by multiplication by a complex parameter $c$, we get a complex generalization of real tent maps.  Because multiplication by a complex number usually rotates (or twists) the complex plane, we call these new maps \emph{twisted tent maps} or TTM's.  This work attempts to both study and lay a foundation for the future study of TTM's.  

This subject has several inherent advantages for those who study TTM's.  The first is the simplicity of the map.  A twisted tent map sends a line segment to either a line segment or to a bent line segment.  This often gives rise to structures that can be studied using geometry, although quite often these structures are very complicated.  The second advantage is the ability to make pictures, which aid in intuition, understanding, and interest, since these pictures are often beautiful and pleasing on their own (see Figure \ref{fig:umbrella}).  The primary disadvantages of the subject are the lack of analyticity and smoothness due to folding.

\begin{figure}[htbp]
	\centering
		\includegraphics[width=0.80\textwidth]{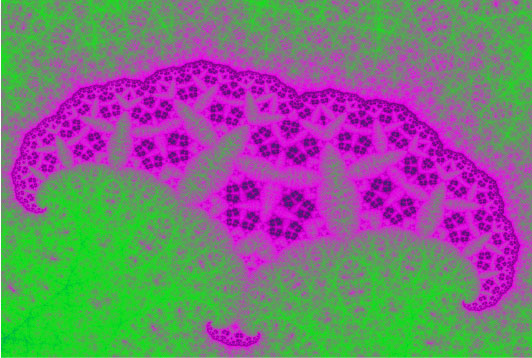}
	\caption{A structure in the coded-coloring of a filled-in Julia set.}
	\label{fig:umbrella}
\end{figure}

For us, the natural object of study is the filled-in Julia set $K=K(c)$, which is the set of all points whose trajectory stays bounded under forward iteration.  Frequently, a TTM restricted to a forward invariant subset of $K$ is conjugate (in the dynamical sense) to a real tent map.  Depending on the choice of $c$, $K(c)$ can be a line segment, a double spiral of line segments, a polygon, a fractal, or even a Cantor set.  In this work, necessary and sufficient conditions will be given for $K$ to be a polygon.  Figure \ref{fig:TypesOfK} shows a few examples of $K$.

\begin{figure}[htbp]
	\centering
		\includegraphics[width=1.00\textwidth]{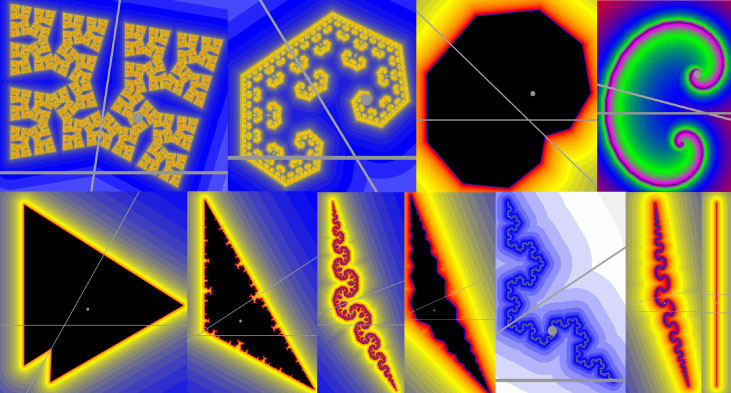}
	\caption{A few examples of filled-in Julia sets for TTM's.}
	\label{fig:TypesOfK}
\end{figure}

For the most interesting TTM's there are no attracting periodic points in $\mathbb{C}$.  Instead, there can be sets of points we call \emph{hungry sets} that attract neighboring points, ``eat'' them, and retain them.  These hungry sets can be simply connected, a topological annulus, or be comprised of periodic components.  There can be several distinct hungry sets in a filled-in Julia sets and some hungry sets can contain smaller hungry sets.

A very standard way to make pictures of filled-in Julia sets is to color pixels based on the escape-time algorithm (such as in Figure \ref{fig:TypesOfK}).  However, when $K$ has nonempty interior, the escape-time algorithm shows only black throughout the interior of $K$.  We introduce a new coloring algorithm called the \emph{coded-coloring algorithm} which makes it possible to see periodic structures and their preimages in the interior of $K$.  Some of these structures behave like periodic copies of filled-in Julia sets for a different choice of parameter (see Figure \ref{fig:TBcoloringDecomposition}).  

\begin{figure}[htbp]
	\centering
		\includegraphics[width=0.40\textwidth]{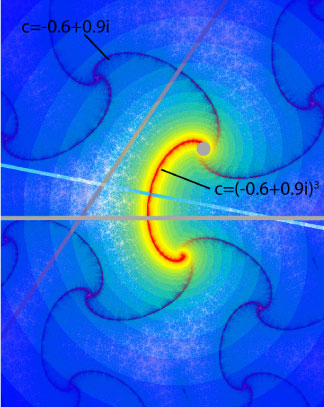}
	\caption{Highlighted is $K(c^3)$ which appears in the coded-coloring of $K(c)$.}
	\label{fig:TBcoloringDecomposition}
\end{figure}

The coded-coloring can also be used in making pictures of the parameter plane.  Any structures that appear near $c$ in the coded-coloring of the parameter plane are approximate previews of the structures found in the coded-coloring of $K(c)$ (see Figures \ref{fig:MandelbrotOutside} and \ref{fig:MandelbrotInside}).

\begin{figure}[htbp]
	\centering
		\includegraphics[width=1.0\textwidth]{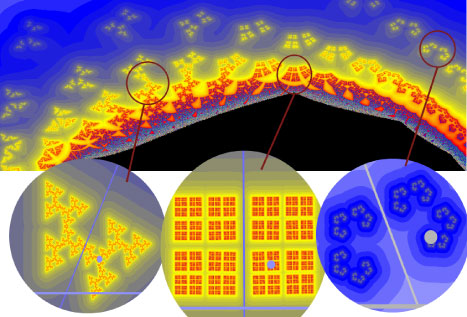}
	\caption[An escape-time coloring of the parameter plane.]{The escape-time algorithm of the parameter plane near $c$ gives a preview of the structure found in $K(c)$ when $K(c)$ has empty interior.}
	\label{fig:MandelbrotOutside}
\end{figure}

\begin{figure}[htbp]
	\centering
		\includegraphics[width=1.0\textwidth]{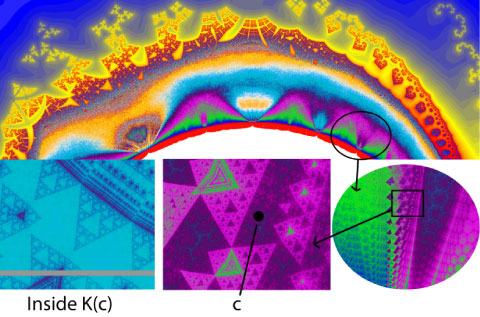}
	\caption[The coded-coloring of the parameter plane]{The coded-coloring of the parameter plane near $c$ gives a preview of structures found in $K(c)$.}
	\label{fig:MandelbrotInside}
\end{figure}
 
There have been several generalizations of real tent maps, such as the c-tent maps which are defined in \cite{polymodials}.  We believe that TTM's are also suitable subjects to study, due to the simplicity of the formula, the complexity of the dynamical behavior, and to the beauty of the pictures.  Further motivation for the study of TTM's include the following:

\begin{enumerate}
	\item This family of maps is new, interesting, and simple enough to ensure the successful completion of the requirement of a thesis.
	\item A similar map is studied in \cite{Arnold}.
	\item It is shown in \cite{antennae} that fractals (for example Koch curves) can be used as designs for frequency independent antennae.  Koch curves appear as filled-in Julia sets for certain TTM's.  Thus, it is possible that TTM's could be useful in finding fractals that have not yet been considered for this purpose.
\end{enumerate}

\begin{figure}[htbp]
	\centering
		\includegraphics[width=0.70\textwidth]{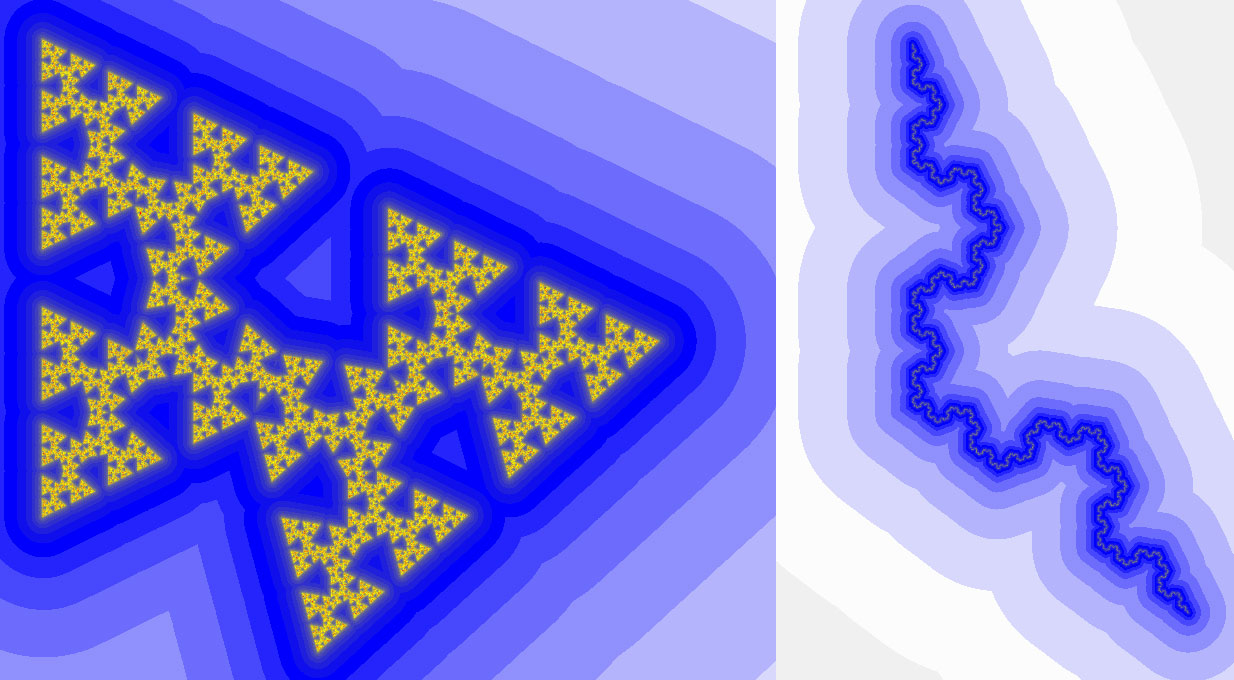}
	\caption[$K$ resembling fractal antenna designs.]{Examples of filled-in Julia sets for TTM's that resemble currently used fractal antenna designs.}
	\label{fig:antenna}
\end{figure}

This present work is organized in the following way:

Chapter \ref{Basic Cases} provides the formal definition of a TTM, some of the basic terms, and shows that due to conjugacy (both topological conjugacy and complex conjugation) there is a canonical subset of the family of twisted tent maps that are sufficient representatives of the entire family.  This subset consists of TTM's whose folding line is the set of points with imaginary part equal to $-1$ and whose parameter has non-negative imaginary part.  For a canonical TTM we show that if $|c|<1$ then the dynamics are trivial.  We then study the dynamics that can occur if $|c|=1$ and finish the chapter with a list of assumptions that will be held throughout the rest of the text.

Chapter \ref{The Filled-in Julia Set} describes $K$ when $c \in \mathbb{R}$.  In this chapter we also give some partial results about the connectivity of $K$ as well as several conjectures.  We also give a several different sufficient criteria for $K$ to be a Cantor set.

In Chapter \ref{The Perimeter Set $P$} we define and use the \emph{perimeter set} $P$, which is the primary tool we use to categorize $K$.  The definition of $P$ was chosen such that the following are true.

\begin{itemize}
	\item $P$ is compact,
	\item $P=K$ for some cases,
	\item $P$ always contains $K$,
	\item $P$ is easily calculated,
	\item $P$ gives us significant information about $K$.
\end{itemize}

In the process of defining $P$, we show that when $K$ is a polygon, then its boundary can be explicitly calculated.

Chapter \ref{Hungry Sets and Crises} lays down a framework for the study of hungry sets and gives the proofs for a few results.  The rest of this chapter is broken up into sections primarily consisting of experimental results.  These experimental results include examples where hungry sets swell in size as the modulus of the parameter is increased as well as examples illustrating that the tools of renormalization might be useful in studying TTM's.

In Chapter \ref{Partitioning the Parameter Space} describes what we know about different pictures of the parameter plane.  These pictures include the \emph{polygonal locus}, which is the set of parameters $c$ such that $K(c)$ is a polygon.  The coded-coloring of the parameter plane is discusses and several open problems and conjectures can be found in this discussion.  

Chapter \ref{Entropy} proves several results about the entropy of a TTM, such as Theorem \ref{strict entropy} which states that $h(f|_K) \leq \log \min(2,|c|^2)$ and that this inequality is sometimes strict.

Lastly, scattered throughout this text are conjectures and open questions to aid those who wish to continue to develop this theory.

\vspace{0.1in}

\noindent
{\bf Aknowledgements}\\
\noindent
This is my doctoral thesis for my Ph.D. in mathematics from IUPUI. I would like to thank Micha\l~Misiurewicz, my advisor, for patiently working with me for many hours to develop this theory, while sharing his love of mathematics with me. I would also like to thank the dynamics group at IUPUI for their many helpful suggestions. I am grateful to the U.S. Department of Education because this work was supported by the GAANN grant 42-945-00.

\end{section}

\begin{section}{Basic Cases}\label{Basic Cases}

Let $\FL$ denote a line in the complex plane.\label{symbol:FL}  ($\FL$ will later be called the \emph{folding line}.)  Then $\FL$ divides the plane into two closed half planes $\mathbb{H}^-$ and $\mathbb{H}^+$, where $0 \in \mathbb{H}^+$.  (If $0 \in \FL$ then the choice of $\mathbb{H}^+$ is arbitrary.)  Let $\Reflect(z,\FL)$\label{symbol:Reflect} be the reflection of the point $z$ about the line $\FL$.  Let $c \in \mathbb{C}$ and fix $\FL$.  We wish to study the dynamics of a family of functions $\{f_c\}$ where $f_c:\mathbb{C} \to \mathbb{C}$ is of the form:

\begin{equation}
\label{eq:Definition of f_c}
f_c(z) =
\begin{cases}
cz & \text{if } cz \in \mathbb{H}^+, \\
\Reflect(cz,\FL) & \text{if } cz \in \mathbb{H}^-. 
\end{cases}
\end{equation}

If $0 \in \FL$ then the dynamics are trivial.  Let $r=|c|$.

\begin{thm}
Assume $0\in \FL$.  Then every point is eventually mapped into the invariant sector $S$ with the negative real axis as its bottom edge, vertex 0, and with angle $\theta$\label{symbol:theta} measured clockwise from the negative real axis. 
\begin{enumerate}
	\item If $r<1$ then $0$ is an attracting fixed point whose basin of attraction is $\mathbb{C}$.
	\item If $r=1$, then for all $z \in S$, $f^2(z)=z$. 
	\item If $r>1$, then the orbit of every nonzero point diverges to infinity.
\end{enumerate}
\end{thm}
\begin{proof}
Each of these follow easily after the observation that since $0 \in \FL$, then reflection about $\FL$ does not affect the modulus of a point. 
\end{proof}

From now on, we will assume that $0 \notin \FL$.  Theorem~\ref{The choice of FL is arbitrary} shows that all choices of $\FL$ that do not pass through the origin are equivalent.  This allows us to make a canonical choice for $\FL$ which simplifies things greatly.  

\begin{thm}\label{The choice of FL is arbitrary}
Let $\FL_F$ and $\FL_G$ be two different choices for $\FL$.  Let $F=\{f_c:c \in \mathbb{C}\}$ and $G=\{g_c:c\in \mathbb{C}\}$ be the resulting families of functions with those choices of folding lines, defined as in \eqref{eq:Definition of f_c}.  Then for every $c \in \mathbb{C}$, $f_c$ is conjugate to $g_c$.
\end{thm}
\begin{proof}
There exists a unique circle centered at the origin such that the line $\FL_F$ is tangent to it at a point $z_F$.  Similarly, we define $z_G$ to be the point where $\FL_G$ is tangent to some circle centered at the origin.  Note that the points $z_F$ and $z_G$ depend only on the choice of $\FL$ and not on $c$.  We denote the closed half planes on either side of $\FL_F$ by $\mathbb{H}_F^-$ and $\mathbb{H}_F^+$.  Likewise, the closed half planes that meet along $\FL_G$ are denoted by $\mathbb{H}_G^-$ and $\mathbb{H}_G^+$.  Now let $\varphi : \mathbb{C} \to \mathbb{C}$ be define by $\varphi(z)=\frac{z_G}{z_F}z$.  Since neither $\FL_F$ nor $\FL_G$ pass through the origin, then $z_F,z_G \neq 0$.  Thus $\varphi$ and its inverse are well defined.  The following hold trivially:

\begin{enumerate}
	\item $\varphi$ is a homeomorphism.
	\item $\varphi(cz)=c\varphi$ for every $c \in \mathbb{C}$.
	\item $\varphi(z_F)=z_G$
	\item $\varphi(FL_F)=\{\varphi(z):z\in \FL_F \}=FL_G$
	\item $\varphi(\mathbb{H}_F^+)=\mathbb{H}_G^+$
	\item $\varphi(\mathbb{H}_F^-)=\mathbb{H}_G^-$
	\item $\varphi(\Reflect(z,FL_F))=\Reflect(\varphi(z), \FL_G)$
\end{enumerate}

Thus, for every $c \in \mathbb{C}$, if $z \in \mathbb{H}_F^+$ then $\varphi(f_c(z))=\frac{z_G}{z_F}cz=c\left(\frac{z_G}{z_F}z \right)=g_c(\varphi(z)).$  Similarly, for every $c \in \mathbb{C}$, if $z \in \mathbb{H}_F^-$ then
\begin{equation}
\varphi(f_c(z))=\varphi(\Reflect(cz,\FL_F))=\Reflect(c\varphi(z),\FL_G)=g_c(\varphi(z)).
\end{equation}

Thus $\varphi(f_c(z))=g_c(\varphi(z))$ as desired.

\end{proof}

\begin{Def}\label{symbol:Im}\label{symbol:Re}
We will denote the real and imaginary parts of $z$ by $\Re(z)$ and $\Im(z)$ respectively.
\end{Def}

Theorem~\ref{The choice of FL is arbitrary} allows us to define $\FL=\{z:\Im(z)=-1\}$ as our canonical choice for $\FL$.  This will be our choice for $\FL$ throughout the rest of this paper and we will denote the family of functions induced by this choice of $\FL$ by $\mathcal{F} =\{f_c:c \in \mathbb{C}\}$.  Now that a specific family of functions has been chosen, we will write $\FL$ instead of $\FL_F$ and will also drop the family subscript for the half-planes and write $\mathbb{H}^+$ and $\mathbb{H}^-$.  Throughout this paper we will write $\overline{z}$ for the complex conjugate of $z$.

Now, for every $z \in \mathbb{C}$ we have $\Reflect(z,FL)=\overline{z+i}-i=\overline{z}-2i$.  Thus, for every $f_c \in F$ we have:

\begin{equation}
f_c(z)= 
\begin{cases}
cz & \text{if } \Im(cz) \geq -1, \\ 
\overline{cz} -2i & \text{if } \Im(cz) \geq -1. 
\end{cases}
\end{equation}

Each $f_c \in F$ scales and rotates the plane and then reflects $\mathbb{H}^-$ across $\FL$.  This family of maps is a generalization of the family of real tent maps.  


Recall that the tent map acting on $\mathbb{R}$ scales the real line by a real constant and then folds it at a particular point on $\mathbb{R}$.  A twisted tent map scales $\mathbb{C}$ by a complex constant and then folds it across the $\FL$.  More specifically, a tent map scales and folds; a twisted tent map scales, rotates (twists), and folds.  This is why we call each $f_c$ a \textbf{twisted tent map}.   

If $f$ is a twisted tent map, then it is fairy common for a one dimensional forward invariant subset of the complex plane to have the dynamics of a tent map.  For example, when $c=2$ the line segment joining the origin and $-i$ has the same dynamics as the tent map $\mathcal{G}:[0,1] \to [0,1]$ defined by:

\begin{equation}
\mathcal{G}(x)=\left\{
\begin{array}{lrlr} 2x &, x < 0.5 \\ 
2-2x &, x \geq 0.5. 
\end{array}\right.
\end{equation}

The following proposition describes the dynamics of the simplest TTM's with canonical $\FL$.

\begin{prop}\label{mod c less than 1}
If $f_c$ is a twisted tent map with $|c|<1$ then every point in the complex plane is attracted to the origin.
\end{prop}
\begin{proof}
Noting that $|\Reflect(z,FL)|\leq |z|$ for all $z \in \mathbb{H}^-$ we see that $|f(z)|\leq |c||z|$ for all $z$.  Denoting the n$^{th}$ iterate of $f$ by $f^n$, we have $|f^n(z)|\leq |c|^n|z|$.  Since $|c|<1$ then $|f^n(z)|\to 0$ as $n \to \infty$. 
\end{proof}

The next proposition shows that $f_c$ is conjugate (in the dynamical sense) to $f_{\overline{c}}$.  Thus, without loss of generality we can focus our study on parameters $c=\alpha + \beta i$\label{symbol:alpha}\label{symbol:beta} where $\beta \geq 0$.  Written another way, if $c=re^{i \theta}$ then we may always assume that $\theta \in [0, \pi]$.

\begin{prop}\label{We only need consider beta>0}
$f_c$ is conjugate to $f_{\overline{c}}$.
\end{prop}
\begin{proof}
Let $\varphi(z)=-\overline{z}$ so that $\varphi$ is reflection about the imaginary axis.  We will show that $\varphi(f_{c}(z))=f_{\overline{c}}(\varphi(z))$.  That is, we wish to show that the following diagram commutes.
\begin{equation}
\begin{CD}
\mathbb{C} @>f_c>> \mathbb{C} \\
@V \varphi VV 			@VV \varphi V\\
\mathbb{C} @>f_{\overline{c}}>> \mathbb{C}\\
\end{CD}
\end{equation}

We have that $\overline{c}\varphi(z)=\overline{c}(-\overline{z})=-\overline{cz}=\varphi(cz)$.  Since $\varphi$ fixes the imaginary part of its argument, this implies that $\Im(\overline{c}\varphi(z)=\Im(cz)$.  Thus, $cz \in \mathbb{H}^-$ exactly when $\overline{c}\varphi(z) \in \mathbb{H}^-$.  Because of this, we can show that $\varphi(f_{c}(z))=f_{\overline{c}}(\varphi(z))$ by proving only two cases.

Case 1:  $\Im(cz)\geq -1$
\begin{equation}
\varphi(f_{c}(z))= \varphi (cz)=-\overline{cz} = \overline{c}(-\overline{z})=f_{\overline{c}}(\varphi(z))\\
\end{equation}

Case 2:  $\Im(cz) < -1$
\begin{equation}
\varphi(f_c(z))=\varphi(\overline{cz}-2i)=-\overline{(\overline{cz}-2i)}=\overline{\overline{c}(-\overline{z})}-2i=f_{\overline{c}}(\varphi(z))\\
\end{equation}

This shows that $\varphi$ semi-conjugates $f_c$ with $f_{\overline{c}}$.  Lastly, since $\varphi$ is an isometry, then it conjugates $f_c$ and $f_{\overline{c}}$.
\end{proof}

The dynamics of a TTM become more interesting when $|c|=1$.  We need the following.

\begin{Def}\label{symbol:PFL}
We will use $ \PFL$ to denote the \textbf{pre-folding line}, which is the preimage of $\FL$ under $f$.  This is also equal to the set of points $\{z/c:z \in \FL\}$. 
\end{Def}

\begin{Def}\label{Pre-half planes}
The $ \PFL$ partitions the plane into two half planes.  We define the \textbf{pre-upper half plane} as $P\mathbb{H}^+=\{z: f(z)=cz\}$ and define the \textbf{pre-lower half plane} as $P\mathbb{H}^-=\{z: f(z)=\overline{cz}-2i\}$.  Note that both of the pre-half planes contain $ \PFL$.
\end{Def}

The $ \PFL$ plays an important role since it is an axis of symmetry.  That is, if $z'=\Reflect(z, \PFL)$, then $f(z')=f(z)$.  Also note that an alternative and equivalent way to define $f$ would be to reflect every point in $P\mathbb{H}^-$ about the $ \PFL$ and then multiply the result by $c$.  That is, fold then multiply.  This way of visualizing the action of $f$ is sometimes helpful.

\begin{Def}
If $c \notin \mathbb{R}$ then the intersection between the $\FL$ and the $ \PFL$ is unique.  In this case we define $\{\gamma_0\} = \FL \cap  \PFL$.\label{symbol:gamma}  
\end{Def}

\begin{lem}\label{Equation of gamma_0}
We have $\gamma_0=\frac{\alpha-1}{\beta}-i$.
\end{lem}
\begin{proof}
Clearly $\frac{\alpha-1}{\beta}-i \in \FL$.  Also, 
\begin{equation}
\begin{aligned}
f \left(\frac{\alpha-1}{\beta}-i\right)&=c \left(\frac{\alpha-1}{\beta}-i\right) = (\alpha + \beta i) \left(\frac{\alpha -1}{\beta}-i \right)\\ 
&= \frac{\alpha^2-\alpha}{\beta}-\alpha i +\alpha i -i +\beta\\
&= \frac{\alpha^2 +\beta^2 -\alpha}{\beta}-i = \frac{|c|^2-\alpha}{\beta}-i \in \FL\\
\end{aligned}
\end{equation}
\end{proof}

\begin{Def}\label{symbol:Disk}
We will write $\overline{\mathbb{D}}=\{z:|z| \leq 1\}$.
\end{Def}

\begin{lem}\label{If |c|=1 then |f(z)| is <= |z|}
If $|c|=1$, then $|f(z)| \leq |z|$ with equality exactly when $cz \in \mathbb{H}^+$.
\end{lem}
\begin{proof}
If $cz \in \mathbb{H}^+$ then $|f(z)|=|cz|=|z|$ since by assumption $|c|=1$.  Now assume that $cz$ is in the interior of $\mathbb{H}^-$.  First note that $\Re(f(z))=\Re(cz)$.  Next, since $cz$ is in the interior of $\mathbb{H}^-$ then $\Im(cz) < -1$ and so $\varepsilon=min\{1,-1-\Im(cz)\}> 0$.  Noting that $\varepsilon$ is the smallest distance from $cz$ to $\FL$, we have $|\Im(f(z))|=|\Im(cz)+2 \varepsilon| < |\Im(cz)|$.  Thus, $|f(z)|=|\Re(cz)+i(\Im(cz)+2 \varepsilon)|< |\Re(cz)+i(\Im(cz))|=|cz|=|z|$. 
\end{proof}

\begin{figure}[htbp]
	\centering
		\includegraphics[width=1\textwidth]{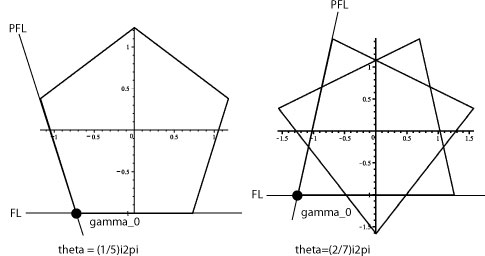}
	\caption{Construction of a convex polygon P for Theorem~\ref{mod c equals 1} case 3.}
	\label{fig:Modulus_of_c_is_1}
\end{figure}

\begin{thm}\label{mod c equals 1}
If $c=e^{i \theta}$ for some $\theta$ with $0 \leq \theta <2\pi$.  Then four cases can occur:

\begin{enumerate}
	\item If $\theta=0$ then $f^n(z)=f(z)\in \mathbb{H}^+$ for $n=1,2,3,...$.
	\item If $\theta = \pi$ then every point is eventually mapped into a horizontal strip of height 2.  For every point $z$ in this strip, $f^2(z)=z$.
	\item If $\theta=\frac{j}{k} 2\pi$, where
	\begin{itemize}
		\item $\frac{j}{k}$ is written in lowest terms,
		\item $j,k \in \{1,2,3,...\}$,
		\item $j < k\neq 2$,
	\end{itemize}
then there is a regular periodic polygon, centered at the origin, with period $k$, such that the orbit of every point eventually intersects the polygon. 
	\item If $\theta$ is an irrational multiple of $\pi$, then for every point $z$ there exists a point $q \in \overline{\mathbb{D}}$ such that $|f^n(z)-f^n(q)|$ goes to $0$ as $n$ goes to infinity.
\end{enumerate}
\end{thm}
\begin{proof}
Case 1:  Let $\theta =0$ so that $c =1$.  The first image of the complex plane is always $\mathbb{H}^+$ and since $c=1$ then $(f_c)$ restricted to $\mathbb{H}^+$ is the identity map.  

Case 2:  Let $\theta = \pi$ so that $c=-1$ and let $z=a+bi \in \mathbb{C}$ be given.  Since $f(\mathbb{C})=\mathbb{H}^+$ then it is enough to prove the claim only for points in $\mathbb{H}^+$.  Let $S=\{z:-1 \leq \Im(z) \leq 1 \}$.  If $z \in S \subset \mathbb{H}^+$ then $f(z)=cz=-z \in S$ and $f^2(z)=-(-z)=z$.  Thus, every point in $S$ is a periodic point of period 2.  Now for every $z \in (\mathbb{H}^+ \setminus S)$ we have that $z = a+bi, b>1$ and so $f(z)=\overline{cz}-2i=-a+bi-2i.$  Thus if $z \in (\mathbb{H}^+ \setminus S)$ then $f$ negates the real part of $z$ subtracts $2i$.  This means if the first $n-1$ iterates of $z$ are in $\mathbb{H}^+ \setminus S$, then $f^n(z)=(-1)^n a+bi-2ni.$  We now write $b=2m+r$ where $m$ is a non-negative integer and $0\leq r <2$.  Then $\Im(f^m(z))=b-2m=2m+r-2m=r<2$.  Now if $r \leq 1$ then $f^m(z) \in S$ and we are done.  Otherwise, $f^{m+1}(z) \in S$.

Case 3:  Let $\theta=\frac{j}{k} 2\pi$, where
	\begin{itemize}
		\item $\frac{j}{k}$ is written in lowest terms,
		\item $j,k \in \{1,2,3,...\}$,
		\item and $j < k$.
	\end{itemize}

(See Figure~\ref{fig:Modulus_of_c_is_1})  If $j\in \{1,k-1\}$, then the line segments $[\gamma_0, f(\gamma_0)]$, $[f(\gamma_0),f^2(\gamma_0)]$, ..., $[f^{k-1}(\gamma_0),\gamma_0]$ form the boundary of a simply connected convex polygon $P$ centered at $0$ with $k$ sides.  By definition $P$ is periodic with period $k$.  If $j \notin \{ 1, k-1\}$, then the line segments $[\gamma_0, c\gamma_0]$, $[c\gamma_0,c^2\gamma_0]$, ..., $[c^{k-1}\gamma_0,\gamma_0]$ form a $k$-pointed star which partitions the plane and has the origin as the center of the star.  By symmetry, the piece of the partition containing the origin is a regular polygon $P$ with $k$ sides.  By the definition of $\gamma_0 \in \FL$, we have $f(\gamma_0)=c\gamma_0 \in \FL$, and so the bottom-most edge of $P$ lies on $\FL$.  In both cases, it is evident that $P = \{z: cf^n(z) \in \mathbb{H}^+, n=0,1,2,...\}$ and so $f(P)=P$.  Note that $f^k(z)=z$ for all $z \in P$.

We now show that every point in the plane is eventually mapped into $P$.  Choose $z \in \mathbb{C}$ and let $D$ be a closed disk, centered at the origin, such that $z \in D$.  If $cz \notin \mathbb{H}^+$, then by Lemma~\ref{If |c|=1 then |f(z)| is <= |z|} $|f(z)|<|z|$.  It is easily seen that for every $z \notin P$ we have $cf^n(z) \notin \mathbb{H}^+$ at least once as $n$ takes on the values $1,2,3,...,k$.  Thus, every $k$ iterates, a point either lands in $P$ or gets closer to the origin.  Consider the sequence of iterates of $z$, $(z_n)$, $n=1,2,3,...$.  Then since $D$ is compact, and since the iterates of $z$ are contained in $D$, then this sequence has an accumulation point $q$ inside of $D$.  If $q$ is in the interior of $P$ then we are done.  If $q \notin P$ then by Lemma~\ref{If |c|=1 then |f(z)| is <= |z|} $|f^k (q)| < |q|$ which contradicts the assumption that $q$ was an accumulation point.  However, we still have the possibility that $q$ is on the boundary of $P$.

We now show that every point sufficiently close to $P$ gets mapped into $P$.  First consider the special case when $k=3$.  Then $j$ is either 1 or 2.  But Proposition~\ref{We only need consider beta>0} implies that it is sufficient to assume $j=1$.  (If $j=2$ then $c=e^{i \theta}$ where $\theta \notin [0,\pi]$.)  When $\frac{j}{k}=\frac{1}{3}$, $P$ is an equilateral triangle with the bottom edge on $\FL$.  Figure~\ref{fig:regular_polygons} shows a neighborhood $U_P$ of $P$ that has been partitioned by triangles congruent to $P$ and shows that $f^5(U_P)=P$.

\begin{figure}[htbp]
	\centering
		\includegraphics[width=0.75\textwidth]{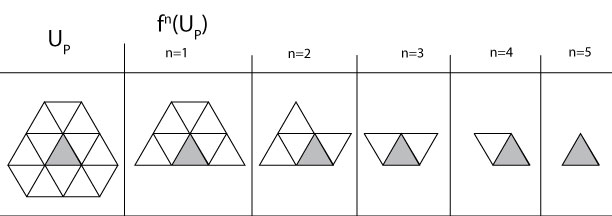}
	\caption[Partitions of $U_P$]{Every point sufficiently close to $P$ is eventually mapped into $P$.  These diagrams show the partitions of $U_P$ and show the first few iterates of $U_P$ in the case that $j=1$.  $P$ is shaded.}
	\label{fig:regular_polygons}
\end{figure}

Now assume that $k \geq 4$.  (See Figure~\ref{fig:Convex_Polygon_closeup})  Let $S$ be the union of lines collinear to the segments forming the $k-$pointed star in the construction of $P$.  We define $U_P=\{(1+\varepsilon)z:z \in P\}$ for $\varepsilon >0$ where $\varepsilon$ is small enough that for every $z \in U_P$ the line segment $[z,0]$ intersects at most 2 lines in $S$.  Note that $S$ partitions $U_P$ into the sets $A,$ $B,$ and $P$ as shown in Figure~\ref{fig:Convex_Polygon_closeup}.  During each iteration, multiplication by $c$ rotates one connected component of $B$ and two connected components of $A$ below $\FL$.  After this, folding maps the component of $B$ into $P$ and maps the two components of $A$ into $P \cup B$.  It is now easily seen that $f^k(B)\subset P$, $f^k(U_P)\subset (P \cup B)$, and thus $f^{2k}(U_P)=P$.

Case 4:  Let $\theta$ be an irrational multiple of $\pi$.  Fix $z_0 \in \mathbb{C}$ and let $D$ be a closed disk containing $\overline{\mathbb{D}}\cup \{z\}$ and centered at the origin.  It is easily seen that since $|c|=1$ then $|f(z)-f(w)| \leq |z-w|$ for all $z,w$.  In particular, when $w=0$ this means that $f$ cannot increase the modulus of any point.  For this reason and since $D$ is centered at the origin, then the iterates of $z$ are contained in $D$.  Since $D$ is compact, then the sequence of iterates $(z_n)=\left(\frac{f^n(z_0)}{c^n}\right)$, $n=1,2,...$ contains a subsequence $(z_{n_j})$ that converges to an accumulation point $q\in D$.  Then $|z_{n_j}-q|$ goes to $0$ as $j$ goes to infinity.  If $q \notin \overline{\mathbb{D}}$ then for some $k>0$ $f^k(q)\in P\mathbb{H}^-$ and by Lemma~\ref{If |c|=1 then |f(z)| is <= |z|} $|f^{k+1}(q)|<|q|$, which contradicts the assumption that $q$ is an accumulation point.  Thus, $q \in \overline{\mathbb{D}}$.  

\begin{figure}[htbp]
	\centering
		\includegraphics[width=1.0\textwidth]{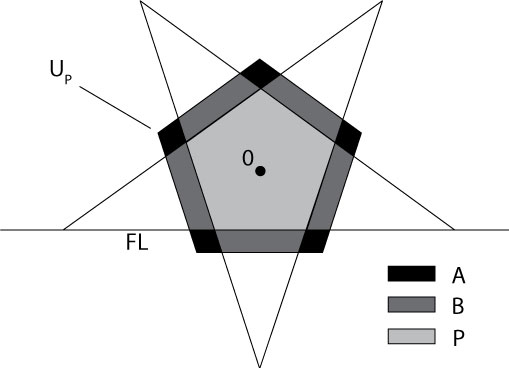}
	\caption{A partitioning of the neighborhood of $U_P$ when $k=5$.}
	\label{fig:Convex_Polygon_closeup}
\end{figure}

Now consider the sequence $(|z_n-q|)$ for $n=1,2,...$.  Note that since $f$ cannot increase the distance between two points, then $|f^{n+1}(z_0)-f^{n+1}(q)|\leq |f^n(z_0)-f^n(q)|$.  We have 
\begin{equation}
\begin{aligned}
|z_{n+1}-q|&=\left|\frac{f^{n+1}(z_0)}{c^{n+1}}-q\right|\\
&=\frac{1}{|c^{n+1}|}\left|f^{n+1}(z_0)-c^{n+1}q\right|\\
&=|f^{n+1}(z_0)-f^{n+1}(q)|\\
&\leq |f^n(z_0)-f^n(q)|\\
&=\frac{1}{|c^n|}|f^n(z_0)-c^nq|\\
&=\left|\frac{f^n(z_0)}{c^n}-q\right|=|z_n-q|.\\
\end{aligned}
\end{equation}

Thus, the sequence $(|z_n-q|)$ is a decreasing sequence of nonnegative numbers and (by the Monotone Convergence Theorem of real numbers) has a limit.  Since the subsequence $(|z_{n_j}-q|)$ goes to $0$ as $j$ goes to infinity, then this limit must be zero.

Lastly,
\begin{equation}
\begin{aligned}
|f^n(z_0)-f^n(q)|&=|f^n(z_0)-c^n q|\\
&=|c^n|\left|\frac{f^n(z_0)}{c^n}-q\right|\\
&=|z_n-q|\\
\end{aligned}
\end{equation}
which we just showed goes to zero as $n$ goes to infinity.
\end{proof}

As we will see, much more interesting dynamics occur when $|c|>1$.
  Thus, unless explicitly mentioned, for the rest of this paper we will assume that $|c|>1$.  The following list of assumptions is for the reader's reference.

\textbf{Standing Assumptions:}  We will use the following definitions and assumptions throughout this text.  They are listed here for the reader's convenience.
\begin{enumerate}
	\item $|c|>1$. 
	\item $\Im(c)=\beta > 0$.  (See Proposition~\ref{We only need consider beta>0})
	\item $c=\alpha + \beta i=|c|e^{i\theta},0\leq \theta<2 \pi$.
	\item $f_c=f$ is a twisted tent map with folding line consisting of points with imaginary part equal to $-1$.  We will only write $f_c$ if the choice of $c$ needs to be explicitly stated.
	\item $g_c(z)=g(z)=cz$.\label{symbol:g}
	\item $h_c(z)=h(z)=\overline{cz}-2i$.\label{symbol:h}
	\item We will use $ \PFL$ to denote the \textbf{pre-folding line}, which is the preimage of $\FL$ under $f$.
	\item We will denote the line segment between $z,w \in \mathbb{C}$ by $[z,w]=[w,z]$.
\end{enumerate}
\end{section}

\begin{section}{The Filled-in Julia Set}\label{The Filled-in Julia Set}
Some of the most notable objects of study in quadratic complex dynamics on the complex plane (or Riemann sphere) are the Julia set $J$, filled-in Julia set $K$, and the Fatou set $F$.  We now define these notions for Twisted Tent Maps, give examples, and note some of their key differences with their counterparts in quadratic dynamics.  

Note that each TTM can be extended to a map from the Riemann sphere to the Riemann sphere by defining $f(\infty)=\infty$.  These maps will also be called TTM's and it should be clear from context which definition of TTM's we are using.

\begin{Def}\label{symbol:K}
The \textbf{filled-in Julia set} is denoted by $K(f_c)=K(c)=K$ and is the set of points with bounded trajectory under iterations of $f$.  The \textbf{Fatou set} is denoted by $F(f_c)=F(c)=F$ and is the domain of equicontinuity in the Riemann sphere for the family of iterates of $f$.  The \textbf{Julia set}\label{symbol:J} is denoted by $J(f_c)=J(c)=J$ and is the complement of $F$.
\end{Def}

\begin{Def}
We denote the boundary of a set $X$ by $\Bd(X)$.\label{symbol:Bd}
\end{Def}

The following is a short list of dynamical properties of TTM's that are distinct from their rational counterparts.  
\begin{enumerate}
	\item The Julia set often does not equal $\Bd(K)$,
	\item The Julia set can have isolated points,
	\item The Filled-in Julia set can have disjoint connected components,
	\item The Filled-in Julia set can have disjoint components with nonempty interior,
	\item The Julia set can have nonempty interior.
\end{enumerate}

\begin{Def}
We will call a point \textbf{preperiodic} if it is strictly preperiodic.  We will call a point \textbf{eventually periodic} if it is either periodic or preperiodic.
\end{Def}

\begin{Def}\label{symbol:diameter}
The diameter of a set $X$ will be denoted $\Diam(X)$.
\end{Def}

\begin{prop}\label{noConvergence}
Let $U$ be a nonempty open subset of $K$.  If $\underset{m \to \infty}{\lim} \Diam(f^m(U))=0$, then $U$ contains no eventually periodic points.  
\end{prop}
\begin{proof}
Let $\underset{m \to \infty}{\lim} \Diam(f^m(U))=0$.  If $U$ contains an eventually periodic point $z_0$, then let $z$ be a periodic point of period $n>0$ in the orbit of $z_0$.  If each of the $n$ points in the periodic orbit of $z$ is on $\FL$, then every short line segment starting from from a point in this orbit increases in length under iteration, a contradiction.  Thus, at least one point of this orbit is not in $\FL$.  Let  $\delta=\min\{|\Im(f^n(z))-(-i)|\neq 0:n=1,2,\cdots,n\}.$  Then $\delta$ is the shortest non zero distance from a point in the orbit of $z$ to $\FL$.  Clearly, every line segment of length $\delta/2$ starting from a point in this orbit will increase in length under one application of $f$.  This also is a contradiction and we are done.
\end{proof}

\begin{prop}\label{One attracting periodic point}
The point at infinity is the only attracting periodic point.
\end{prop}
\begin{proof}
Since $|c|>1$ then multiplication by $c$ increases the modulus of a point.  Folding can only decrease the imaginary part of a point, and can decrease the modulus of a point by at most $2$.  Thus, $|f(z)|\geq |cz|-2$.  Let $k>2$ where $k \neq |c|$.  Then solving $|cz|-2 >k|z|$ for $|z|$, we get $|z|>\frac{2}{|c|-k}$.  Thus, for all $z$ such that $|z|>\frac{2}{|c|-k}$, we have $|f(z)|>k|z|>2|z|$.  It follows easily by induction that $|f^n(z)|>2^n|z|$ which goes to infinity exponentially as $n$ increases.  It then follows easily that infinity is an attracting fixed point.  

Now suppose that $K$ contains an attracting periodic point.  Then the basin of attraction of this point contains an open subset $U \subset K$ and we have
\begin{equation*}
\underset{m\to \infty}{\lim}\Diam(f^m(U))=0.\\
\end{equation*}
  By Proposition~\ref{noConvergence}, we are done.
\end{proof}

A likely consequence of Propositions~\ref{noConvergence} and~\ref{One attracting periodic point} is that the basin of attraction of infinity is the only Fatou domain.  However, there is still the possibility of the existence of wandering domains within $K$.  We make the following conjectures.

\begin{conj}\label{No wandering domains}
Wandering domains do not exist.
\end{conj}

\begin{conj}
If $|c|>1$ then $K=J$.
\end{conj}

If these conjectures are true, then the study of $K,$ $J,$ and $F$ are in many ways equivalent.  Since $K$ is the easiest to define, it is the one we choose to study.

We now wish to describe the simplest $K$ when $|c|>1$.  We begin by characterizing $K(c)$ for $c \in \mathbb{R}$; these $K$ are contained in the imaginary axis, and are line segments or Cantor sets.  We then describe $K(c)$ where $\Im(c)$ is small; these $K$ are double spirals made from line segments.  We conclude this chapter with a discussion on how continuously changing $c$ can lead to the discontinuous creation or deletion of isolated points or components of $K$.  

For this chapter, let $\mathcal{G}(x)$ be the real tent map defined by
\begin{equation}
\mathcal{G}(x)=
\begin{cases}
cx & \text{if } x < \frac{1}{2}, \\ 
c-cx & \text{if } x \geq \frac{1}{2}. 
\end{cases}
\end{equation}

\begin{lem}
If $c \in \mathbb{R}$ then $K(f)$ is a subset of the imaginary axis.
\end{lem}
\begin{proof}
Since $c \in \mathbb{R}$ then $|\Re(f(z))|= |c||\Re(z)|$.  Thus, if $\Re(z) \neq 0$, then the sequence $(|\Re(z)|,|f(\Re(z))|,|f^2(\Re(z))|,...)=(|\Re(z)|,|c||\Re(z)|,|c|^2|\Re(z)|,...)$ diverges and so $z \notin K$.
\end{proof}

\begin{lem}\label{conjugate to real tent map}
If $c \in \mathbb{R}$ then $f$ restricted to the imaginary axis is conjugate to the real tent map $\mathcal{G}$.
\end{lem}
\begin{proof}
Let $z=xi$ where $x \in \mathbb{R}$.  Then we claim that $\varphi(f(z))=\mathcal{G}(\varphi(z))$ where $\varphi(z)=\frac{ciz}{2}$.  We first note that the intersection of $ \PFL$ with the imaginary axis is the point $\frac{-i}{c}$, and that $\varphi(\frac{-i}{c})=\frac{1}{2}$.  It then follows easily that $z$ is a point in the intersection of $P\mathbb{H}^+$ and the imaginary axis if and only if $\varphi(z)\leq \frac{1}{2}$.  Thus, we have two cases.

Case 1:  If $z \in P\mathbb{H}^+$, then $\varphi(f(z))=\varphi(f(xi))=\varphi(cxi)=\frac{ci(cxi)}{2}=c\frac{ci(xi)}{2}=\mathcal{G}(\varphi(z))$.

Case 2:  If $z \in P\mathbb{H}^-$, then $\varphi(f(z))=\frac{ci(-cxi-2i)}{2}=c-c \frac{-cx}{2}=\mathcal{G}(\frac{-cx}{2})=\mathcal{G}(\frac{ciz}{2})=\mathcal{G}(\varphi(z))$.
\end{proof}

\begin{lem}\label{K(g) is a Cantor set}
Let $\mathcal{G}(x)$ be the real tent map.  Then if $|c|>2$ then the set of points on the real line whose trajectories do not diverge to infinity is a Cantor set.
\end{lem}
\begin{proof}
This is a well known result.
\end{proof}

\begin{lem}\label{real Cantor sets}
If $c \in \mathbb{R}$ and $|c|>2$, then $K(f)$ is a Cantor set.
\end{lem}
\begin{proof}
This follows immediately from Lemmas~\ref{conjugate to real tent map} and~\ref{K(g) is a Cantor set}.
\end{proof}

\begin{lem}\label{1<|c| leq 2}
If $1< c \leq 2$ then $K=[-2i/c,0]$.  If $-2 \leq c < -1$ then $K=\left[\frac{-2i}{c(1+c)},\frac{-2i}{1+c}\right]$.
\end{lem}
\begin{proof}
This follows from Lemma~\ref{conjugate to real tent map} and well known results in the study of the real tent map.
\end{proof}

For ease of reference, we state as a theorem the collection of the above results to describe $K$ for all $c \in \mathbb{R}$.

\begin{thm}\label{K for real c}
Let $c \in \mathbb{R}$.  
\begin{enumerate}
	\item If $|c|<1$ then $K=\mathbb{C}$.
	\item If $c=1$ then $f(\mathbb{C})=\mathbb{H}^+$ and $f$ restricted to $\mathbb{H}^+$ is the identity map.
	\item If $c=-1$ then every point is eventually mapped into the strip $\{z:-1 \leq \Im(z) \leq 1 \}$ and becomes periodic with period 2 thereafter.  (0 remains a fixed point.)
	\item If $1< c \leq 2$ then $K=[-2i/c,0]$.
	\item If $-2 \leq c < -1$ then $K=\left[\frac{-2i}{c(1+c)},\frac{-2i}{1+c}\right]$.
	\item If $2<|c|$ then $K$ is a Cantor set. 
\end{enumerate}
    
\end{thm}
\begin{proof}
The cases when $|c|<1, c=1,$ and $c=-1$ follow from Proposition~\ref{mod c less than 1} and Theorem~\ref{mod c equals 1}.  The rest of the cases follow from Lemmas~\ref{1<|c| leq 2} and~\ref{real Cantor sets}.
\end{proof}

Things start to get more exciting when $\theta \neq 0,\pi$.  When $\Im(c)$ is small and nonzero, then a spiral made of line segments emanates from the origin.  The $ \PFL$ acts as an axis of symmetry and so $K$ is a double spiral.  As the parameter changes, the bottom spiral may collide with $\FL$.  When the intersection of the bottom spiral with $\mathbb{H}^+$ is disconnected, then $K$ will have isolated point and/or components which are preimages of the portion of the lower spiral that was introduced into $\mathbb{H}^+$.  An illustration of this process is given in Figure~\ref{fig:KcontinuousChange}.

\begin{figure}[htbp]
	\centering
		\includegraphics[width=1.00\textwidth]{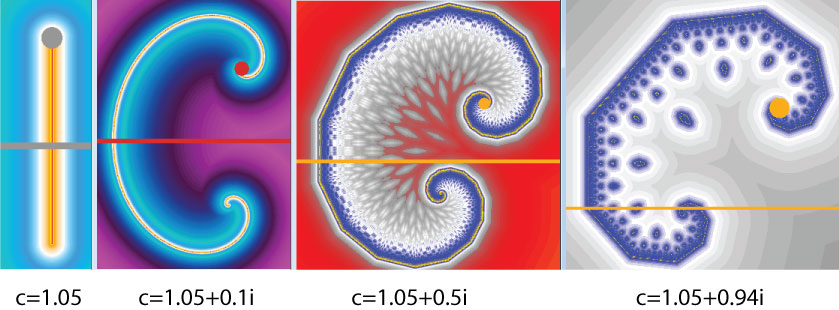}
	\caption[Creation of isolated components of $K$]{A continuous change of parameter resulting in the creation of isolated components of $K$.}
	\label{fig:KcontinuousChange}
\end{figure}

The creation of these new components of $K$ can be more fully understood when considering the inverse map.  The preimage under $f$ is found by removing the open lower half-plane, unfolding a copy of $\mathbb{H}^+$ onto $\mathbb{H}^-$ and then dividing by $c$.  Thus, if locally the intersection $K \cap \mathbb{H}^+$ consists of one point, then there will be infinitely many preimages of that point which are all isolated points of $K$. 

As an example, there exists a $K(c)$ that has both isolated points and disjoint connected components where $c=0.971+(0.851+\varepsilon)i$ and $0 < \varepsilon<0.001$.  (See Figure~\ref{fig:isolatedPoints})

\begin{figure}[htbp]
	\centering
		\includegraphics[width=0.70\textwidth]{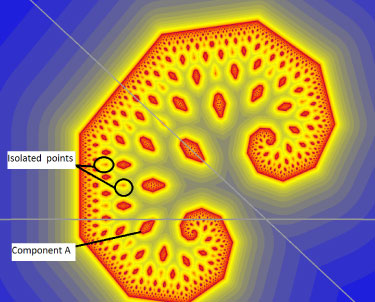}
	\caption[Isolated points of $K$]{The isolated points are preimages of component A intersected with the upper half plane.  In this case the intersection is one point.  }
	\label{fig:isolatedPoints}
\end{figure}

We formalize these ideas in the following results.

\begin{lem}\label{K is completely invariant}
$K$ is completely invariant.  Moreover, $f^{-1}(K \cap \mathbb{H}^+)=K$.
\end{lem}
\begin{proof}
$K$ is completely invariant trivially.  Thus, $f^{-1}(K)=K$.  Since the first step in finding the preimage of $K$ is to remove the open lower half plane, then $f^{-1}(K \cap \mathbb{H}^+)=f^{-1}(K)=K$.
\end{proof}

\begin{prop}\label{K has uncountably many components}
If every point $z \in K$ has the property that $\Im(z)>-1$, then $K$ has uncountably many components and is a Cantor set.
\end{prop}
\begin{proof}
We will closely follow the proof in \cite[pg 99]{Mil}.
Assume that every point $z \in K$ has the property that $\Im(z)>-1$.  Then $ \PFL$ cuts the plane into two open half planes $V_0, V_1$, where $V_0 \subset P\mathbb{H}^+$.  Now let $K_0=K \cap V_0$ and $K_1=K \cap V_1$.  Then $f(K_0)=f(K_1)=K$.  Note that $K_0$ and $K_1$ are disjoint compact sets with $K_0 \cup K_1 = K$.  Similarly, we can split each $K_n$ into two disjoint compact subsets $K_{n0}=K_n \cap f^{-1}(K_0)$ and $K_{n1} = K_n \cap f^{-1}(K_1)$, with $f(K_{n\ell})=K_\ell$.  Continuing inductively, we split $K$ into $2^{p+1}$ disjoint compact sets
\begin{equation}
K_{n_0\cdots n_p} = K_{n_0} \cap f^{-1}(K_{n_1}) \cap \cdots \cap f^{-p}(K_{n_p}),\\
\end{equation}
with $f(K_{n_0\cdots n_p})=K_{n_1\cdots n_p}$.  Similarly, for any infinite sequence $n_0 n_1 n_2 \cdots$ of zeros and ones, let $K_{n_0 n_1 n_2 \cdots}$ be the intersection of the nested sequence
\begin{equation}
K_{n_0} \supset K_{n_0 n_1} \supset K_{n_0 n_1 n_2} \supset \cdots\\
\end{equation}

Each such intersection is compact and nonvacuous.  In this way, we obtain uncountably many disjoint nonvacuous subset with union $K$.  Every connected component of $K$ must be contained in exactly one of these, so $K$ has uncountably many components.  Each of the sets $K_{n_0 n_1 \cdots n_p}$ is disjoint from $\FL$ and since $|c|>1$ then by construction $\Diam(K_{n_0 n_1 \cdots n_p})=|c|\Diam(K_{n_0 n_1 \cdots n_p n_{p+1}}$ and thus the diameter of these sets go to zero.  Thus, $K$ is a Cantor set.
\end{proof}

It is worth noting that when $K$ is a Cantor set, then the dynamics on $K$ are conjugate to the one-sided 2-shift.

\begin{figure}
	\centering
		\includegraphics[width=0.90\textwidth]{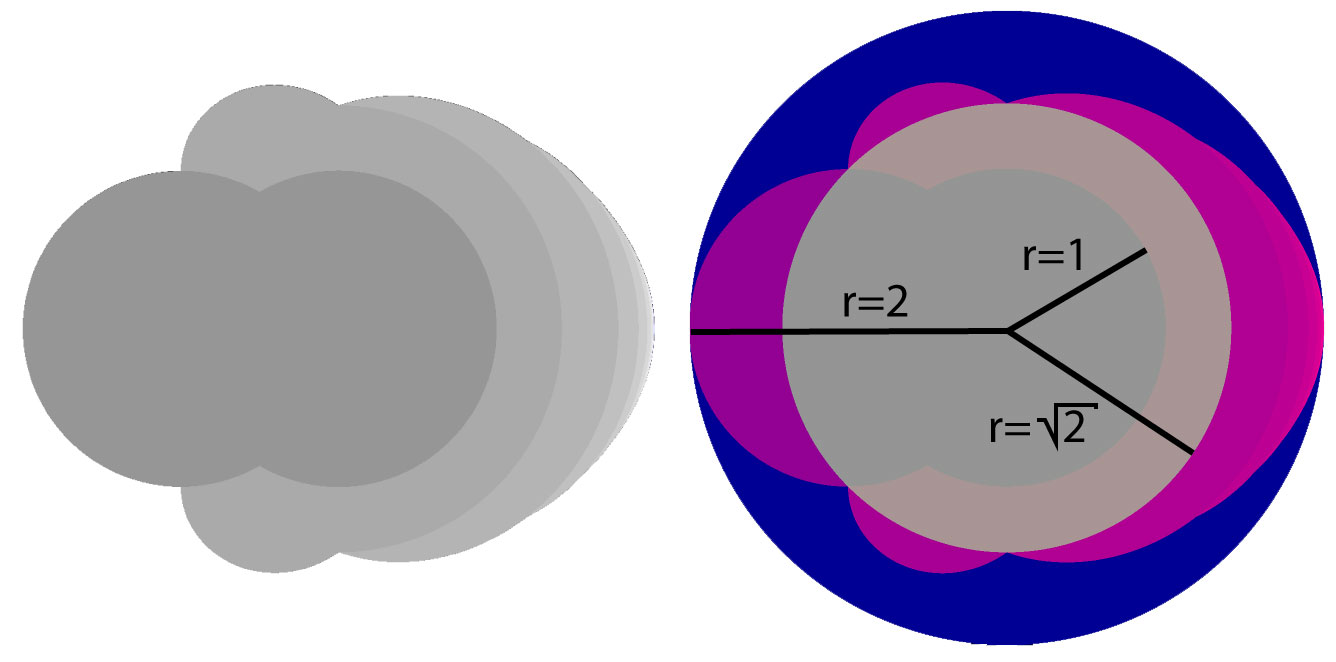}
	\caption[The set of parameters where $L \cap \FL \neq \emptyset$]{On the left is the set of parameters $c$ in the parameter plane such that $K(c)$ has at least one point in $\mathbb{H}^-$.  On the right is the same set along with 3 disks centered at the origin with labeled radii for reference.}
	\label{fig:ComplementOfCantorLocus}
\end{figure}

Figure~\ref{fig:ComplementOfCantorLocus} shows the region in the parameter plane where $K(c)$ has at least one point in $\mathbb{H}^-$.  (This is equivalent to the existence of $\ell_j \in \mathbb{H}^-$, $j\in \mathbb{Z}$.)  Outside of this region, $K(c)$ is contained in the open upper half plane and by Proposition~\ref{K has uncountably many components}, $K(c)$ is a Cantor set.  We conjecture that $K(c)$ is a Cantor set if and only if $\ell_j \notin \mathbb{H}^-$ for all $j\in \mathbb{Z}$.  (We define $\ell_j$ on page \pageref{symbol:ell}.)

In Chapter~\ref{The Perimeter Set $P$} we prove that $f$ has two fixed points $\ell_0,0$ where $f(0)=g(0)=0,$ $f(\ell_0)=h(\ell_0)=\ell_0$.  Since the image of every point is in $\mathbb{H}^+$ then these fixed points are also in $\mathbb{H}^+$.

\begin{lem}\label{K connected implies K touches FL}
If $K \cap \mathbb{H}^+$ is connected then $K \cap \mathbb{H}^+ \cap \FL \neq \emptyset$.
\end{lem}
\begin{proof}
If $K \cap \mathbb{H}^+$ is connected then both fixed points $\ell_0, 0$ are in the same component $K_0$ of $K \cap \mathbb{H}^+$.  Let $A=g^{-1}(K_0)=\{z/c:z \in K_0\}$ and $B=h^{-1}(K_0).$  Then by Lemma~\ref{K is completely invariant} $f^{-1}(K_0)=A \cup B=K_0$.  Thus $A, B$ are in the same component of $K \cap \mathbb{H}^+$ which implies $A \cap B \cap  \PFL \neq \emptyset$.  Thus $K_0 \cap \FL \neq \emptyset$.
\end{proof}

\begin{prop}
$K$ is connected if and only if $K \cap \mathbb{H}^+$ is connected.
\end{prop}
\begin{proof}
If $K \cap \mathbb{H}^+$ is connected, then by Lemma~\ref{K connected implies K touches FL} $f^{-1}(K \cap \mathbb{H}^+)$ is connected.  But by Lemma~\ref{K is completely invariant} $f^{-1}(K \cap \mathbb{H}^+)=K$.  Thus $K$ is connected.

Now let $K$ be connected and (by way of contradiction) assume that $K \cap \mathbb{H}^+$ is disconnected.  But then $f^{-1}(K \cap \mathbb{H}^+)$ is disconnected.  By Lemma~\ref{K is completely invariant} $f^{-1}(K \cap \mathbb{H}^+)=K$, a contradiction.  Thus, $K \cap \mathbb{H}^+$ is connected.
\end{proof}

\begin{Def}
A component will be call a \textbf{trivial component} if it consists of a single point.
\end{Def}

\begin{prop}
If $K$ has no trivial components and if $K \cap \FL$ is connected, then $K$ is connected.
\end{prop}
\begin{proof}
Assume $K$ has no trivial components.  Assume $K \cap \FL$ is connected.  Let $K_1$ be the unique component containing $K \cap \FL$.  Take any other component $K_2 \subset K$.  If for every $n>0$, $f^n(K_2) \not \subset K_1$, then $\Diam(f^n(K_2))$ goes to infinity as $n$ increases, since $f^n(K_2) \cap \FL = \emptyset$.  This is a contradiction.  Thus, for all $K_2$ there exists an $n$ such that $f^n(K_2) \subset K_1$.  Let $K_3$ be the component containing $0$.  Since $f(0)=0$, then for all $m>0$, $f^m(K_3) \subset K_3$.  So there exists an $n>0$ such that $f^n(K_3) \subset(K_3 \cap K_1)$.  Since $f^n(K_3) \neq \emptyset$ and $f^n(K_3) \subset(K_3 \cap K_1)$, then $K_3 \cap K_1 \neq \emptyset$.  Since $K_1$ and $K_3$ are components, then $K_3=K_1$.  Recalling that $f(K_3)\subset K_3$ we have $f(K_1) \subset K_1$.  Therefore $K_1 \subset f^{-1}(K_1)$ and since $K_1 \cap \FL$ is connected then $f^{-1}(K_1)=K_1$.  Thus $K=K_1$.
\end{proof}

Figure~\ref{fig:KconnectedKcapFLdisconnected} shows an example when $K$ can be connected even if $K \cap \FL$ is disconnected.  For this example $c \approx 1.191487884+1.191487884i$.  The exact value of $c$ has $\alpha=\beta=(1/6)(54+6\sqrt{33})^{1/3}+2/(54+6\sqrt{33})^{1/3}$.  For this choice of $c$, the bottom of $K$ is both a subset of $\FL$ and a Cantor set.  In particular $\ell_{-1}, \ell_{-2} \in \FL$ (see Chapter~\ref{The Perimeter Set $P$}).  The set $K$ can be seen to be connected by the following argument.  Cover $K$ with closed balls of radius $\varepsilon$ centered at every point of $K$ where $\varepsilon>0$ is chosen such that the union of these balls, $B$, is connected.  Since every point in $K$ is in $\mathbb{H}^+$, then removing the open lower half plane does not disconnect $B$ (this is the crux of the argument).  Since $\ell_{-2} \in \FL \cap B$ then $B \cup \Reflect(B,\FL)$ is a connected set.  Dividing by c reduces the diameters of the balls and rotates, but has no effect on connectivity.  Obviously, $f^{-1}(B)$ is a compact connected set and $K \subset f^{-1}(B) \subset B$.  It is now easily seen that $K$ is the intersection of a nested sequence of nonempty, compact, connected sets of the form $f^{-n}(B)$.

\begin{figure}[htbp]
	\centering
		\includegraphics[width=0.80\textwidth]{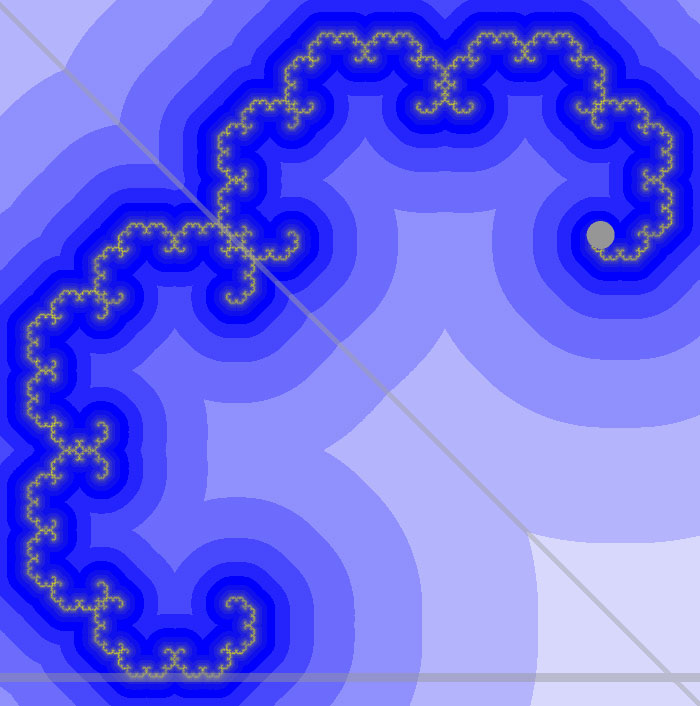}
	\caption{An example where $K$ is connected but $\FL \cap K$ is disconnected.}
	\label{fig:KconnectedKcapFLdisconnected}
\end{figure}

\begin{prop}
If $|c|>2$ then $K$ is totally disconnected.
\end{prop}
\begin{proof}
Let $|c|>2$ and assume that $K$ has nontrivial connected components.  Let $K_1$ be one of these components with maximal diameter.  Then $\Diam(K_1)=d>0$ and $\Diam(f(K_1))\geq \frac{d|c|}{2}>d.$  Since the continuous image of a connected set is connected, then $f(K_1)\subset K$ is connected with diameter strictly larger than $d$, a contradiction.  Thus, $K$ has only trivial components.
\end{proof}

\begin{conj}
If $|c|>2$ then $K$ is a Cantor set.
\end{conj}

\begin{lem}
If $|c|>\sqrt{2}$ then $K$ has Lebesgue measure $0$.
\end{lem}
\begin{proof}
Let $\mu(S)$ denote the Lebesgue measure of a set $S$ and for all $w \in \mathbb{C}$, let $wS=\{ws:s \in S\}$.  By definition $K$ is forward invariant.  Noting that multiplication by $c$ scales both dimensions by a factor of $|c|$, we see that if $|c|>\sqrt{2}$ then $\mu(cK)=|c|^2 \mu(K) \geq 2 \mu(K)$ with equality exactly when $\mu(K)=0$.  Since folding can reduce the measure of a set by at most half, then $\mu(f(K)) \geq \mu(K)$ with equality exactly when $\mu(K)=0$.  Thus, if $|c|>\sqrt{2}$ then $K$ is forward invariant only if $\mu(K)=0$.
\end{proof}

\begin{prop}
If $|c|>3$ then $K \subset B(0,1)$.
\end{prop}
\begin{proof}
If $|z|=1$ then $|f(z)|\geq |c||z|-2>1=|z|$.  Now if $|z|=1 +\varepsilon$ for some $\varepsilon>0$, then $|f(z)|\geq |c|(1 +\varepsilon)-2>1+|c|\varepsilon=|z|.$  Now suppose that $|f^n(z)|>1+|c|^n\varepsilon$ for some $n>0$.  Then $|f^{n+1}(z)|>|c|(1+|c|^n\varepsilon)-2>1+|c|^{n+1}$.  Since $|c|>1$, then by induction we have that every point outside of the open unit disk has a trajectory that diverges.
\end{proof}
\end{section}

\begin{section}{The Perimeter Set $P$} \label{The Perimeter Set $P$}

In this chapter we introduce the \textbf{perimeter set, $P_c$}, which will be of primary importance throughout this paper.  We will write $P$ in stead of $P_c$ whenever possible.  Note that $\Reflect(z, \PFL)=\frac{1}{c}\left(\overline{cz}-2i\right)$.  We begin with some definitions.

\begin{Def}\label{symbol:ell}
(See Figure~\ref{fig:Introducing_ell})  We define, $\ell_0=\frac{2i(1-\overline{c})}{|c|^2-1}.$  We also define $\ell_k$, where $k \neq 0$, as
\begin{equation}
\ell_k=
\begin{cases}
\frac{\ell_0}{c^k} & \text{if } k>0,\\
\Reflect(\ell_{-k+1}, \PFL)=\frac{1}{c} \left( \overline{\frac{\ell_0}{c^{-k}}+2i }\right) & \text{if } k <0.\\ 
\end{cases}
\end{equation} 
\end{Def}

\begin{figure}[htbp]
	\centering
		\includegraphics[width=0.80\textwidth]{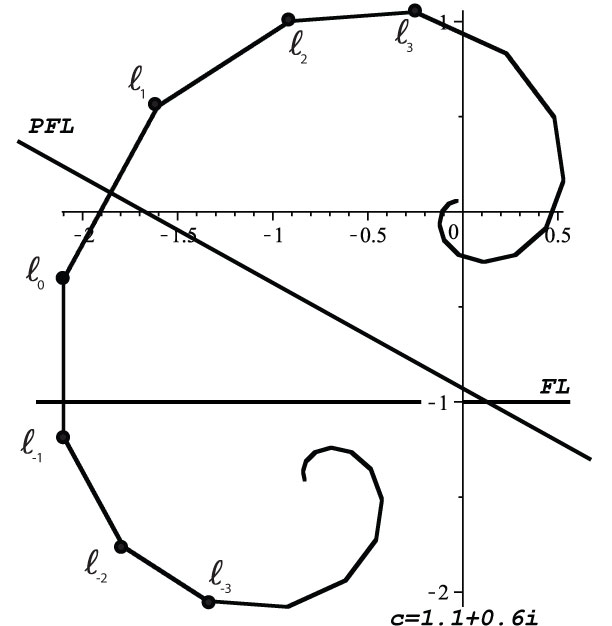}
	\caption{An example of various $\ell_k$.}
	\label{fig:Introducing_ell}
\end{figure}

We now show that the point $\ell_0$ is a fixed point of $f$.  First we need a lemma.

\begin{lem}\label{Im(cz0)}
We have $\Im(c\ell_0) < -1.$
\end{lem}
\begin{proof}
\begin{equation}
\Im(c\ell_0) = \Im\left(\frac{2i(c-|c|^2)}{|c|^2-1}\right)=\frac{2\alpha-2|c|^2}{|c|^2-1}\\
\label{eq:1a}
\end{equation}

Since $|c|>1$ then $|c|^2-1>0$ and so the inequality $\Im(c\ell_0)<-1$ is true precisely when $2\alpha-2|c|^2<-|c|^2+1$, or equivalently when 

\begin{equation}
|c|^2-2\alpha+1>0
\label{eq:1b}
\end{equation}
where the inequality is strict.  Now since 
\begin{equation}
|c|^2 \geq \alpha^2
\label{eq:2a}
\end{equation}
(with equality only when $c$ is real) then 

\begin{equation}
|c|^2-2 \alpha+1 \geq \alpha^2-2 \alpha+1=(\alpha-1)^2 \geq 0
\label{eq:3a}
\end{equation}  
with equality only when $\alpha =1$.  Now if $c$ is real and $\alpha =1$ then $c=1$.  Since we assume that $|c|>1$ then at least one of the inequalities in \eqref{eq:2a} or \eqref{eq:3a} must be strict and \eqref{eq:1b} is satisfied.  Thus $\Im(c\ell_0)<-1.$
\end{proof}

\begin{lem}\label{Fixed point}
The map $f_c$ has exactly two fixed points $0$ and $\ell_0.$
\end{lem}
\begin{proof}
Suppose that $\Im(cz) \geq -1$.  Then since $|c|>1$ we see that $f(z)=cz=z$ has a unique solution $z=0.$  Now suppose that $\Im(cz)<-1$.  Then we want to solve $f(z)=\overline{cz}-2i=z$ for $z$.  We will show that $z=\ell_0$ is the desired solution.  This solution is unique since $h(z)=\overline{cz}-2i$ multiplies all distances by $|c|$.  By Lemma~\ref{Im(cz0)} we have: 
\begin{equation}
\begin{aligned}
f(\ell_0)&=\overline{c\ell_0}-2i \\
&=\frac{-2i \overline{c}(1-c)}{|c|^2 - 1}-2i \\
&=\frac{-2i \overline{c}+2i|c|^2+(-2i|c|^2+2i)}{|c|^2 - 1} \\
&=\frac{2i(- \overline{c}+1)}{|c|^2-1} \\
&=\ell_0.\\ 
\end{aligned} 
\end{equation}
\end{proof}

Lemma~\ref{Symmetry of ell_k} will show that $\ell_{k+1}$ and $\ell_{-k}$ are symmetric with respect to the $ \PFL$ for all $k \in \mathbb{Z}$. (See Figure~\ref{fig:Introducing_ell})

\begin{lem}\label{Symmetry of ell_k}
For all $k \in \mathbb{Z}$ we have $\Reflect(\ell_k, \PFL)=\ell_{-k+1}$.
\end{lem}
\begin{proof}
We have three cases.  

Case 1:  If $k < 0$ then $\ell_k = \Reflect(\ell_{-k+1}, \PFL)$ by definition and taking $\Reflect$ of both sides we get $\Reflect(\ell_k, \PFL)=\ell_{-k+1}$.

Case 2:  If $k=0$ then
\begin{equation}
\Reflect(\ell_0, \PFL)=\frac{1}{c}\left(\overline{c\ell_0}-2i\right)=\frac{1}{c}(\ell_0)=\ell_{1}.
\end{equation}

Case 3:  If $k>0$ then \begin{equation}
\Reflect(\ell_k, \PFL)=\frac{1}{c}\left(\overline{c\ell_k}-2i\right)=\frac{1}{c}\left(\frac{\overline{\ell_0}}{\overline{c^{-(-k+1)}}}-2i \right)=\ell_{-k+1}.
\end{equation}
\end{proof}

In quadratic complex dynamics, every connected filled-in Julia set has diameter less than or equal to 4.  This is useful since it allows one to choose a fixed bailout value for all computer pictures of connected filled-in Julia sets.  We now show that for TTM's, the diameter of $K$ can be arbitrarily large when $|c|$ is close to 1.  Thus, the choice of our bailout value will need to depend on $c$.

\begin{lem}\label{unbounded Escape radius}
Let $c=a\lambda$, $a>1,$ $|\lambda|=1$.  Then, for a fixed $\lambda \neq 1$ we have $\underset{a \searrow 1}{\lim}\Diam(K(c))=\infty$.
\end{lem}
\begin{proof}
We have $\Diam(K(c)) \geq |\ell_0|=\frac{2|a \lambda-1|}{a^2-1}>\frac{2|\lambda-1|}{a^2-1}.$
\end{proof}

%


\begin{Def}\label{$L_j$}\label{symbol:L}
(See Figure~\ref{fig:Introducing_L})  Let $L_j=[\ell_j, \ell_{j+1}]$ for $j \in \mathbb{Z}$ and let $L= \underset{j=1}{\overset{\infty}{\bigcup}}L_j$.  (See Figure~\ref{fig:Introducing_L})
\end{Def}

\begin{figure}[htbp]
	\centering
		\includegraphics[width=0.50\textwidth]{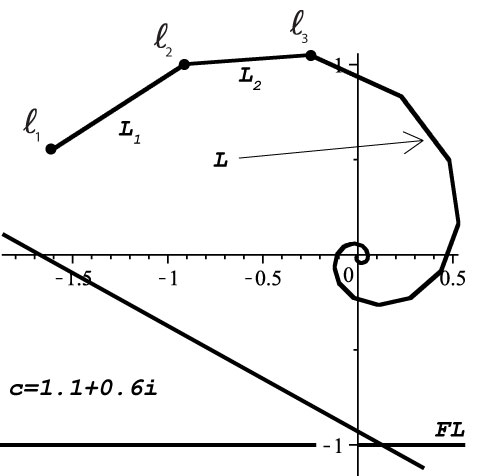}
	\caption{An example of $L$.}
	\label{fig:Introducing_L}
\end{figure}

\begin{Def}\label{symbol:Conv}
We will denote the convex hull of a set $X$ by $\Conv(X)$.
\end{Def}

\begin{lem}\label{PFL is perp to L0}
The $ \PFL$ is the perpendicular bisector of $L_0$.
\end{lem}
\begin{proof}
By Lemma~\ref{Symmetry of ell_k}, the endpoints of $L_0$ are reflections of each other across $ \PFL$.
\end{proof}

\begin{cor}\label{L sub minus 1 is vertical}
$L_{-1}$ is a vertical line segment.
\end{cor}
\begin{proof}
Lemma~\ref{PFL is perp to L0} gives us that the $ \PFL$ is perpendicular to $L_0$.  Also, $f$ restricted to $P\mathbb{H}^+$ maps $ \PFL$ onto $\FL$, maps $L_0$ to $L_{-1}$, and preserves angles.
\end{proof}

The next proposition states that if $L$ self-intersects, then $L$ intersects $\FL$ before it intersects itself.
\begin{prop}\label{LcapFL}
If $L_n\cap L_m \neq \emptyset$ for some $0<m<n$, then there exists $k$ such that $0<k<m$ and $L_k \cap \FL\neq \emptyset$.
\end{prop}
\begin{proof}
Suppose $L_n\cap L_m \neq \emptyset$ for integers $0<m<n$.  Since $|c|>1$ then $L_n\cap L_m$ consists of a single point which we will call $x_m$, since $x_m \in L_m$.  It follows that $f^m(L_n)\cap f^m(L_m)=L_{n-m}\cap L_0=f^m(x_m)$ and we call this point $x_0$.  Then $x_0$ is seen to be the first intersection of $L\cup L_0$ with itself.  

Assume that $\alpha \geq -1$.  Later in Proposition~\ref{Real part of gamma_0} we will show that $\Re(\gamma_0)<\Re(\ell_0)$ if and only if $\alpha<-1$.  Thus, if $\alpha\geq -1$, then $ \PFL$ either does not intersect $cL_0$ or $ \PFL$ intersects $cL_0$ on/below $\FL$.  That is, if $ \PFL \cap cL_0\neq \emptyset$ then $\Im(\PFL \cap cL_0)\leq -1.$  Thus, if $\alpha \geq -1$ then $f(L_0) \subset P\mathbb{H}^-$.  Since $x_0 \in L_0$ then $f(x_0)\in f(L_0)$ and is the intersection between $cL_0$ and $L_{n-m-1}$.  Since $f(x_0)\in P\mathbb{H}^-$ then $L_{n-m-1}$ intersects $\PFL$ and we conclude that $L_{n-m-2} \cap \FL \neq \emptyset$ and we are done.

Now assume that $\alpha < -1$.  Let $\omega=\pi-\theta$.  Then $\omega$ is the smallest angle between any two consecutive $L_k$ and since $\alpha < -1$ then $\pi/2<\theta<\pi$.  Since $4 \theta>2\pi$ and from the self-similarity of the $L_k$, it is easily seen that if $L_0 \cap L_n = \emptyset, n=2,3$ then $L_0 \cap L_n = \emptyset, n>3$.  Thus, we have only two more cases.

Case 1:  (See Figure~\ref{fig:L0capL2})  Assume $L_0 \cap L_2 \neq \emptyset$.  Then $L_0 \cup L_1 \cup L_2$ bound an isosceles triangle $T$ where at least two of the interior angles are $\omega.$  By Lemma~\ref{PFL is perp to L0}, $\PFL$ is the perpendicular bisector of $L_0$, which implies  that $\PFL/c$ is the perpendicular bisector of $L_1$.  Thus, $\PFL/c$ bisects $T$ and so $\{x_0\}=(L_0 \cap L_2) \in \PFL/c$.  But this implies that $f^2(x_0)\in \FL \cap L_0$ and we are done.

\begin{figure}[htbp]
	\centering
		\includegraphics[width=0.70\textwidth]{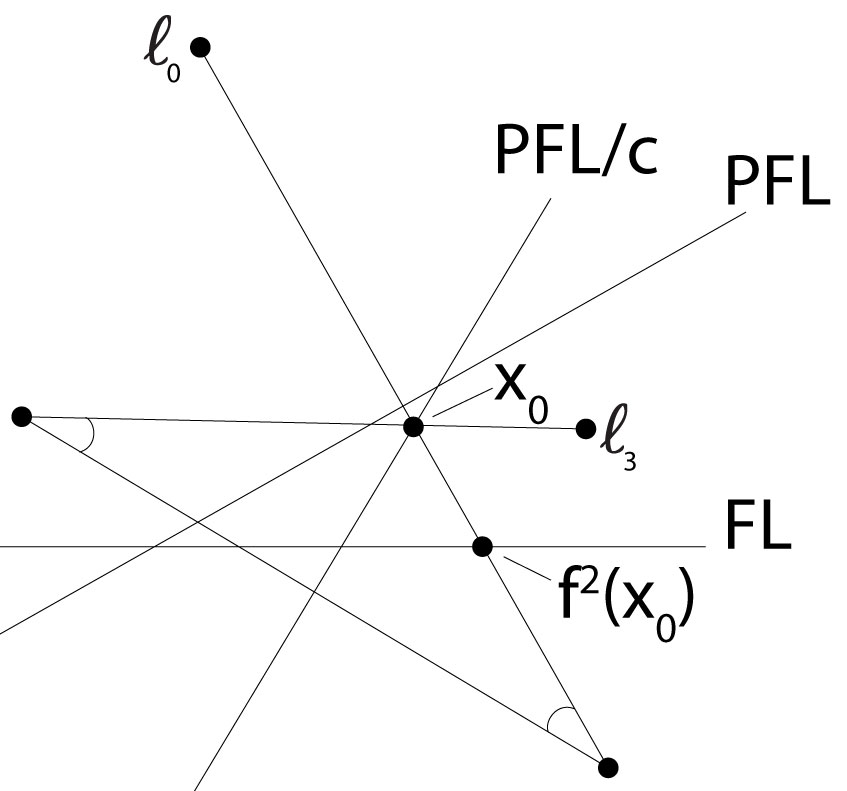}
	\caption{Proposition~\ref{LcapFL} when $\alpha <-1$ and $L_0 \cap L_2 \neq \emptyset$.}
	\label{fig:L0capL2}
\end{figure}

Case 2:  Lastly, assume $L_0 \cap L_3 \neq \emptyset$ but $L_0 \cap L_2 = \emptyset$.  A necessary condition for this is $\frac{\pi}{2}<\theta<\frac{2\pi}{3}.$  The parameter with the smallest modulus satisfying $c=re^{i\theta}$ and $\Re(c)\leq-1$ is $-1+i\sqrt{3}$ which has modulus $2$.  Thus, we may assume that $|c|\geq 2$.  To increase the chances of $L_0$ intersecting $L_3$ we clearly want to minimize both $|c|$ and the angle between $L_0$ and $L_1$.  Thus, we assume that $|c|=2$ and $\omega=\frac{\pi}{3}$.  Denote the length of $L_0$ by $d$.  Now, plotting $L_n$, $n=0,1,2,3$ under these conditions, treating vertices $\ell_2, \ell_3$ as flexible joints by straightening out $L_2,L_3$ towards $L_0$, we obtain the simplified diagram shown in Figure~\ref{fig:NoSelfIntersections}.  Since $1.5d<\sqrt{3}d$, then $L_3 \cap L_0=\emptyset$.  Since everything that was done increased the chance of an intersection between $L_0$ and $L_3$, then if $\alpha \leq -1$ and $(L_1 \cup L_2 \cup L_3)\cap \FL = \emptyset$, then $L_0 \cap L_3 = \emptyset$.
\end{proof}

\begin{figure}[htbp]
	\centering
		\includegraphics[width=0.70\textwidth]{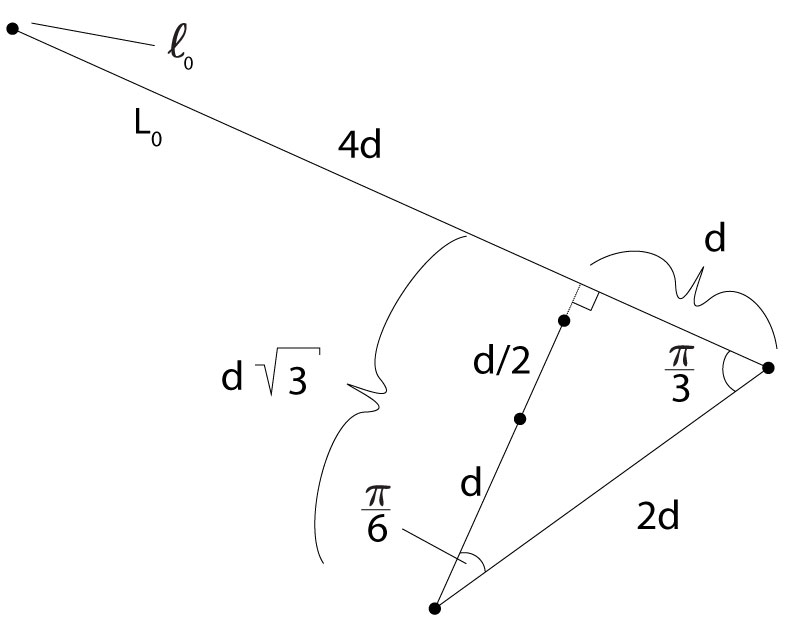}
	\caption[A worst case]{This figure shows $L_n$, $n=0,1,2,3$ where $\theta =\frac{\pi}{3}$ and after $L_2,L_3$ have been moved as close to $L_0$ as is possible under certain worst case conditions.}
	\label{fig:NoSelfIntersections}
\end{figure}

\begin{Def}\label{symbol:z'}
\begin{itemize}
	\item Let $d(z,w)=|z-w|$ and let $d(z,S)=min\{d(z,w):w\in~S\}$ for any compact set $S$.
	\item We will denote the complement of a set $S$ by $S^c$.
	\item We will denote $\Reflect(z,\PFL)$ by $z'$ and will write $X'=\{z':z\in~X\}=\Reflect(X,\PFL)$ for any set $X\subset \mathbb{C}$.
	\item We will denote the midpoint of $L_k$ by $m_k=\frac{\ell_k+\ell_{k+1}}{2}$ for $k \geq 0$.  We will also write $m=f(m_0)$.\label{symbol:m}
	\item We will denote the boundary of a set $X$ by $\Bd(X)$.
\end{itemize}
\end{Def}

\begin{figure}[htbp]
	\centering
		\includegraphics[width=0.75\textwidth]{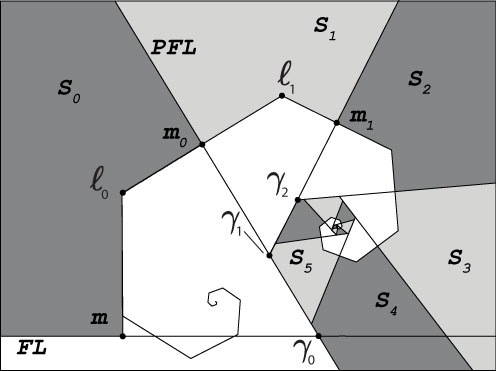}
	\caption{The construction of $P$.}
	\label{fig:Outside_of_P}
\end{figure}

\begin{Def}\label{symbol:S}
(See Figure~\ref{fig:Outside_of_P})  Let $S_1$ be the simply connected closed region satisfying the following.
\begin{enumerate}
	\item $S_1$ is one of two regions bounded on four sides by $\PFL$, $[m_0,\ell_1]$, $[\ell_1,m_1]$, and $\PFL/c$. 
	\item $S_1$ is the region containing points of arbitrarily large modulus.
	\item $S_1$ is closed.
\end{enumerate}

We also define $S_k$ recursively by $S_{k}=\frac{S_{k-1}}{c}\cap P\mathbb{H}^+,k=2,3,...$.  Let $S= \underset{k=1}{\overset{\infty}{\bigcup}}S_k$.  For $k>0$ we define $S_{-k}=S_k '$.  We also define $S_0=f(S_1)$. 
\end{Def}

\begin{Def}\label{symbol:Int}
We denote the interior of a set $U$ by $\Int(U)=U\setminus \Bd(U)$.
\end{Def}

\begin{Def}\label{symbol:P}
We define $P=(\Int(S \cup S'))^c$.
\end{Def}

Some examples of $P$'s are given in Figures \ref{fig:Examples_of_K} and \ref{fig:ExamplesOfP}.

\begin{figure}[htbp]
	\centering
		\includegraphics[width=1.00\textwidth]{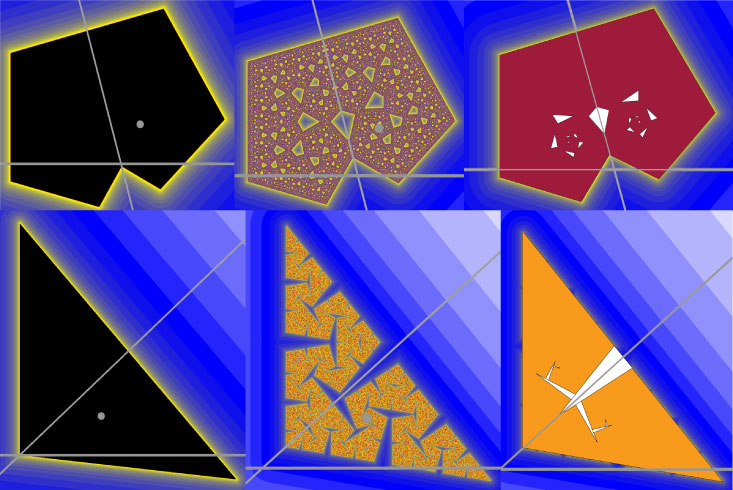}
	\caption[Examples of $P$ and $K$]{The left-most images show examples where $P=K$ and $K$ is a polygon.  Changing the parameters slightly results in filled-in Julia sets that are clearly not polygons.  The right-most pictures show the perimeter sets $P$ for the middle-most pictures. }
	\label{fig:Examples_of_K}
\end{figure}

\begin{figure}[htbp]
	\centering
		\includegraphics[width=0.90\textwidth]{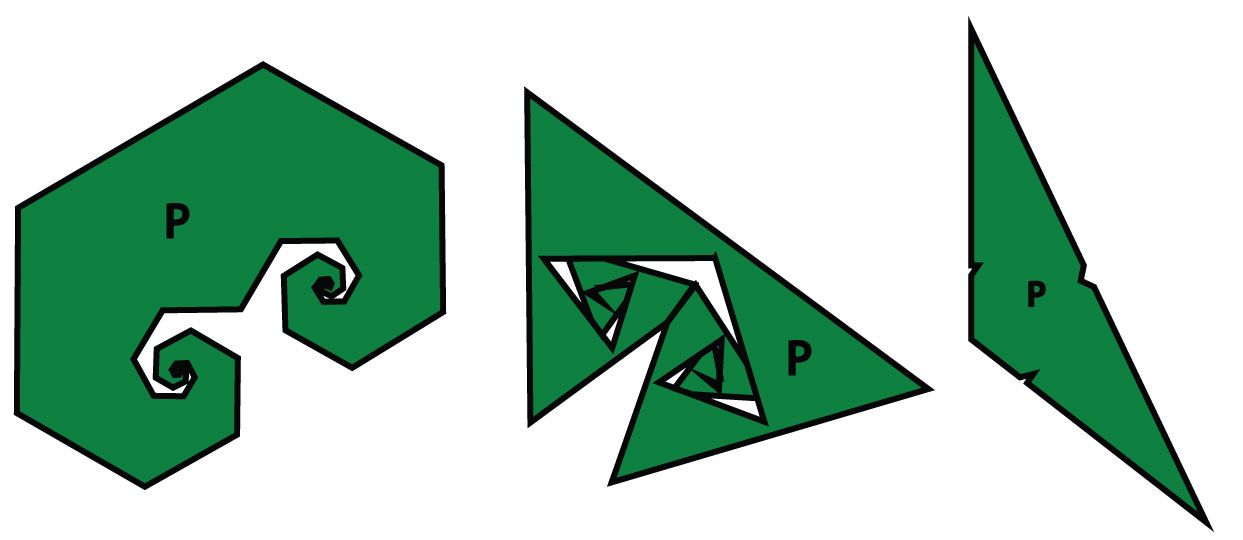}
	\caption{Some examples of $P$.}
	\label{fig:ExamplesOfP}
\end{figure}

\begin{cor}\label{P is symmetric with respect to the PFL}
$P$ is symmetric with respect to the $\PFL$.  Furthermore, $L_j$ and $L_{-j}$ are symmetric about the $\PFL$.
\end{cor}
\begin{proof}
This follows immediately from the definition of $P$ and from Lemma~\ref{Symmetry of ell_k}.
\end{proof}

\begin{lem}\label{f(S_k) subset S_{k-1}}
$f(S_k) \subset S_{k-1}$ for $k=1,2,...$.
\end{lem}
\begin{proof}
Let $z\in S_k$ for some $k\in \{1,2,...\}$.  Then by the definition of $S_k$ we have $z \in \left ( \frac{S_{k-1}}{c}\cap P\mathbb{H}^+ \right )$.  Since $z \in P\mathbb{H}^+$ then $f(z)=cz\in c \left ( \frac{S_{k-1}}{c} \right ) =S_{k-1}$.  Since $z$ was arbitrary, then $f(S_k) \subset S_{k-1}$ for $k=1,2,...$.
\end{proof}

\begin{lem}\label{S_0 subset of P^c}
$\Int(S_0) \subset P^c$.
\end{lem}
\begin{proof}
This follows immediately from symmetry and Corollary~\ref{L sub minus 1 is vertical}, which says that $f(L_0)$ is a vertical line segment.  It is easy to see that $f(L_0)$ is collinear with $L_{-1}$.
\end{proof}

\begin{prop}\label{P^c is invariant}
$P^c$ is forward invariant.
\end{prop}
\begin{proof}

Since $f$ restricted to $P\mathbb{H}^+$ is multiplication by $c$, then for any set $U \subset P\mathbb{H}^+$ we have $f(\Int(U))=c\Int(U)=\Int(cU)$.  Since $P^c$ is symmetric with respect to $\PFL$, then $f(P^c)=f(P^c \cap P\mathbb{H}^+)=cInt\left(\underset{k=1}{\overset{\infty}{\bigcup}}S_k\right)=Int\left(\underset{k=1}{\overset{\infty}{\bigcup}}cS_k\right)\subset Int\left(\underset{k=1}{\overset{\infty}{\bigcup}}S_{k-1}\right)$.  Lastly, Lemma~\ref{S_0 subset of P^c} gives us that $Int\left(\underset{k=1}{\overset{\infty}{\bigcup}}S_{k-1}\right) \subset P^c.$

%
%
%
\end{proof}

\begin{cor}\label{P subset of f(P)}
$P\cap \mathbb{H}^+ \subset f(P)$.
\end{cor}
\begin{proof}
Let $f(z) \in P \cap \mathbb{H}^+$.  Then $z \in P$ since otherwise $z \in P^c$ and then by Proposition~\ref{P^c is invariant} we would have $f(z) \in P^c$.
\end{proof}

\begin{lem}\label{P is compact}
$P$ is compact.
\end{lem}
\begin{proof}
$P^c=\Int(S \cup S')$ and is so it is open.  Thus, $P$ is closed.
To show that $P$ is bounded we begin with a disk $D$ centered at the origin and containing the line segment $L_0=[\ell_0,\ell_1]$. Now for all $k>1$ we have that $|\ell_k| \leq \frac{|\ell_0|}{|c|}$.  Thus $L_k \subset D,k \geq 0$.  Similarly, $L' \subset D'$.  Let $\gamma_1=\gamma_0/c$.  Then by construction $\Conv(P)=\Conv((L \cup L' \cup \{\gamma_1\})$.  Therefore, any disk centered at the origin containing $D \cup D' \cup \{\gamma_1\}$ will contain $\Conv(P)$ and necessarily also $P$.  Thus, $P$ is a closed and bounded subset of the complex plane, and is therefore compact.
\end{proof}

\begin{lem}\label{Expansion outside of P}
If $z \in P^c$ then $d(f(z),P)\geq|c|d(z,P)$.
\end{lem}
\begin{proof}
Let $z \in P^c$ be given.  Since $P$ is symmetric with respect to the $\PFL$ then $d(z,P)=d(z',P)$.  Also, since $f(z)=f(z')$ then  $d(f(z),P)=d(f(z'),P)$.  Thus, $d(f(z),P)\geq|c|d(z,P)$ if and only if $d(f(z'),P)\geq|c|d(z',P)$, and so we may assume that $z \in P\mathbb{H}^+$.  

Since $P$ is compact, then there exists $w \in P$ such that $d(z,w)=d(z,P)$.  Clearly, $w \in P\mathbb{H}^+$, for otherwise $d(z,w')<d(z,w)$ which would contradict $d(z,w)=d(z,P)$.  Noting that $f$ restricted to $P\mathbb{H}^+$ is multiplication by $c$, we have that $d(f(z),f(P))=d(f(z),f(w))=d(cz,cw)=|cz-cw|=|c||z-w|=|c|d(z,P)$.  Now by Corollary~\ref{P subset of f(P)} $P\cap \mathbb{H}^+ \subset f(P)$, and so for every point $q \in P\cap \mathbb{H}^+$ we have $d(q,P) \geq d(q,f(P)).$  Setting $q=f(z)$ this becomes $d(f(z),P) \geq d(f(z),f(P)).$  Thus if $z \in P^c$ then $d(f(z),P)\geq d(f(z),f(P))=|c|d(z,P)$.  
\end{proof}

\begin{thm}\label{K subset P}
$K \subset P$.  
\end{thm}
\begin{proof}
Let $z \in P^c$.  Then by Lemma~\ref{Expansion outside of P}, $d(f(z),P) \geq |c|d(z,P)$.  By Proposition~\ref{P^c is invariant} we have that $f(z)\in P^c$.  Thus, $d(f^n(z),P)\geq |c|^n d(z,P)$ ans so $d(f^n(z),P)$ grows exponentially as $n$ increases.  This shows that $z \notin K$.  Since $z\in P^c$ was arbitrary, then $K \cap P^c = \emptyset$ and thus $K \subset P$.
\end{proof}

\begin{Def}\label{symbol:zeta}
If $L \cap \PFL \neq \emptyset$, then there exists a smallest positive integer $n$ such that $L_n \cap \PFL$.  In this case, we define $\{\zeta\}=L_n \cap \PFL$.
\end{Def}

We now define a curve $\Gamma$ which is a subset of the preimages of $\PFL$.  $\Gamma$ is constructed in the same way as $L$.  
Recall that $\{\gamma_0\}=\{\frac{\alpha-1}{\beta}-i\}= \FL \cap \PFL$.

\begin{Def}\label{Gamma}
We make the following definitions.

\begin{enumerate}
	\item We define $\gamma_k= \frac{\gamma_0}{c^k}, k \in \mathbb{Z}$.\label{symbol:gamman}
	\item We define $\Gamma_k$ to be the line segment $[\gamma_k, \gamma_{k+1}]$ for all $k \in \mathbb{Z}$.
	\item We also define $\Gamma=\underset{k=1}{\overset{\infty}{\bigcup}}\Gamma_k$.\label{symbol:Gamma}
\end{enumerate}
\end{Def}

Note that $\Gamma_0 \cap \Gamma= \{\gamma_1\}$ in the same way as $L_0 \cap L= \{\ell_1\}$.

\begin{Def}
If $\zeta$ exists and $n$ is the smallest positive integer such that $\zeta \in L_n$, then we define the \textbf{outer-most boundary of $P$} to be 

\begin{equation}
[\ell_n,\zeta]\cup[\ell_{-(n-1)},\zeta]\cup \left(\underset{|k|\leq n-1}{\bigcup}L_k\right).\\
\end{equation}

If $\zeta$ does not exist then the \textbf{outer-most boundary of $P$} is $L' \cup L_0 \cup L$ and we additionally define the \textbf{inner-most boundary of $P$} to be $\Gamma \cup \Gamma'$ (see Figure \ref{fig:inner_most_boundary}).
\end{Def}

\begin{figure}
	\centering
		\includegraphics[width=1.0\textwidth]{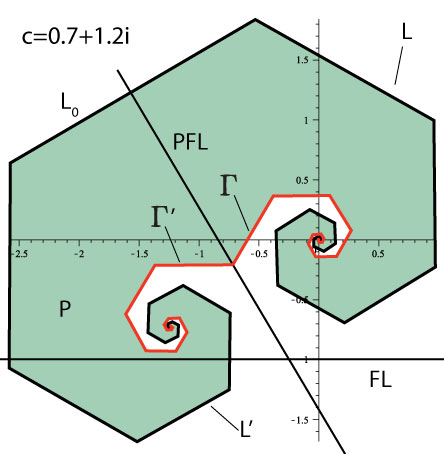}
	\caption[The outer-most and inner-most boundaries of $P$]{On the left is an example of the outer-most and inner-most boundaries of $P$.  Shown on the right is $K$.}
	\label{fig:inner_most_boundary}
\end{figure}

When $P$ is a polygon then the outer-most boundary of $P$ is just the boundary of $P$.  This is not true when $P$ is not a polygon.  Figure~\ref{fig:outermostBoundary} shows an example where $K=P$, $K$ has nonempty interior, and yet the outer-most boundary of $P$ is not the boundary of $P$.

\begin{figure}[htbp]
	\centering
		\includegraphics[width=1.00\textwidth]{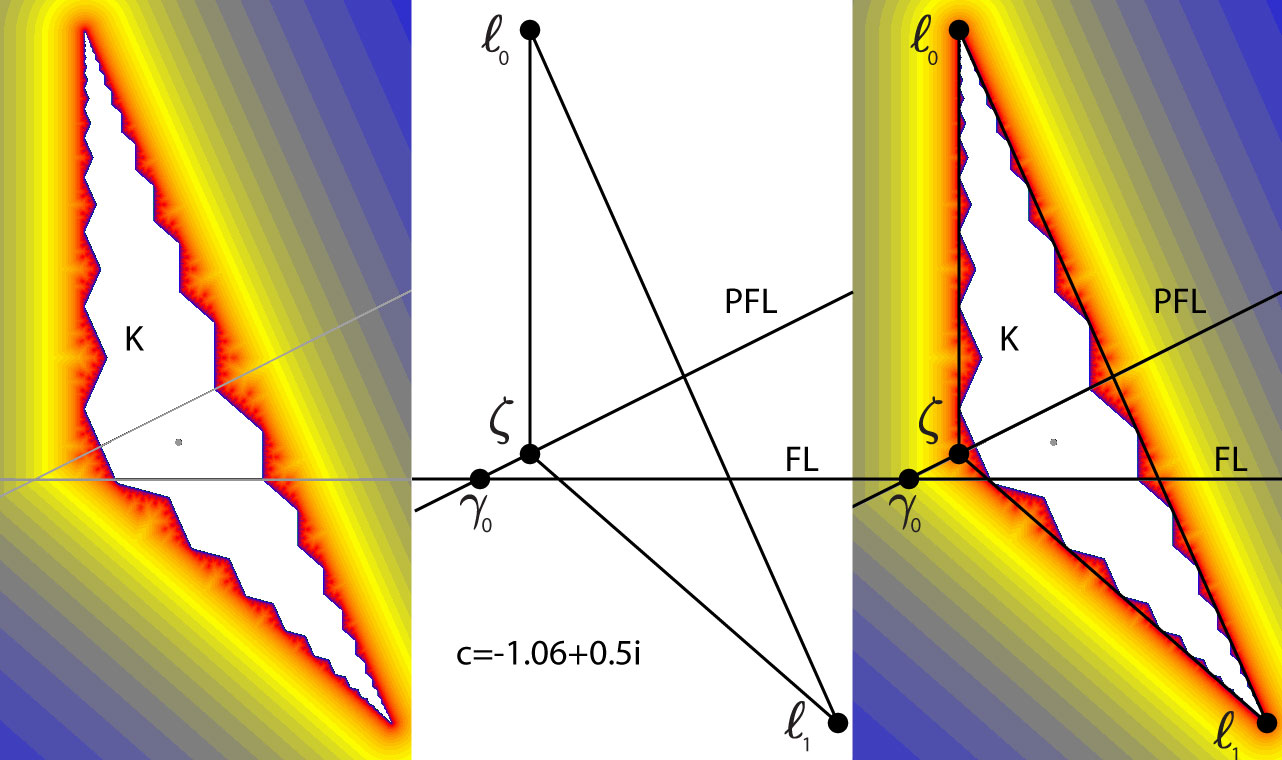}
	\caption[The outer-most boundary]{In the middle is the outer-most boundary of $P$ for $c=-1.06+0.5i$.  For this choice of $c$, the outer-most boundary is a triangle.  However, $K$ is a polygon with infinitely many sides.  }
	\label{fig:outermostBoundary}
\end{figure}

\begin{prop}\label{P is right of ell_0}
If $z \in P$ then $\Re(z)\geq \Re(\ell_0)$.  
\end{prop}
\begin{proof}
It is easily seen that the points of $P$ with the smallest real part belong to the outer-most boundary of $P$.  Due to this and the spiraling of $L'$, for there to be a point $z \in P$ with $\Re(z)<\Re(\ell_0)$, then $L'$ must have a self-intersection.  Then by Proposition \ref{LcapFL}, $\zeta$ must exist and by definition, $P$ is contained in the polygon bounded by the outer-most boundary of $P$.  It is clear that every point in this polygon has real part greater than or equal to $\Re(\ell_0)$.
\end{proof}

\begin{cor}\label{K is right of ell_0}
If $z \in P$, then $\Re(z) \geq \Re(\ell_0)$.
\end{cor}
\begin{proof}
This follows immediately from Theorem \ref{K subset P} and Proposition \ref{P is right of ell_0}.
\end{proof}

\begin{lem}\label{Re(gamma_0)<Re(ell_0)}
If $\Re(\gamma_0) < \Re(\ell_0)$ then $m \notin P$.
\end{lem}
\begin{proof}
(See Figure~\ref{fig:Outside_of_P}) If $\Re(\gamma_0) < \Re(\ell_0)$ then the $\PFL$ crosses $f(L_0)$ above the $\FL$ and so by the definition and symmetry of $P$ the point $m=\Re(\ell_0)-i \notin P$.  
\end{proof}

\begin{prop}\label{Real part of gamma_0}
$\Re(\gamma_0) < \Re(\ell_0)$ if and only if $\alpha <-1$.  Furthermore, if $\alpha <-1$ then $\gamma_0 \notin K$.
\end{prop}
\begin{proof}
$\Re(\gamma_0)= \frac{\alpha-1}{\beta}$ and $\Re(\ell_0)=\frac{-2 \beta}{\alpha^2+\beta^2-1}$.  Since $\beta > 0$ and $|c|=\alpha^2+\beta^2>1$, then both denominators are positive.  We have
\[
\Re(\gamma_0)<\Re(\ell_0)
\]
is equivalent to
\[
\frac{\alpha-1}{\beta}<\frac{-2\beta}{\alpha^2+\beta^2-1}
\]
is equilvalent to
\[
(\alpha-1)(\alpha^2+\beta^2-1)<-2\beta^2
\]
is equivalent to
\[
(\alpha-1)(\alpha^2-1)+\beta^2(\alpha+1)<0
\]
is equivalent to
\[
((\alpha-1)^2+\beta^2)(\alpha+1)<0
\]
is equivalent to $\alpha<-1$.  Thus $\alpha <-1$ if and only if $\Re(\gamma_0)<\Re(\ell_0)$.  As a result, if $\alpha <-1$ then $\Re(\gamma_0) < \Re(\ell_0)$ and so $\gamma_0 \notin P$.  Theorem~\ref{K subset P} then gives us that $\gamma_0 \notin K$.

%
%

%
\end{proof}

\begin{cor}\label{Re(gamma_0)=Re(ell_0)}
$\Re(\gamma_0)=\Re(\ell_0)$ if and only if $\alpha =-1$.
\end{cor}

\begin{lem}\label{If L_0 is a subset of K then alpha geq -1}
If $L_0 \subset K$ then $\alpha \geq -1$.
\end{lem}
\begin{proof}
We will show that if $\alpha < -1$ then $L_0 \not \subset K$.  Now assume that $\alpha < -1$.  By Proposition~\ref{Real part of gamma_0} $\Re(\gamma_0)<\Re(\ell_0)$.  By Lemma~\ref{Re(gamma_0)<Re(ell_0)} the point $m \notin P$.  Thus by Theorem~\ref{K subset P} we have that $m \notin K$.  This is a contradiction since $m \in f(L_0) \subset K$.   Thus $\alpha \geq -1$.
\end{proof}

\begin{lem}\label{f(L_{-1}) in L_0}
$f(L_{-1}\cap P \mathbb{H}^-) \subset L_0$.
\end{lem}
\begin{proof} 
$f(L_{-1}\cap P \mathbb{H}^-)=f(L_1\cap P \mathbb{H}^+) \subset g(L_1) = L_0$.
\end{proof}

It is important to notice that there are times when the full line segment $[\ell_{-1},\ell_0] = L_{-1}$ extends past a side of $P$.  One such case is illustrated in Figure~\ref{fig:triangular_P}.

\begin{figure}[htbp]
	\centering
		\includegraphics[width=0.50\textwidth]{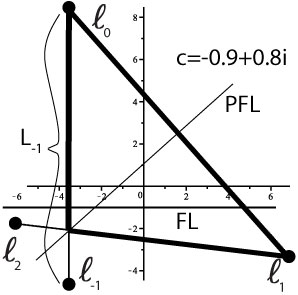}
	\caption{An example where $L_{-1}$ extends past a side of $P$.}
	\label{fig:triangular_P}
\end{figure}

\begin{prop}\label{Im(ell_{-1})}
$\Im(\ell_{-1})>-1$ if and only if $|c| > \sqrt{2}$.
\end{prop}
\begin{proof}
We first note that $\ell_0=\frac{2i(1-\overline{c})}{|c|^2-1}$.  Now
\begin{equation}
\begin{aligned}
\ell_{-1}&=\frac{1}{c}\left(\frac{\overline{\ell_0}}{\overline{c}}-2i\right)\\
&=\frac{1}{c}\left(\frac{-2i(1-c)}{\overline{c}(|c|^2-1)}-2i \right)\\
&=\frac{-2i(1-c)-2i\overline{c}(|c|^2-1)}{|c|^2(|c|^2-1)}\\
&=\frac{-2i(1-c-\overline{c}+\overline{c}|c|^2)}{|c|^2(|c|^2-1)}.\\
\end{aligned}
\end{equation}

So
\begin{equation}
\begin{aligned}
\Im(\ell_{-1})&=\frac{-2(1-2\alpha+\alpha|c|^2}{|c|^2(|c|^2-1)}\\ &=\frac{-2+4\alpha-2\alpha|c|^2}{|c|^2(|c|^2-1)}\\
&=\frac{-2+4\alpha-2\alpha\delta}{\delta(\delta-1)}.
\end{aligned}
\end{equation}
where $\delta=|c|^2=\alpha^2+\beta^2$.  Now the claim that $\Im(\ell_{-1})>-1$ is equivalent to the claim that $\Im(\ell_{-1})+1>0$.  We have:
\begin{equation}
\begin{aligned}
\Im(\ell_{-1})+1 &=\frac{-2+4\alpha-2\alpha\delta}{\delta(\delta-1)}+1\\
&=\frac{-2+4\alpha-2\alpha\delta+\delta^2-\delta}{\delta(\delta-1)}\\
&=\frac{(\delta-2)(\delta+1-2\alpha)}{\delta(\delta-1)}\\
&=(\delta-2)\frac{(\alpha-1)^2+\beta^2}{\delta(\delta-1)}
\end{aligned}
\end{equation}
which is greater than $0$ exactly when $\delta=|c|^2>2$.
\end{proof}

\begin{cor}\label{Im(ell_{-1})=-1}
If $|c|=\sqrt{2}$ then $\Im(\ell_{-1})=-1$.
\end{cor}

\begin{lem}\label{If L_0 in K then the mod(c)<sqrt(2)}
If $L_0 \subset K$ then $|c| \leq \sqrt{2}$.
\end{lem}
\begin{proof}
Since $L_0 \subset K$ then $f(L_0) =[m,\ell_0]\subset (L_{-1}\cap P \mathbb{H}^-)$.  (Otherwise part of $f(L_0)$ would be outside of $P$ and by Theorem~\ref{K subset P} $L_0 \not \subset K$.)  By Lemma~\ref{f(L_{-1}) in L_0} we have that $f(L_{-1}\cap P \mathbb{H}^-) \subset L_0$.  Thus $f^2(L_0) \subset L_0$.  Let $\lambda=|\ell_0-\ell_1|$ be the length of $L_0$.  Then the length of $f(L_0)$ is equal to $\frac{|c| \lambda}{2}$.  Thus, the length of the line segment $f^2(L_0)=\frac{|c|^2\lambda}{2}$.  Since $f^2(L_0) \subset L_0$ then $\frac{|c|^2\lambda}{2}<\lambda$, which implies that $|c| \leq \sqrt{2}$.
\end{proof}

\begin{thm}\label{L_0 is a subset of K}
$L_0 \subset K$ if and only if $\alpha \geq -1$ and $|c| \leq \sqrt{2}.$
\end{thm}
\begin{proof}

Assume that $L_0 \subset K$.  Then by Lemma~\ref{If L_0 is a subset of K then alpha geq -1} and Lemma~\ref{If L_0 in K then the mod(c)<sqrt(2)} we have that $\alpha \geq -1$ and that $|c| \leq \sqrt{2}$.

Now assume that $\alpha \geq -1$ and $|c| \leq \sqrt{2}.$  We will show that $L_0 \subset K$ by proving that $f^2(L_0) \subset L_0$.  It is always true that $f(L_0)=[m,\ell_0]$.  Also, since $|c|\leq \sqrt{2}$, then Proposition~\ref{Im(ell_{-1})} gives us that $\Im(\ell_{-1})\leq -1$ and so $m \in L_{-1}=[\ell_{-1},\ell_0]$.  Since $\alpha \geq -1$ then by Proposition~\ref{Real part of gamma_0} $\Re(\gamma_0)\geq \Re(\ell_0)$.  This means that $\PFL$ cannot cross $L_{-1}$ above $\FL$ and so $[m,\ell_0]\subset (L_{-1}\cap P\mathbb{H}^-)$.  Thus $f^2(L_0)=f([m,\ell_0])\subset f(L_{-1}\cap P\mathbb{H}^-)$ which by Lemma~\ref{f(L_{-1}) in L_0} is a subset of $L_0$.  Thus, $f^2(L_0) \subset L_0$ and so $L_0 \subset K$.
\end{proof}

%
%
%

\begin{lem}\label{Prefolding lines are dense in K}
Let $z \in K$ and let $U$ be an open neighborhood of $z$.  Then there exists an $n>0$ such that $f^n(U) \cap \FL \neq \emptyset$.
\end{lem}
\begin{proof}
Let $B(\varepsilon,z)\subset U$ be a closed ball of radius $\varepsilon>0$ centered at $z\in K$. If the images of $U$ do not intersect $\FL$, then $f^n(B(\varepsilon,z))=B(c^n\varepsilon,f^n(z))$. Since $K$ is bounded, this is impossible, since for a sufficiently large $n$ the largest distance from a point of $K$ to $\FL$ is smaller than $c^n\varepsilon$.
\end{proof}

\begin{lem}\label{L-1capPFL}
If $L_{-1}\cap \PFL\neq \emptyset$ then $\{\zeta\}=L_{-1}\cap \PFL.$
\end{lem}
\begin{proof}
Since $L_{-1} \cap \PFL \neq \emptyset$ then by the symmetry of $L' \cup L_0 \cup L$ about $\PFL$ we have that the outermost boundary of $P$ has 3 sides and $\{\zeta\} = L_{-1} \cap \PFL$.  (See Figure~\ref{fig:Lminus1capPFL})
\end{proof}

\begin{figure}[htbp]
	\centering
		\includegraphics[width=0.40\textwidth]{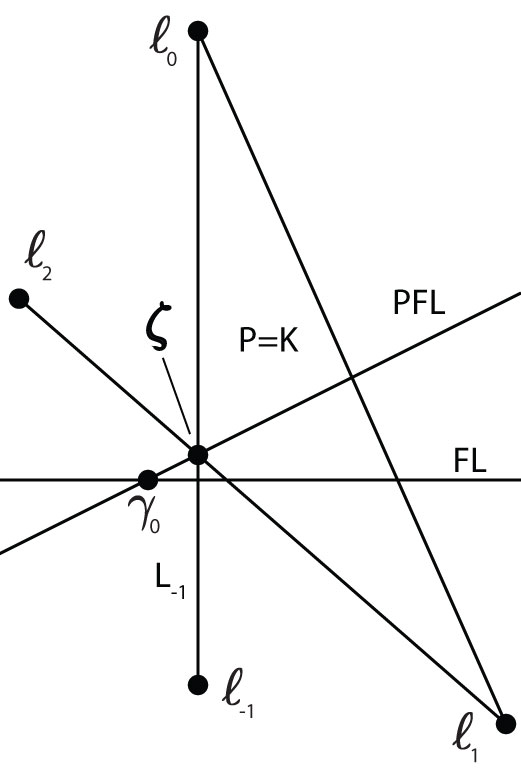}
	\caption[If $L_{-1}\cap \PFL\neq \emptyset$ then $\{\zeta\}=L_{-1}\cap \PFL$]{An example (when $c=-1.06+0.5i$) showing that if $L_{-1}\cap \PFL\neq \emptyset$ then $\{\zeta\}=L_{-1}\cap \PFL$.}
	\label{fig:Lminus1capPFL}
\end{figure}

\begin{lem}\label{numSides}
If $\zeta$ exists, then the outer-most boundary of $P$ has an odd number of sides.
\end{lem}
\begin{proof}
If $\zeta$ exists, then the outer-most boundary of $P$ is a polygon.  Due to symmetry and the fact that $L_0$ is shared by by both $P\mathbb{H}^+$ and $P\mathbb{H}^-$, it is clear that the outer-most boundary of $P$ will have an odd number of sides.  
\end{proof}

Note that when $c$ is purely imaginary and $1<|c|\leq \sqrt{2}$, then $P$ looks like a rectangle, but the sides $L_2$ and $L_{-2}$ are counted separately. In this case, the outer-most boundary of $P$ has 5 sides.

\begin{lem}\label{3or5orInfinity}
If $\zeta$ exists and $\alpha\leq 0$, then the outer-most boundary of $P$ has either 3 or 5 sides.
\end{lem}
\begin{proof}
Suppose $\zeta\in L_n,$ for some $n>0$.  If $n=1,2$, then the outer-most boundary of $P$ has 3 or 5 sides respectively, and we are done.  Thus we may assume that $n\geq 3$ and that $(L_1 \cup L_2) \cap \PFL= \emptyset$.  In particular this means that $\ell_3\in P\mathbb{H}^+$.  Then $\alpha \leq 0$ implies $\theta \geq \frac{\pi}{2}$ and so $L_3$ starts at $\ell_3$ and slopes away from $\PFL$.  Thus, $L_3\cap \PFL=\emptyset$.  It is now easy to see that for $\zeta \in L_n$ to be true, it is neccessary for $L$ to have a self-intersection along at least one of $L_1, L_2, L_3$.  However, Proposition~\ref{LcapFL} would then imply that $(L_1\cup L_2 )\cap \FL \neq \emptyset$.  But this means that $(L_1 \cup L_2 \cup L_3)\cap \PFL \neq \emptyset$, a contradiction.
\end{proof}

\begin{prop}\label{Im(zeta) less than -1 implies alpha geq -1 and |c| leq sqrt(2)}
If $\Im(\zeta)\leq -1$ then $\alpha \geq -1$ and $|c|\leq \sqrt{2}$.
\end{prop}
\begin{proof}
By Lemma~\ref{numSides} we may assume the outer-most boundary of $P$ has an odd number of sides.  

Let $\Im(\zeta) \leq -1$ and assume that $\alpha < -1$.  Then by Proposition~\ref{Re(gamma_0)<Re(ell_0)} $\Re(\gamma_0) < \Re(\ell_0)$.  This means that if $L_{-1}\cap \FL\neq \emptyset$ then $L_{-1} \cap \PFL =\{\zeta\}$.  Clearly, in this case $\Im(\zeta)>-1$, a contradiction.  On the other hand, if $L_{-1}\cap \FL = \emptyset$ then $\Im(\ell_{-1})>-1$ and by Proposition~\ref{Im(ell_{-1})} $|c|>\sqrt{2}$.

Thus, if $\Im(\zeta) \leq -1$ then $\alpha \geq -1$.

We now assume that $\Im(\zeta)\leq -1$ and that $\alpha \geq -1$ and show that for every possible number of sides that the outer-most boundary of $P$ can have, that $|c|\leq \sqrt{2}$.

Case 1:  Assume that the outer-most boundary of $P$ has 3 sides as shown in Figure~\ref{fig:3sides}.  Then $\{\zeta\}=L_{-1} \cap \PFL$ which implies that $\Im(\ell_{-1}) \leq \Im(\zeta)\leq -1$.  By Proposition~\ref{Im(ell_{-1})} we have that $|c|\leq \sqrt{2}$.
\begin{figure}[htbp]
	\centering
		\includegraphics[width=0.50\textwidth]{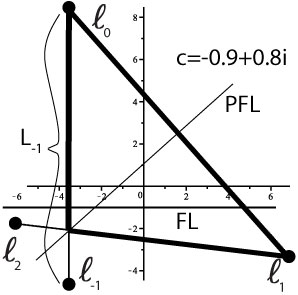}
		
	\caption{An example where the outer-most boundary of $P$ has 3 sides.}
	\label{fig:3sides}
\end{figure}

\begin{figure}[htbp]
	\centering
		\includegraphics[width=0.70\textwidth]{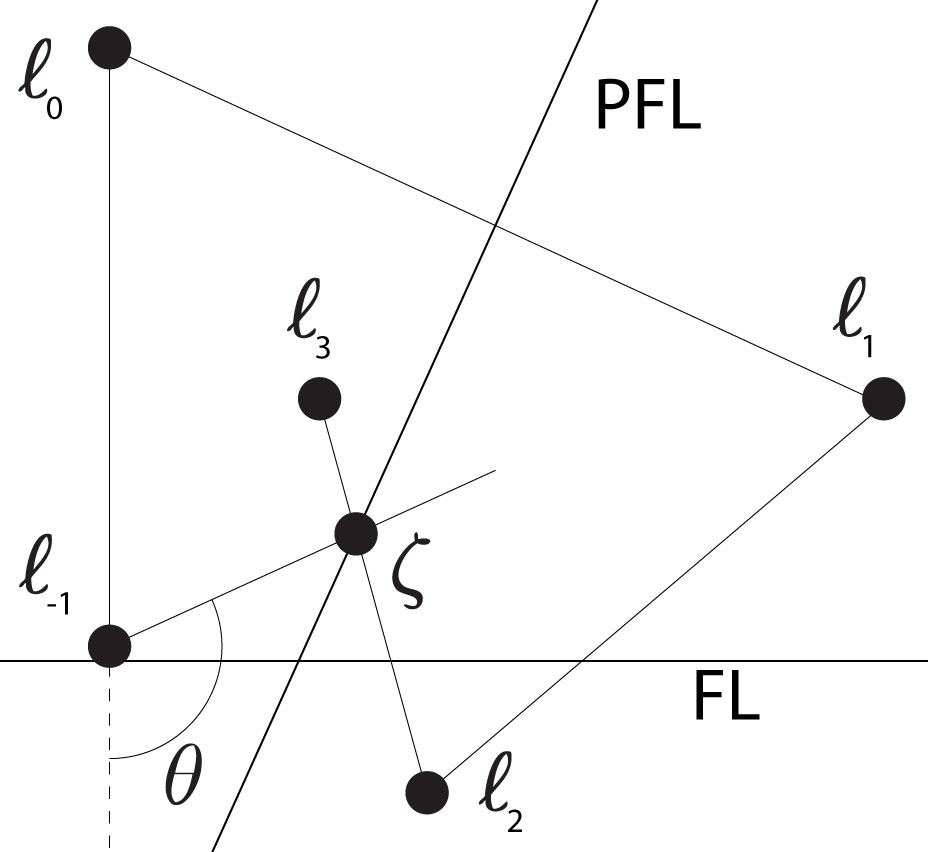}
	\caption[Example where the outer-most boundary of $P$ has 5 sides and $\theta > \pi/2$.]{An example where the outer-most boundary of $P$ has 5 sides and $\theta > \pi/2$.  This illustrates that if $\theta > \pi/2$ then $\Im(\zeta)>\Im(\ell_{-1})$.}
	\label{fig:thetaLargerThanHalfPi}
\end{figure}

Case 2:  Assume that the outer-most boundary of $P$ has 5 sides.  We have 3 subcases.
\begin{enumerate}
	\item If $\theta > \pi/2$, as shown in Figure~\ref{fig:thetaLargerThanHalfPi}, then $\Im(\zeta)>\Im(\ell_{-1})$ so that $\Im(\ell_{-1})\leq -1$. 
	\item If $\theta = \pi/2$, as shown in Figure~\ref{fig:thetaEqualsPiHalves}, then $\Im(\zeta)=\Im(\ell_{-1})$ so that $\Im(\ell_{-1})\leq -1$.
	\item If $\theta < \pi/2$, as shown in Figure~\ref{fig:5sides}, then $\zeta \in L_2\cap \PFL$ implies that $L_1 \cap \FL \neq \emptyset$.  Thus, $\Im(\ell_2)\leq -1$.  Since $\theta < \pi/2$ then the tilt of $\PFL$ guarantees that $\Im(\ell_{-1})<\Im(\ell_2)\leq -1$. 
\end{enumerate}
In each of these cases, Proposition~\ref{Im(ell_{-1})} implies that $|c|\leq \sqrt{2}$.

\begin{figure}[htbp]
	\centering
		\includegraphics[width=0.70\textwidth]{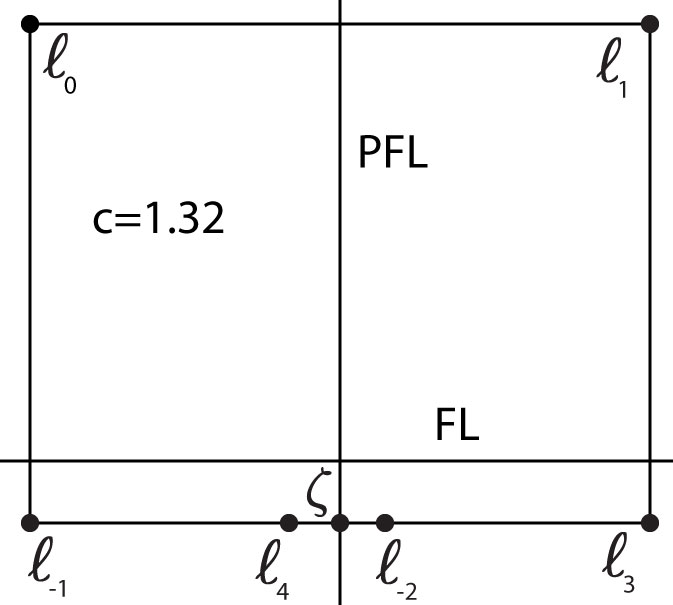}
	\caption[Example where $P$ is a rectangle but still has 5 sides.]{An example where $P$ is a rectangle but still has 5 sides.}
	\label{fig:thetaEqualsPiHalves}
\end{figure}

\begin{figure}[htbp]
	\centering
		\includegraphics[width=.70\textwidth]{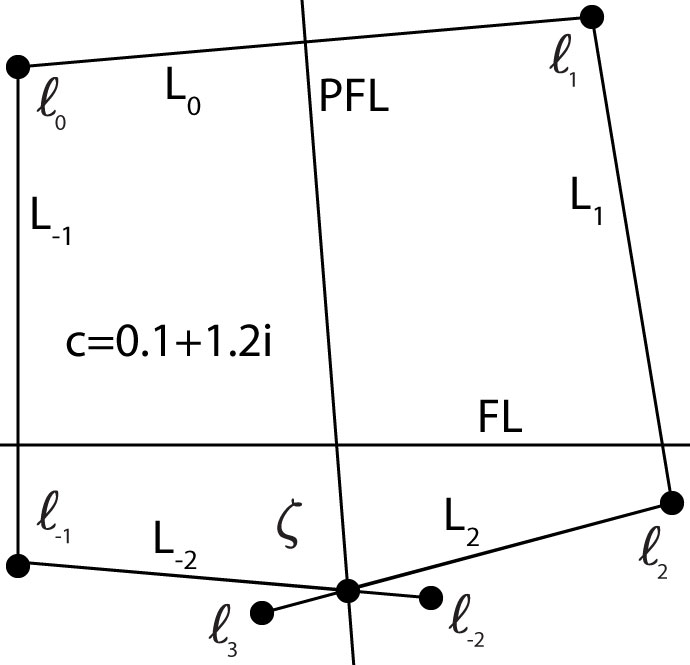}
	\caption[Example where the outer-most boundary of $P$ has 5 sides and $\theta <\pi/2$.]{An example where the outer-most boundary of $P$ has 5 sides and $\theta <\pi/2$.}
	\label{fig:5sides}
\end{figure}

Case 3:  Assume that the outer-most boundary of $P$ has 7 sides.  This implies that $\zeta \in L_3\cap \PFL$ and so $L_2 \cap \FL \neq \emptyset$.  The assumption that $\Im(\zeta)\leq -1$ then implies that $L_1 \cap \FL = \emptyset.$  We now look at three subcases.
\begin{enumerate}
	\item By Lemma~\ref{3or5orInfinity} $\alpha$ can not be less than 0.
	\item Assume $0 < \alpha \leq 3/4$.  As shown in Figure~\ref{fig:SevenSides}, $\Im(\zeta)\leq -1$ if and only if $\Re(\ell_0)\leq \Re(c^4\gamma_0)$.  We have:
	
\begin{equation}
\Re(c^4\gamma_0)= \left( {\alpha}^{4}-6\,{\alpha}^{2}{\beta}^{2}+{\beta}^{4} \right) 
 \left( {\frac {\alpha-1}{\beta}} \right) +4\,{\alpha}^{3}
\beta-4\,\alpha\,{\beta}^{3},\\
\end{equation}

and $\Re(\ell_0)=\frac{-2\beta}{\alpha^2+\beta^2-1}.$  Now $\Re(\ell_0)\leq \Re(c^4\gamma_0)$ is equivalent to $0\leq \beta(\alpha^2+\beta^2-1)( \Re(c^4\gamma_0)-\Re(\ell_0))$, which is equivalent to:

\begin{equation}
0\leq  \left( {\beta}^{2}+1-2\,\alpha+{\alpha}^{2} \right)  \left( {
\alpha}^{5}+{\alpha}^{4}-2\,{\alpha}^{3}{\beta}^{2}-3\,\alpha\,{\beta}
^{4}+4\,\alpha\,{\beta}^{2}+2\,{\beta}^{2}-{\beta}^{4} \right) \\
\label{eq:SevenSidedStar}
\end{equation}

We now make the change of variables $\beta^2=t-\alpha^2$ in \eqref{eq:SevenSidedStar} and get:

\begin{equation}
0\leq \left( t+1-2\,\alpha \right)  \left( 4\,{\alpha}^{3}t-3\,
\alpha\,{t}^{2}+4\,t\alpha-4\,{\alpha}^{3}+2\,t-2\,{\alpha}^{2}-{t}^{2
}+2\,t{\alpha}^{2} \right)\\
\label{eq:SevenSidedStarWithT}
\end{equation}

\begin{figure}[htbp]
	\centering
		\includegraphics[width=1.00\textwidth]{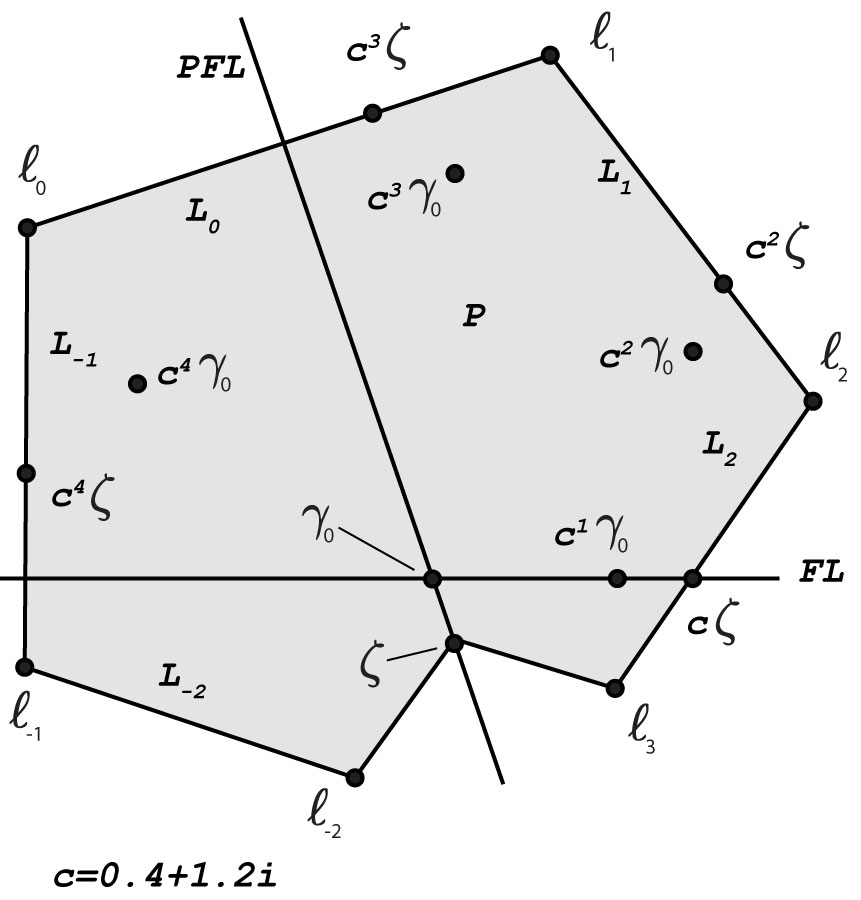}
	\caption[Outer-most boundary when $P$ has 7 sides and $\alpha>0$]{An illustration showing that if $\alpha>0$ and if the outer-most boundary of $P$ has 7 sides, then $\Im(\zeta)\leq -1$ if and only if $\Re(c^4\gamma_0) \geq \Re(\ell_0)$.}
	\label{fig:SevenSides}
\end{figure}

We wish to show that when $0< \alpha\leq 3/4$ \eqref{eq:SevenSidedStarWithT} is not true for $t \geq 2$, which corresponds to $|c|\geq \sqrt{2}$.  Now $t+1-2\alpha>0$ when $0< \alpha\leq 3/4$ and $t \geq 2$.  Thus, to get our contradiction, we need to show:

\begin{equation} 4\,{\alpha}^{3}t-3\,\alpha\,{t}^{2}+4\,t\alpha-4\,{\alpha}^{3}+2\,t-2
\,{\alpha}^{2}-{t}^{2}+2\,t{\alpha}^{2}<0.\\
\label{eq:SevenSidesFactor}
\end{equation}  

For $t=2$, \eqref{eq:SevenSidesFactor} becomes:

\begin{equation}
4\,{\alpha}^{3}-4\,\alpha+2\,{\alpha}^{2}=2\,\alpha\, \left( 2\,{\alpha}^{2}-2+\alpha \right) <0.\\
\label{eq:SevenSidesT2}
\end{equation}

Now $\left( 2\,{\alpha}^{2}-2+\alpha \right)$ has exactly one positive root, which is $\frac{\sqrt{17}-1}{4}>\frac{3}{4}$.  Thus, \eqref{eq:SevenSidesT2} is easily seen to be true for $0<\alpha\leq 3/4$.

Now we need to show that \eqref{eq:SevenSidesFactor} is still true for $t>2$, keeping $0< \alpha \leq 3/4$.  For this it suffices to show that the partial derivative of the left hand side of \eqref{eq:SevenSidesFactor} is negative for $t>2$ and $0<\alpha \leq 3/4$.  We have:

\begin{equation}
\begin{aligned}
 \frac{\partial}{\partial t}& \left(4\,{\alpha}^{3}t-3\,\alpha\,{t}^{2}+4\,t\alpha-4\,{\alpha}^{3}+2\,t-2
\,{\alpha}^{2}-{t}^{2}+2\,t{\alpha}^{2}\right)\\
=& 4\,{\alpha}^{3}-6\,t\alpha+4\,\alpha+2-2\,t+2\,{\alpha}^{2}\\
=& \alpha(4\alpha^2-6t+4+2\alpha)+2-2t.\\
\end{aligned}
\label{eq:SevenSidesPartial}
\end{equation}

By our choices of $t,\alpha$ we have $\alpha>0$ and $2-2t<0$.  Thus we have only to show that $4\alpha^2-6t+4+2\alpha<0$.  Now
\begin{equation}
4\alpha^2-6t+4+2\alpha \leq \frac{9}{4}-12+4+\frac{3}{2}=\frac{-17}{4}<0.\\
\end{equation}

This completes the proof that if the outer-most boundary of $P$ has 7 sides, and $0<\alpha\leq 3/4$, then $\Im(\zeta)\leq -1$ implies $|c|\leq \sqrt{2}$.

	\item  Assume that $\alpha>3/4$.  We will assume that the outer-most boundary of $P$ has 7 or more sides, which means that $\Im(\ell_{3+n})\leq -1$ for some $n=0,1,2,...$.  Then assuming $|c|\geq \sqrt{2}$ we show that $|\ell_3|<1$ by proving that $|\ell_3|^2<1$, a contradiction.

We will need the following estimate:
\begin{equation}
\begin{aligned}
|1-\overline{c}|^2&=|(1-\alpha)+\beta i|^2\\
&=1-2\alpha+\alpha^2+\beta^2\\
&\leq \frac{-1}{2}+|c|^2.\\
\end{aligned}
\label{eq:alphaBigBound}
\end{equation}

Applying the estimate given in \eqref{eq:alphaBigBound} and using the assumption that $|c|\geq \sqrt{2}$, we now show that $|\ell_3|^2<1$.

\begin{equation}
\begin{aligned}
|\ell_3|^2&=\left|\frac{\ell_0^2}{c^6}\right|\\
&=\frac{4|1-\overline{c}|^2}{|c|^6(|c|^2-1)^2}\\
&\leq \frac{4\left(\frac{-1}{2}+|c|^2\right)}{|c|^6(|c|^2-1)^2}\\
&\leq\frac{-2}{|c|^6}+ \frac{4}{|c|^4}< 1.\\
\end{aligned}
\label{eq:alphaBig}
\end{equation}
\end{enumerate}

Case 4:  Assume the outer-most boundary of $P$ has 9 or more sides.  Lemma~\ref{3or5orInfinity} implies this cannot happen if $\alpha \leq 0$.  Thus we may assume $0<\alpha$.  Assume also that $\Im(\zeta)\leq -1$.  Then if $\zeta \in L_k$ then $L_{k-1}$ is the first segment of $L$ to intersect $ \FL$.  Thus $\Im(\ell_k)\leq -1$ and so $|\ell_k|\geq 1$.  We will now show this is impossible when $|c|\geq \sqrt{2}$, proceeding very much as before.

We will need the following estimate, which relies on $\alpha>0$:
\begin{equation}
\begin{aligned}
|1-\overline{c}|^2&=|(1-\alpha)+\beta i|^2\\
&=1-2\alpha+\alpha^2+\beta^2\\
&=1-2\alpha+|c|^2\\
&\leq 1+|c|^2.\\
\end{aligned}
\end{equation}

Now we show that if $|c|\geq \sqrt{2}$ then $|\ell_4|^2<1$.

\begin{equation}
\begin{aligned}
|\ell_4|^2&=\left|\frac{\ell_0^2}{c^8}\right|\\
&=\frac{4|1-\overline{c}|^2}{|c|^8(|c|^2-1)^2}\\
&\leq \frac{4\left(1+|c|^2\right)}{|c|^8(|c|^2-1)^2}\\
&\leq\frac{4}{|c|^8}+\frac{4}{|c|^6}\\
&\leq \frac{7}{8}< 1.\\
\end{aligned}
\end{equation}

This completes the proof.
\end{proof}

\begin{figure}[htbp]
	\centering
		\includegraphics[width=1.0\textwidth]{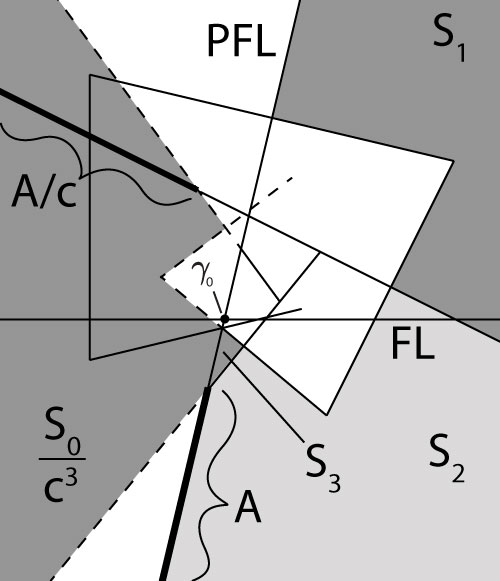}
	\caption{An example where $S_3=\frac{S_0}{c^3}\cap P\mathbb{H}^+$.}
	\label{fig:S_k_and_S_0}
\end{figure}

\begin{lem}\label{S_k_and_S_0}
$S_2=\frac{S_0}{c^2}\cap P\mathbb{H}^+$.  Also, for every integer $k \geq 3$, if $S_{k-1} = \emptyset$, then $S_k=\emptyset$.  Furthermore, if $S_{k-1}\neq \emptyset$ and if $\gamma_0 \notin \underset{n=2}{\overset{k-1}{\bigcup}}S_n$ then $S_{k}=\frac{S_0}{c^{k}}\cap P\mathbb{H}^+$.
\end{lem}
\begin{proof}
(See Figure~\ref{fig:S_k_and_S_0})  By definition $S_0=c S_1$.  Thus $S_2=\frac{S_1}{c}\cap P\mathbb{H}^+ = \frac{S_0}{c^2}\cap P\mathbb{H}^+$.  Clearly, $S_2 \neq \emptyset$ since $M_1 \subset S_2$ and $0<\theta<\pi$.  If $S_{k-1}=\emptyset$, then $S_k=\emptyset$ trivially.  Now fix an integer $k \geq 3$ and assume that $S_{k-1}=\frac{S_0}{c^{k-1}}\cap P\mathbb{H}^+\neq \emptyset$ and that $\gamma_0 \notin \underset{n=2}{\overset{k-1}{\bigcup}}S_n$.  Let $A=\PFL \cap S_{k-1}$.  Since $\gamma_0 \notin S_{k-1}$ then $A\subset \mathbb{H}^-\setminus \FL.$  (Here, $A$ may be the empty set.)  Thus $\frac{A}{c}\subset P\mathbb{H}^-\setminus \PFL$ and so $\frac{A}{c}\cap S_k=\emptyset$.  Thus $S_k=\frac{S_{k-1}}{c}\cap P\mathbb{H}^+=\frac{S_0}{c^{k}}\cap P\mathbb{H}^+.$
\end{proof}

\begin{figure}[htbp]
	\centering
		\includegraphics[width=1.0\textwidth]{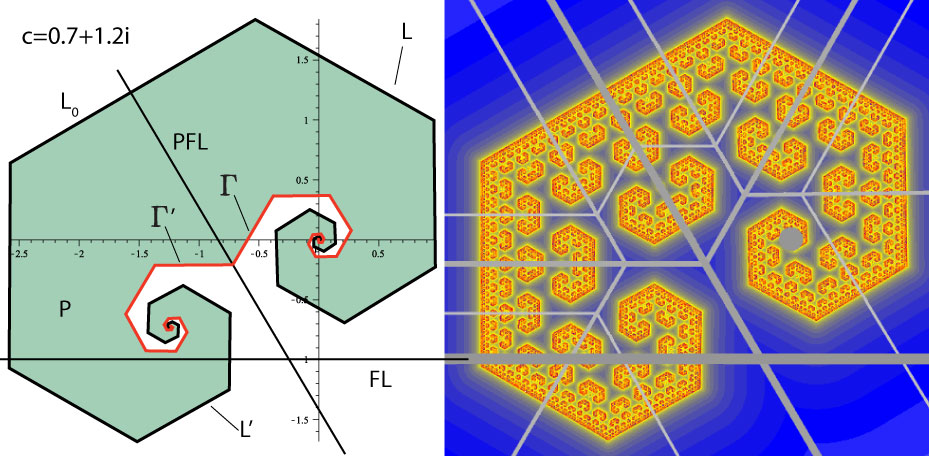}
	\caption[An example where $P$ resembles a ram's head]{An example where $P$ resembles a ram's head whose horns spiral and never self-intersect.}
	\label{fig:rams_head}
\end{figure}

We now wish to define the ray $M_0$ which intuitively is the ray perpendicular to $L_0$ that starts at $m_0$ and ``radiates outward'' from $P$.  

\begin{Def}
We define $M_0=(S_0 \cap S_1)$.  Also, for all integers $k>0$, we define $M_k=\frac{M_0}{c^k}$ and $M_{-k}=M_k '$.
\end{Def}

\begin{thm}\label{thm:noZeta}
If $\zeta$ does not exist, then:
\begin{enumerate}
	\item $P$ is a ram's head bounded by inner-most and outer-most boundaries of $P$,
	\item $\gamma_0 \notin P$,
	\item $\gamma_0 \notin K$,
	\item $K \subsetneq P.$
\end{enumerate}
\end{thm}
\begin{proof}
Since $|c| >1$, then $L$ spirals in toward the origin.  Furthermore, since $\zeta$ does not exist, then $L \subset P\mathbb{H}^+$.  Since $L_n$ intersects $S_n$ non-trivially, then $S_n \neq \emptyset$, for $n=1,2,...$.  Now the collection of sets $\frac{S_0}{c^n}$, $n=1,2,...$ cover $\mathbb{C} \setminus \{0\}$.  Thus $\gamma_0 \in \frac{S_0}{c^{n}}$ for some integer $n>0$.  In fact, there may be infinitely many such positive values for $n$ and we will denote the smallest of these values by $k$.  By the choice of $k$, $\gamma_0 \notin \underset{n=2}{\overset{k-1}{\bigcup}}\frac{S_0}{c^n}$ and so $\gamma_0 \notin \underset{n=2}{\overset{k-1}{\bigcup}}S_n$.  As previously argued, $S_{k-1} \neq \emptyset$.  Then by Lemma~\ref{S_k_and_S_0} $\gamma_0 \in \frac{S_0}{c^k}\cap P\mathbb{H}^+=S_k$.  Since $\gamma_0 \in S_k$ then $\Im(M_k\cap \PFL)\geq -1$.  Then $M_{k-1}\cap \FL \neq \emptyset$.  Thus $S_k \cap \mathbb{H}^+$ is nonempty and bounded.  It follows immediately that $S_{k+n}=\frac{S_k \cap \mathbb{H}^+}{c^n}$ for $n>0$ and so $\underset{n=1}{\overset{\infty}{\bigcup}}S_n$ spirals in toward the origin.  It is now easy to see that $P$ resembles a ram's head as shown if Figure~\ref{fig:rams_head}.

Now if $S_k \cap \FL = \{\gamma_0\}$ then $\gamma_0 \in S_{k-1}$ contradicting our choice of $k$.  Letting $A=(S_k\cup S_{k+1}) \cap \PFL$ we see that $\gamma_0\in A$ and that $\gamma_0$ is not an endpoint of the line segment $A$.  Then $\gamma_0 \in \Int(S_k \cup S_k' \cup S_{k+1} \cup S_{k+1}')$.  Thus, $\gamma_0 \notin P$.  Since $K \subset P$, then $\gamma_0 \notin K$.  Lastly, since $\gamma_0 \notin P$ and since $P$ is a closed set, then there is a neighborhood $U\ni \gamma_0$ such that $U\cap P = \emptyset$.  Now $\gamma_1 \in P$ and every neighborhood of $\gamma_1$ gets mapped into a neighborhood of $\gamma_0$.  Thus, there is a small enough neighborhood $V \ni \gamma_1$ such that $f(V) \subset U$ implying that $V \not \subset K$ and so $K \subsetneq P.$  
\end{proof}

\begin{figure}[htbp]
	\centering
		\includegraphics[width=1.0\textwidth]{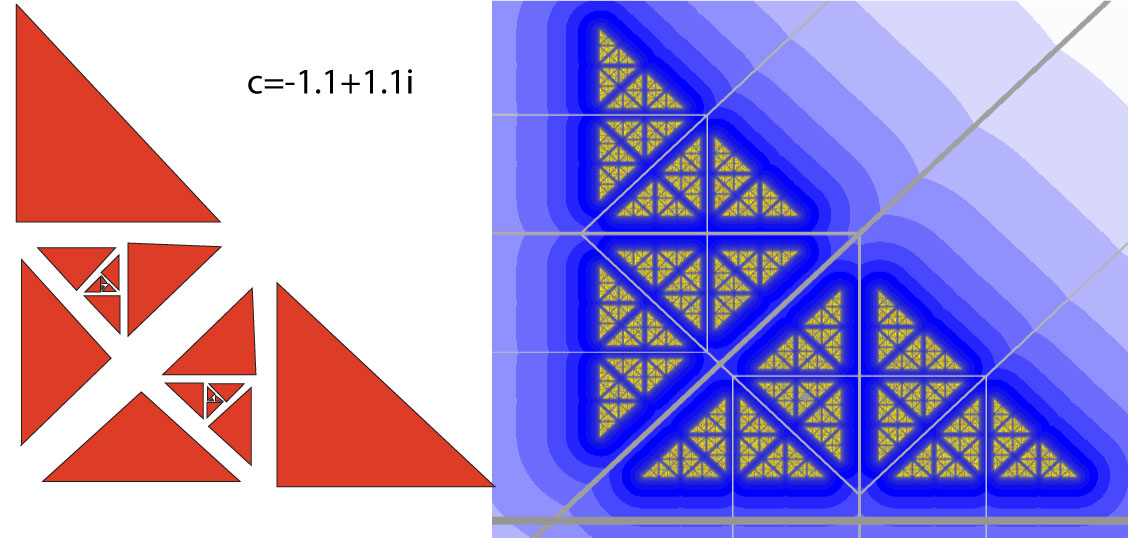}
	\caption{An example of $P$ when $\zeta$ does not exist and $K$ is totally disconnected.}
	\label{fig:totallyDisconnected}
\end{figure}

\begin{thm}\label{thm:zetaLow}
If $\zeta$ exists and $\Im(\zeta)\leq -1$, then: 
\begin{enumerate}
  \item $P$ is a polygon bounded by the outer-most boundary of $P$,
  \item $P=K$,
  \item $\gamma_0 \in K.$
\end{enumerate}
\end{thm}
\begin{proof}
Since $\zeta$ exists, then $\zeta \in L \cap P\mathbb{H}^+ \subset \underset{n=1}{\overset{\infty}{\bigcup}}S_k$.  Therefore $\zeta \in S_k$ for some $k>0$.  Since $\Im(\zeta) \leq -1$, then $S_k \subset \mathbb{H}^-$, for if not, then $m_{k-1}\in (S_k \cap S_{k-1})$ would be in $\mathbb{H}^+\setminus \FL$.  Then $m_k \in (P\mathbb{H}^+ \setminus \PFL)$ contradicting the assumption that $\zeta \in S_k$.  Thus, $S_k \subset \mathbb{H}^-$ and so either $S_{k+1}=\emptyset$ or $S_{k+1}\subset \PFL$.  In both cases it is easy to see that $S_{k+n} \subset P$ for $n>0$.  This means that $P$ is simply connected.  Thus, $P$ is a polygon bounded by the outer-most boundary of $P$ which must contain $\gamma_0$ since $\Im(\zeta)\leq -1$.  (See Figure~\ref{fig:m})

\begin{figure}[htbp]
	\centering
		\includegraphics[width=1.00\textwidth]{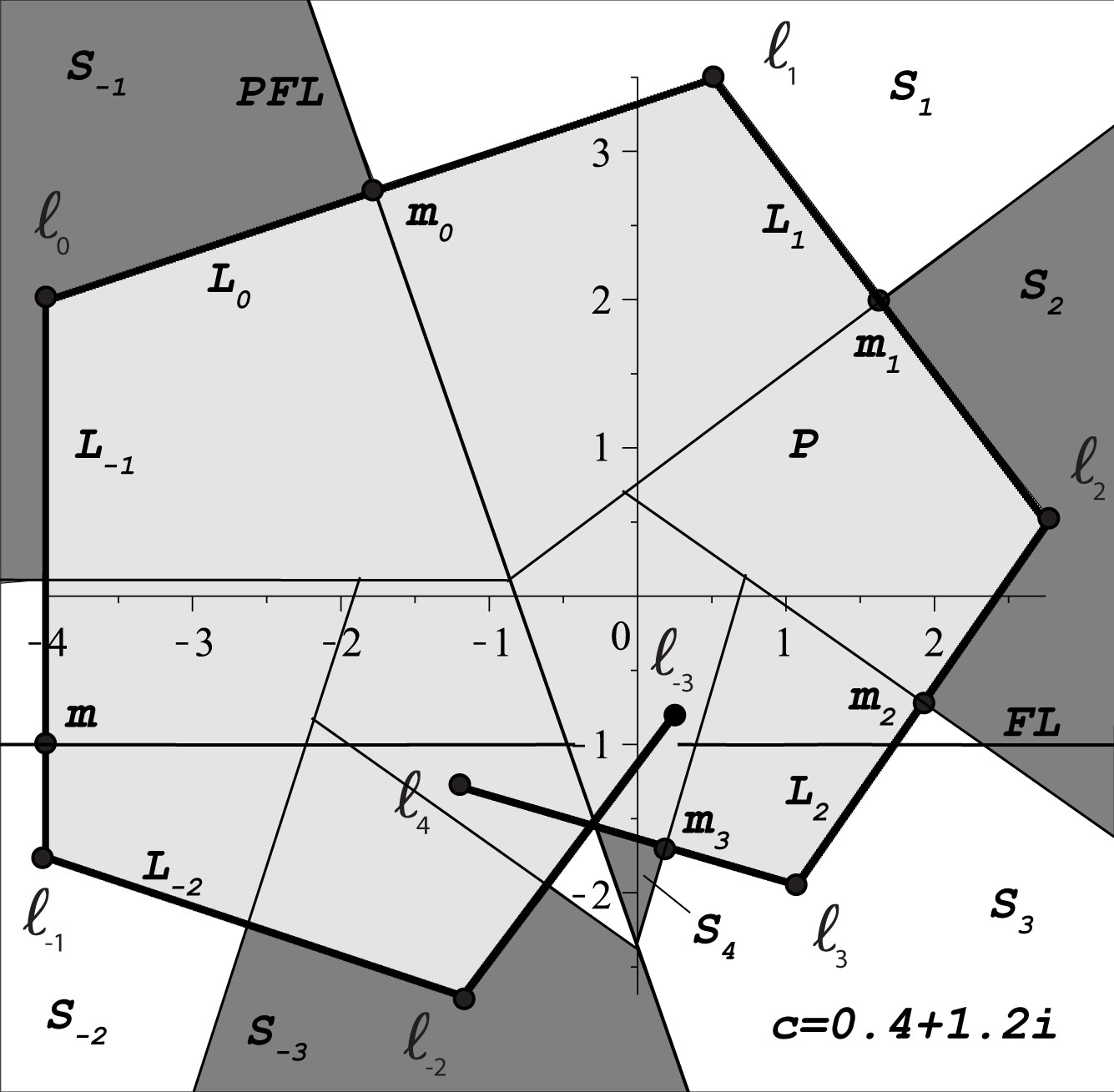}
	\caption{An illustration showing the locations of $m_k,k=0,1,2,3$. }
	\label{fig:m}
\end{figure}

Let $A=P\cap P\mathbb{H}^+$.  It is easy to see that $f(P)=f(A)=cA$ and so $\Bd(f(P))=f(\Bd(A)).$  Thus, to show that $f(P) \subset P$ we only need to show that $f(\Bd(A))\subset P$.  Now $\Bd(A)\subset \left(\underset{n=1}{\overset{k}{\bigcup}}L_n \right)\cup[\zeta,m_0]\cup[m_0,\ell_1]$.  By the construction of $L$ we have that $f(\underset{n=1}{\overset{k}{\bigcup}}L_n) \subset P$.  By Proposition~\ref{Im(zeta) less than -1 implies alpha geq -1 and |c| leq sqrt(2)} and Theorem~\ref{L_0 is a subset of K} we have that $L_0\subset K$.  This implies that $f(L_0)\subset K$.  Since $K \subset P$ then $f(L_0) \subset P$.  Since $m \in f(L_0)$ then $m \in P$.  Since $\Im(\zeta)\leq -1$ and since $m \in P$ then $f([\zeta,m_0])=[m,c\zeta]\subset P$.  Thus, $f(\Bd(A))\subset P$ and so $f(P) \subset P$ and $P \subset K$.  By Theorem~\ref{K subset P} $K \subset P$ and so $P=K$.  Since $\gamma_0 \in P$ then $\gamma_0 \in K$. (See Figure~\ref{fig:P_into_P})
\end{proof}

\begin{figure}[htbp]
	\centering
		\includegraphics[width=1.00\textwidth]{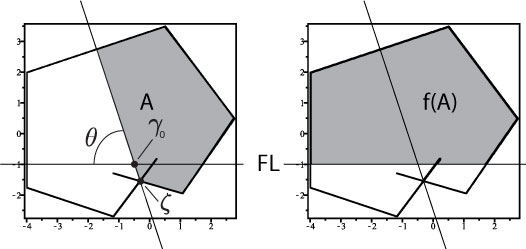}
	\caption{An example where $f(P) \subset P$.}
	\label{fig:P_into_P}
\end{figure}

\begin{thm}\label{thm:zetaHigh}
If $\zeta$ exists and $\Im(\zeta) > -1$ then: 
\begin{enumerate}
  \item $\gamma_0 \notin K.$
  \item $P$ is not the polygon bounded by the outer-most boundary of $P$.
  \item $K \subsetneq P$
\end{enumerate}
\end{thm}
\begin{proof}
Since $|c| >1$, then $L$ spirals in toward the origin.  Since $\zeta$ exists, then $\zeta \in L_k$ for some $k>0$.  We may assume that $k$ is the smallest positive integer such that $\zeta \in L_k$.  Since $\Im(\zeta)>-1$, (and by our choice of $k$) then $\gamma_0 \in \Int(S_k \cup S_k')$.  Thus, $\gamma_0 \notin P$ and so $\gamma_0 \notin K$.  Since $L_k \cap \PFL \neq \emptyset$ then $L_{k-1} \cap \FL \neq \emptyset$.  By our choice of $k$ we see also that $L_k \cap \FL \neq \emptyset$.  Since $L_k$ intersects both $\FL$ and $\PFL$, then we can let $A$ be the closed triangle bounded by $ \FL$, $\PFL$, and $L_k$.  It is clear that $A$ has nonempty interior since $\Im(\zeta) >-1$ (see Figure \ref{fig:zetaHighProof}).

It follows immediately that $\Int(S_{k+1})=\Int(\frac{A}{c})\neq \emptyset$.  Now let $X$ be the polygon bounded by outer-most boundary of $P$.  Then $P \subset X$ and $A \not \subset X$.  By definition $P=X \setminus \underset{n=k+1}{\overset{\infty}{\bigcup}}S_n\subset X\setminus \Int(\frac{A}{c})$.  Since $\Int(A) \neq \emptyset$ and by Proposition \ref{P is right of ell_0}, there are points in $\Int(A)$ with real part greater than $\Re(\ell_0)$.  Since $f(L_0)$ is the vertical line segment from $\ell_0$ to $\FL$, then is easy to see that $\Int(A/c)\cap X \neq \emptyset$.  This means that $P$ is not the polygon bounded by the outer-most boundary of $P$ and is instead this polygon minus at most countably many open sets.  Each of these open sets are bounded on two sides by preimages of $\FL$.  Now (when $\zeta$ exists) the outer-most boundary of $P$ must have at least 3 sides.  Let $B=\frac{h^{-1}(A)}{c}$ (see Figure \ref{fig:zetaHighProof}).  Then $\Int(B)\cap P \neq \emptyset$ and $f^2(B)=A$.  Thus, $K \subset P$, $P \neq K$ and so $K \subsetneq P$.
\end{proof}

\begin{figure}[htbp]
	\centering
		\includegraphics[width=1.00\textwidth]{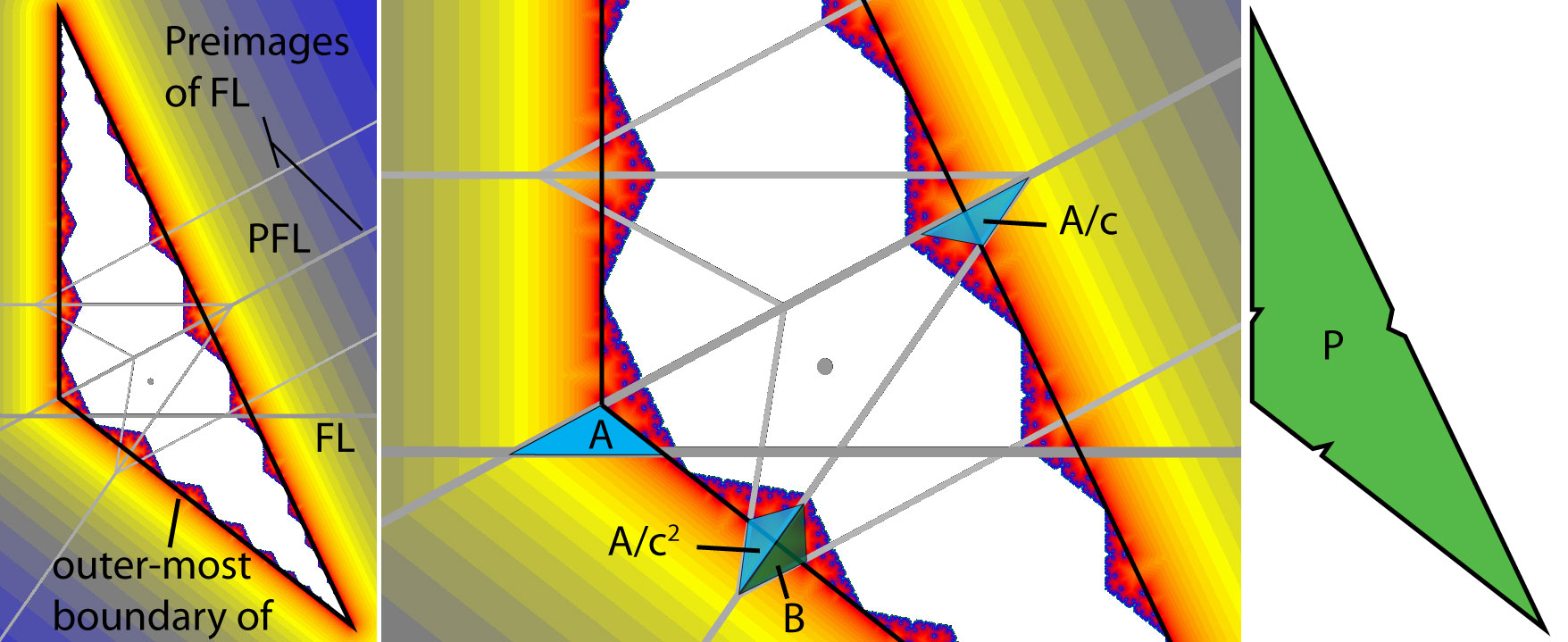}
	\caption[$P$ when $\Im(\zeta)>-1$]{The left-most image is $K(c)$ surrounded by the outer-most boundary of $P$.  The middle image shows the construction of the sets $A,B$ used in Theorem \ref{thm:zetaHigh}.  The right-most image is the corresponding perimeter set $P$.}
	\label{fig:zetaHighProof}
\end{figure}

Another example of $P$ when $\Im(\zeta)>-1$ is shown in Figure \ref{fig:P_alpha_dot9}.  In this example, $P$ is a polygon with countably many open sets removed.  Compare this to Figure \ref{fig:zetaHighProof} which shows on the far right a $P$ which is the polygon bounded by the outer-most boundary of $P$ with finitely many open sets removed.

\begin{figure}[htbp]
	\centering
		\includegraphics[width=1.00\textwidth]{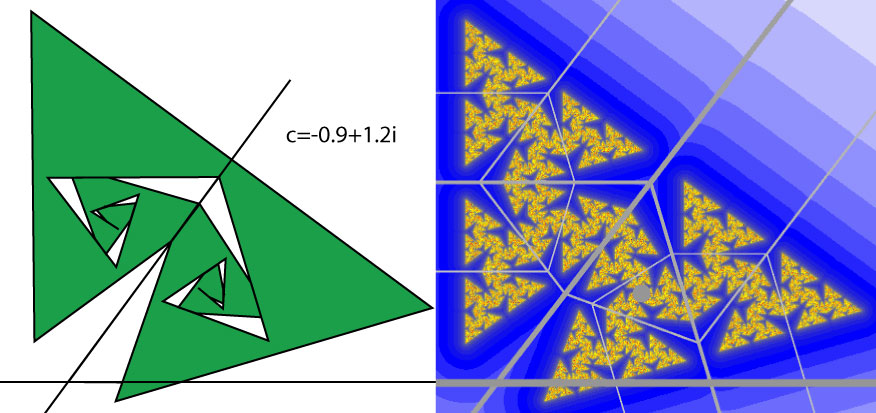}
	\caption{An example of $P$ and $K$ when $\zeta$ is above $\FL$.}
	\label{fig:P_alpha_dot9}
\end{figure}
\end{section}

\begin{section}{Hungry Sets and Crises}\label{Hungry Sets and Crises}

Proposition~\ref{One attracting periodic point} states that there are no attracting periodic orbits (other than the point at infinity if we are on the Riemann sphere.)  However, we can still have sets that attract in some sense.  

\begin{Def}
A compact set $A$ with positive 2-dimensional Lebesgue measure, and with the property that $f(A)=A$ will be called a \textbf{hungry set} if every point whose
orbit contains a subsequence that converges to $A$ eventually lands inside $A$ (gets eaten).
\end{Def}

Hungry sets often ``attract'' sets of positive Lebesgue measure.  For this reason, we will borrow (and possibly modify) some familiar terminology.

\begin{Def}
The basin of attraction (or consumption, if you like) of a hungry set $A$ is definined as $\underset{j=0}{\overset{\infty}{\bigcup}}f^{-j}(A)$.
\end{Def}

Note that under this definition, a basin of attraction is not neccessarily open.

\begin{lem}\label{KisHungry}
If $K$ has nonempty interior, then $K$ is a hungry set.
\end{lem}
\begin{proof}
$K$ is a compact set with positive Lebesgue measure and with the property that $f(K)=K$.  Furthermore, if $z \notin K$ then it spends finite time near $K$ and thus, there is a positive lower bound to how close the orbit of $z$ comes to $K$.
\end{proof}

It is important to note that a hungry set is not necessarily contained in an open subset of its basin of attraction.  In fact, Lemma~\ref{KisHungry} implies that no $K$ with positive 2-dimensional Lebesgue measure can be contained in an open forward invariant set.

\begin{lem}\label{A subset K}
If $A$ is a hungry set then $A \subset K\cap \mathbb{H}^+$.
\end{lem}
\begin{proof}
Since every point outside of $K$ diverges, it is easily seen that no compact forward invariant sets exist outside of $K$.  Also, $f(A)\subset f(K)\subset K\cap \mathbb{H}^+$ and we are done.
\end{proof}

\begin{prop}
If $A$ is a hungry set with a nontrivial component (more than one point) $C$, then there exists an $n>0$ such that $f^n(C) \cap \FL \neq \emptyset$.
\end{prop}
\begin{proof}
Since $C$ is a nontrivial component, then it is connected and $\Diam(C)=d>0$.  By Lemma~\ref{A subset K} $C \subset K$ and so the sequence $(\Diam(f^n(C)))$ is bounded.  But $|c|>1$ and if $f(C) \cap \FL = \emptyset$ then $\Diam(f(C))=|c|\Diam(C)>\Diam(C)$.  The result follows easily.
\end{proof}

Note that the union of two hungry sets is a hungry set.  For this reason, we need the following definition.

\begin{Def}
We say that a hungry set is \textbf{reducible} if it contains a proper subset which is a hungry set.  Otherwise, a hungry set is said to be \textbf{irreducible}.
\end{Def}

\begin{Def}
A hungry set $A$ is called \textbf{greedy} if it is equal to the closure of its basin of attraction.  (That is, $A$ is greedy if it has already eaten everything it possibly can.)
\end{Def}

\begin{lem}\label{hungry and greedy}
A hungry set is greedy if and only if it is fully invariant.
\end{lem}
\begin{proof}
This follows immediately from the definitions.
\end{proof}

\begin{prop}
Let $A$ be a greedy hungry set.  Then $f(\Bd(A))\subset \Bd(A)$.
\end{prop}
\begin{proof}
This is consequence of full-invariance.  Let $z \in \Bd(A)$ and suppose $f(z)\in A \setminus \Bd(A)$.  Then there is an open set $V \subset \Int(A)$ where $f(z)\in V$.  Since $f$ is continuous and $A$ is fully invariant, then $U=f^{-1}(V)$ is an open subset of $A$ containing $z$.  This contradicts the assumption that $z \in \Bd(A)$.
\end{proof}

\begin{Def}
We will denote the closure of a set $X$ by $\Cl(X)$.\label{symbol:Cl}
\end{Def}

\begin{Def}
Recall that the \textbf{omega-limit set of $z$} is \label{symbol:omega}
\begin{equation*}
\omega(z,f)= \underset{n \in \mathbb{N}}{\bigcap} \Cl(\{f^k(z):k>n\}).
\end{equation*}
\end{Def}

\begin{Def}
We will call the boundary of any set that is significant dynamically a \textbf{dynamical boundary}.  A continuous change in parameter can cause dynamical boundaries to come together, meet, and then cross.  We will adapt the term \textbf{boundary crisis} to mean the meeting of dynamical boundaries.
\end{Def}

Some examples of dynamical boundaries include the outer-most boundary of $P$, the inner-most boundary of $P$, the boundary of hungry sets, $\FL$, and the boundary of any periodic sets.

We now discuss some experimental results giving each type of result its own subsection.  Let $A$ be a hungry set.  We will give examples where the following seem to occur.

\begin{enumerate}
	\item An example where the coded-coloring of $K$ shows periodic structures in $K$.
	\item An example where $A$ is a topological annulus.
	\item Renormalization.
	\item Examples where boundary crises cause sudden changes in dynamics.
	\item An example where $A$ has multiple components bounded by a topological annulus.
	\item Examples where an increase in the modulus of $c$ causes the components of $A$ to swell.  This swelling leads to boundary crises causing sudden changes in dynamics.
	\item An example where $K$ contains multiple disjoint hungry sets.
\end{enumerate}

\subsection{The Coded-Coloring of $K$ Shows Periodic Structures in $K$}

A traditional way to make pictures when studying the dynamics in one complex variable, is to color a pixel based on how quickly the trajectory of the corresponding point in the plane leaves a region known to contain $K$.  This region is usually a ball whose radius is called the \textbf{bailout value}.  This coloring method is known as the escape time algorithm.  If $K$ has nonempty interior, then using this escape time algorithm produces colorful pictures like the first one in Figure~\ref{fig:TBcoloring}, where there is a large black region (which is $K$) surrounded by color.  But the escape time algorithm alone does not give any information as to what is happening on the inside of $K$.  In Figure~\ref{fig:TBcoloring} the last three pictures use the escape time algorithm but also use a new coloring method we call the \textbf{coded-coloring}.  Informally, for TTM's, the method of coded-coloring is an escape time algorithm for how long it takes for the orbit of a point to be in $P\mathbb{H}^+$ $N$ times, where $N$ is the bailout value.  We now give a more formal and general definition of coded-colorings.
	
\begin{Def}\label{symbol:eta}
Let $X$ be a topological space and let $f$ be a map from $X$ to $X$.  Let $\varphi$ be a map from $X$ to $\mathbb{R}$ and for every $x \in X$ we call the sequence $S(x)=(\varphi(f^n(x)))=(x_n), n=0,1,2,...$ the \textbf{itinerary of x.} We will abuse notation and will let $\eta^n(x)=\underset{j=0}{\overset{n}{\Sigma}}x_j.$  Then a \textbf{coded-coloring of $X$} is any escape-time coloring under iteration of $\eta$.
\end{Def}

\begin{figure}[htbp]
	\centering
		\includegraphics[width=1.00\textwidth]{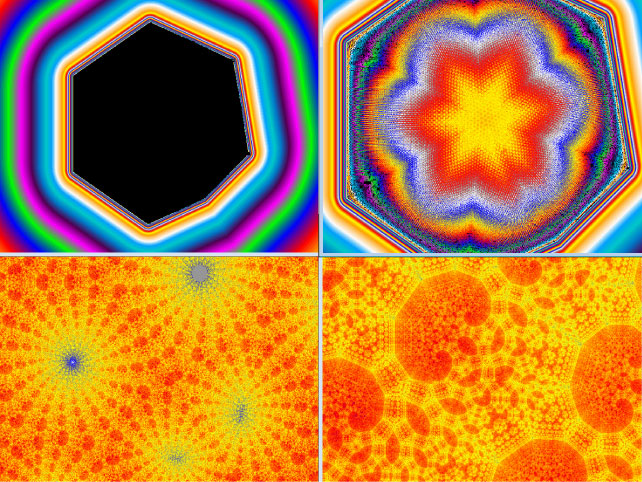}
	\caption[The coded-coloring of $K$]{This shows first the traditional escape time picture of the filled-in Julia set followed by the three progressively closer looks inside $K$ when using the coded-coloring.  A close look at the coded-coloring reveals that there is structure inside of $K$ and gives clues as to what that structure is.  In each picture $c=0.5567+0.8471i$.}
	\label{fig:TBcoloring}
\end{figure}

Most of our pictures use the fastest coloring between the traditional escape time coloring and the coded-colorings.  For our coded-colorings we typically use a bailout value between 80 and 400 and have $\varphi$ defined by

\begin{equation}
\varphi(z)=
\begin{cases}
1 & \text{if } z \in P\mathbb{H}^+, \\ 
0 & \text{if } z\in P\mathbb{H}^-.
\end{cases}
\end{equation}

The common refinement of a partition is a new partition defined by intersection of preimages of a partition.  The coded-coloring is a way to color the common refinement of a partition in a way that is sometimes useful.  However, the bailout value used in the coded-coloring needs to be chosen appropriately based on the size of the structures for which one is looking. Larger structures most easily seen with a coded-coloring bailout value that is low, but can become very hard to see with much larger values.  Obviously, to see very small structures, a very fine partition is needed and this is achieved by using a large bailout value.

Their are two known reasons why the coded-coloring shows the internal structure of $K$.  First, points that stay close together for a long time will tend to be colored similarly.  Second, the boundaries of the structures that appear in the coded-coloring are repelling periodic structures.  Thus, patterns of color accumulate on the repelling structures.  This is comparable to the Theorem in Rational Complex Dynamics that the preimages of almost every point will accumulate on the Julia set.
	
It is worth noting that a coded-coloring can be used to see structure in the filled-in Julia sets of quadratic complex polynomials.  Attracting periodic points cause large groups of points to be colored the same way.  Also, since the Julia set is the closure of the set of periodic repelling points, then patterns of color accumulate only on the Julia set.  Figure~\ref{fig:julia} shows the filled-in Julia set for the map $f(z)=z^2-0.6$.  Because the parameter is real, then $K$ is symmetric about the real axis.  To have a changes in color occur when crossing a preimage of the real axis, we used:

\begin{equation}
\varphi(z)=\left\{
\begin{array}{lrlr} 1 &, \Im(z)\geq 0 \\ 
0 &, \Im(z)<0.
\end{array}\right.
\end{equation}

\begin{figure}[htbp]
	\centering
		\includegraphics[width=0.70\textwidth]{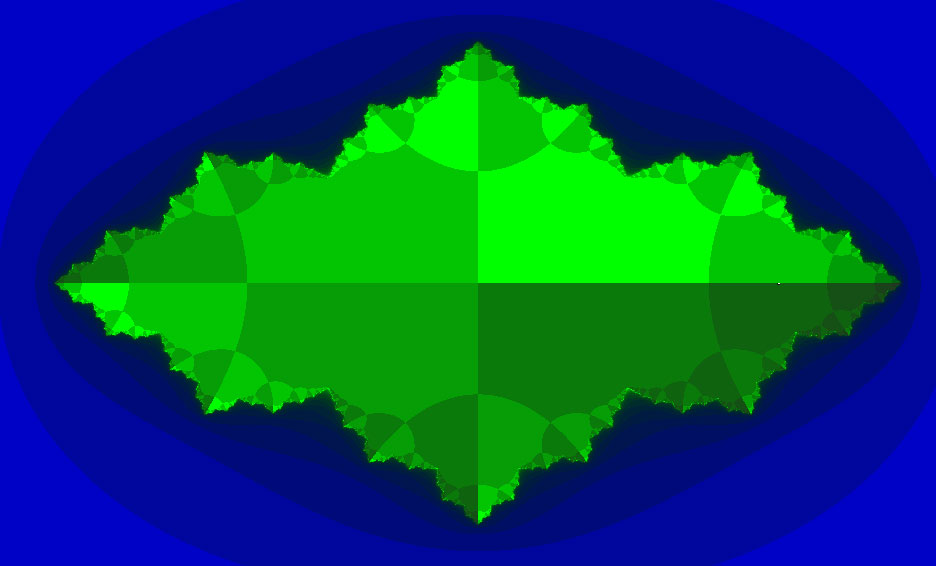}
	\caption{The coded-coloring of the filled-in Julia set of the map $f(z)=z^2-0.6$.}
	\label{fig:julia}
\end{figure}

\subsection{A Hungry Set that is a Topological Annulus}
	Figure~\ref{fig:annulus} shows an example of where $A$ is expected to be a topological annulus bounded inside and out by the images of $ \FL \cap A$.  This figure shows the trajectory of a point in $A$.  It seems likely then that $A=\omega(z,f)$ for some $z \in A$.
	
\begin{figure}[htbp]
	\centering
		\includegraphics[width=0.62\textwidth]{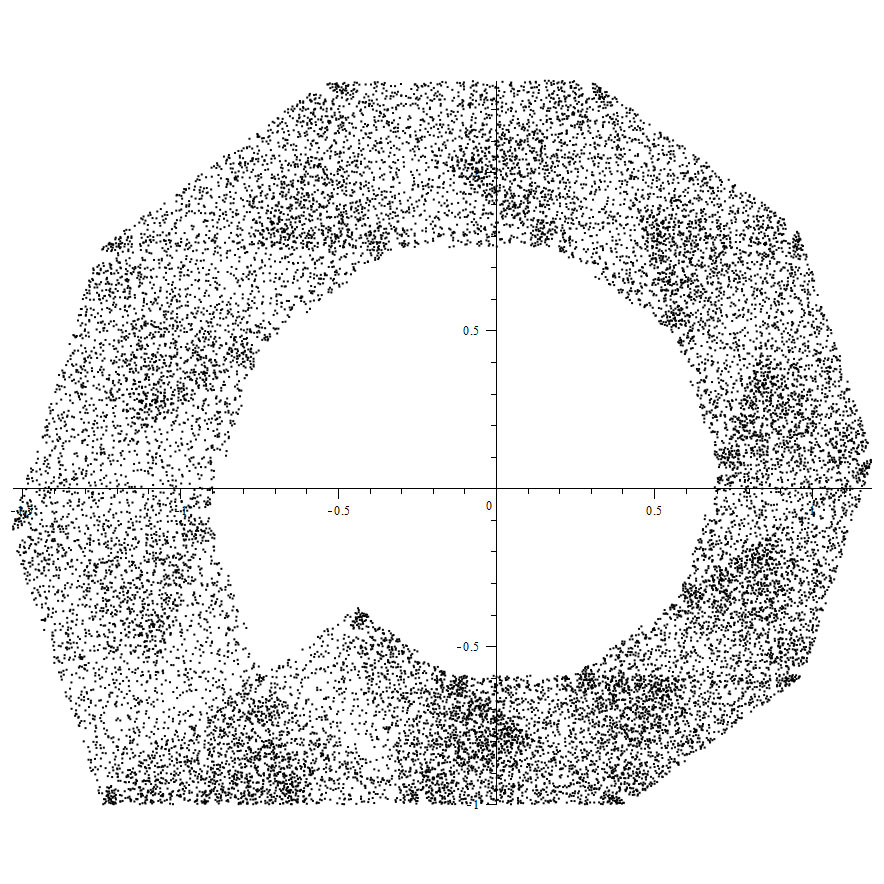}
	\caption{An example where $A$ is a topological annulus.}
	\label{fig:annulus}
\end{figure}

	\subsection{Renormalization}

Renormalization is an important topic that is studied in papers like \cite{Milb} and \cite{Hub}.  The following examples seem to suggest that the tools of renormalization could be used to study TTM's.

%

Letting $c=1.05i$ we have $c^4=1.21550625\in \mathbb{R}$.  Then by Theorem~\ref{K for real c} $K(c^4)=[-2i/c,0]\approx[-1.6454i,0]$.  It is easily seen that for all $z \in K(c^4)$, $f_c ^4(z)=f_{c^4}(z)$.  In particular, this means that $K(c^4)$ is embedded by the identity map into $K(c)$.  The embedded structure of $K(c^4)$ and its preimages under $f_c$ can be seen in the coded-coloring of $K(c)$ given in Figure~\ref{fig:rectangles}.  A separate example is given in Figure~\ref{fig:embedding}, which shows an overlay of $K(c^3)$ and $K(c)$ for $c=-0.6+0.9i$.
		
\begin{figure}[htbp]
	\centering
		\includegraphics[width=0.60\textwidth]{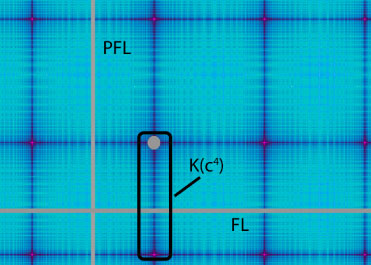}
	\caption[Simple structures found in a coded-coloring of $K$]{The coded-coloring near the origin of $K(1.05i)$.  The origin is marked and $K((1.05)^4)$ and its preimages under $f_c$ are seen as straight line segments.}
	\label{fig:rectangles}
\end{figure}

\begin{figure}[htbp]
	\centering
	\includegraphics[width=0.80\textwidth]{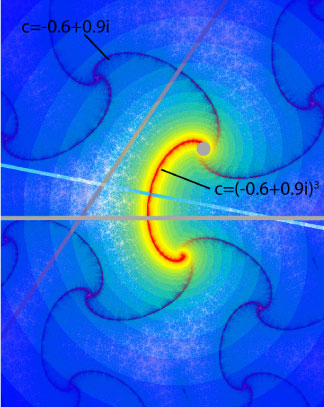}
	\caption{The coded-coloring of $K(c)$ overlaid with $K(c^3)$.}
	\label{fig:embedding}
\end{figure}

\begin{prop}
Let $c=-0.65+0.88i$.  There is an affine transformation $\varphi$ such that for every $z \in K(|c|^2c)$, we have $\varphi(f_{|c|^2c}(z))=f_c^3(\varphi(z))$.  Furthermore, $\varphi(K(|c|^2c))$ is a period 3 component of a hungry set in $K(c)$.
\end{prop}
\begin{proof}
	Figure~\ref{fig:conjugacyP3} shows an example where the embedding map is an affine transformation.  Let $c=-0.65+0.88i$, $g(z)=cz$, and $h(z)=\overline{cz}-2i$.  It is easily checked that if $z_0=\frac{2i(1-\overline{c})}{|c|^2 c-1}$, then $f_c ^3(z_0)=h(h(g(z_0)))=z_0$.  Let $\varphi(z)=(\Im(z_0)+1)z+z_0$.  Let $B$ be a ball of radius 5.85 centered at $0+i$.  The left-most picture in Figure~\ref{fig:conjugacyproof} shows that $K(|c|^2c)\subset B$.  The right-most picture in Figure~\ref{fig:conjugacyproof} shows that $\varphi(B)$ is contained in a non-shaded region, which is the set of points where $f_c ^2(z)=h(g(z))$.  It is then a straightforward computation to show that for every $z \in K(|c|^2c)$, we have $\varphi(f_{|c|^2c}(z))=f_c^3(\varphi(z))$.	This conjugacy also implies that $\varphi(K(|c|^2c))$ is a period 3 component of a hungry set in $K(c)$.
\end{proof}
	
\begin{figure}[htbp]
	\centering
		\includegraphics[width=1.00\textwidth]{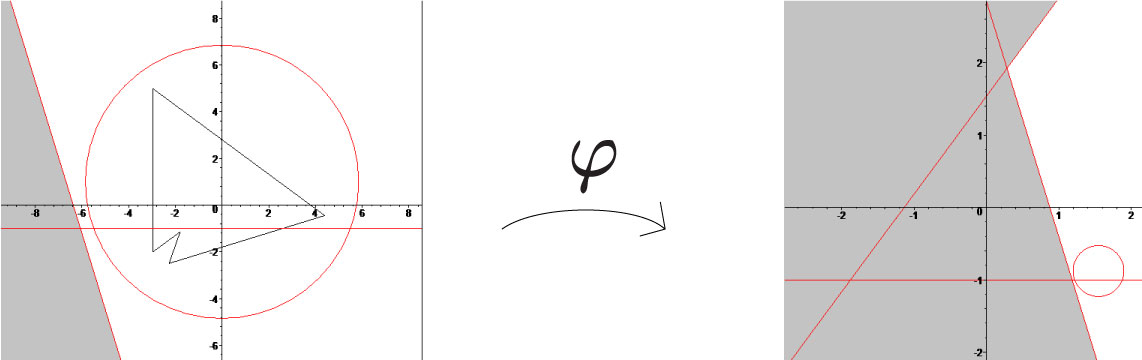}
	\caption[The image of $K(|c|^2c)$ under a conjugating map]{An illustration showing $K(|c|^2c)$ is contained in a ball whose image under $\varphi$ is contained in the non-shaded region, which is the set of points where $f_c ^2(z)=h(g(z))$.}
	\label{fig:conjugacyproof}
\end{figure}

\begin{figure}[htbp]
	\centering
		\includegraphics[width=1.00\textwidth]{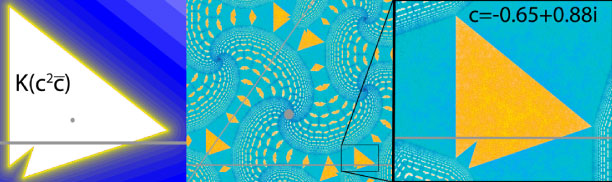}
	\caption[$K(|c|^2 c)$ embeds into $K(c)$ by an affine transformation]{An example of where $K(|c|^2 c)$ embeds into $K(c)$ by an affine transformation.}
	\label{fig:conjugacyP3}
\end{figure}
	
Figures~\ref{fig:Decomp1} and~\ref{fig:Decomp2} show pairs of images grouped vertically.  For each pair, the top image is $K(c)$ and the bottom image is an overlay of up to three of $K(c),K(c^2),$ and $cK(c^2)$.  These pictures suggest some of the more complicated $K(c)$ can often be seen as the result of piecing together several other filled-in Julia sets.

\begin{figure}[htbp]
	\centering	
		\includegraphics[width=1.0\textwidth]{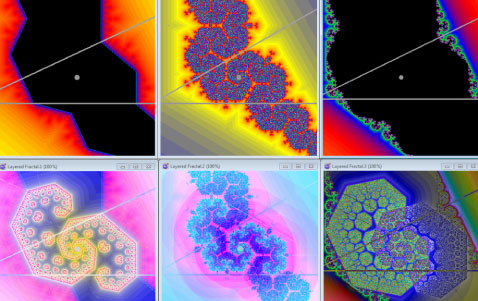}
		\caption{$K(c)$ can sometimes be decomposed into smaller filled-in Julia sets.}
	\label{fig:Decomp1}
\end{figure}

\begin{figure}[htbp]
	\centering
		\includegraphics[width=1.0\textwidth]{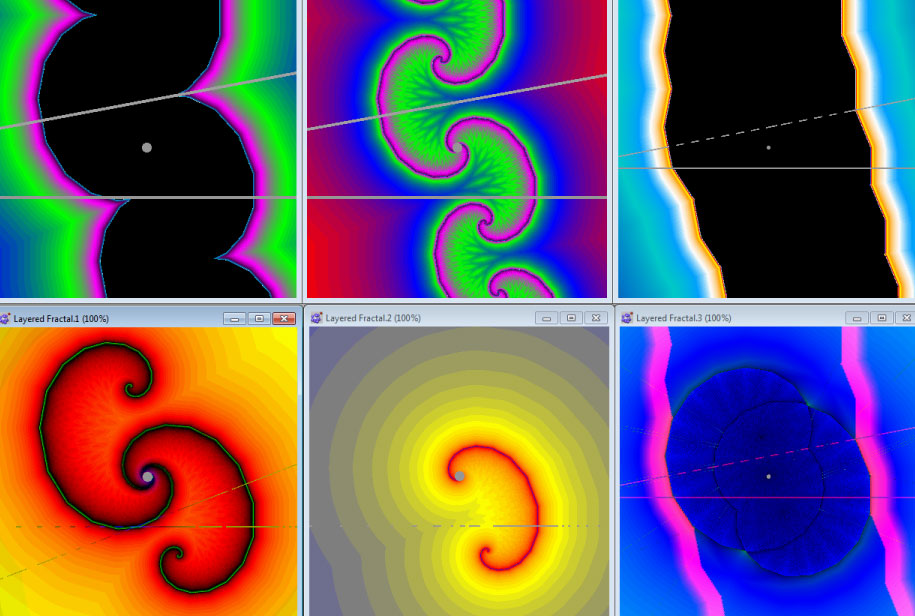}
		\caption[Composite dynamics]{$K(c)$ can have nonempty interior when areas within it are bounded by structures locally conjugate to other filled-in Julia sets.}
	\label{fig:Decomp2}
\end{figure}	

	\begin{Def}
	The embedded image of one filled-in Julia set into another will be called a \textbf{sub-K}.
	\end{Def}

	\subsection{Boundary Crises and Sudden Changes in Dynamics}

By perturbing the parameter so that $c=-0.04+1.05i$ we find that $c^4=1.20492481+0.18495120i\notin \mathbb{R}$, and yet $K(c^4)$ is still embedded into $K(c)$ by the identity map.  This is seen in the coded-coloring of $K(c)$ given in Figure~\ref{fig:axes}.

\begin{figure}[htbp]
	\centering
		\includegraphics[width=0.90\textwidth]{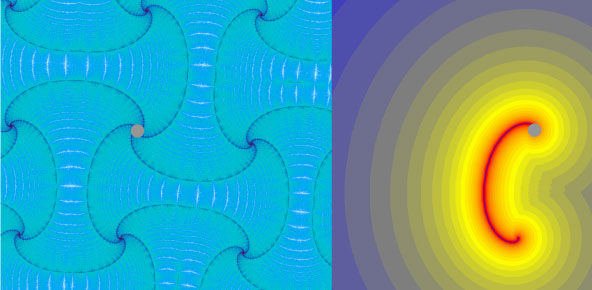}
	\caption[An embedding of $K(c^4)$]{On the left is the coded-coloring near the origin of $K(c)$ where $c=-0.04+1.05i$.  The origin is marked and $K(c^4)$ and its preimages under $f_c$ are seen.  On the right is $K(c^4)$.}
	\label{fig:axes}
\end{figure}

In both of the cases given by Figures~\ref{fig:rectangles} and~\ref{fig:axes} applying $f_c$ four times was equivalent to applying $f_{c^4}$ once.  This is due to no point getting folded more than once every four iterations.  

The coded-coloring becomes more interesting when a continuous change in parameter brings$ \FL$ and a preimage of $K(c^4)$ together causing a boundary crisis.  A small perturbation of parameter can cause these dynamical boundaries to cross.  This leads to the creation of islands as shown in Figure~\ref{fig:branchingOut} and is the same mechanism as was discussed in reference to Figure~\ref{fig:KcontinuousChange}.

\begin{figure}[htbp]
	\centering
		\includegraphics[width=1.0\textwidth]{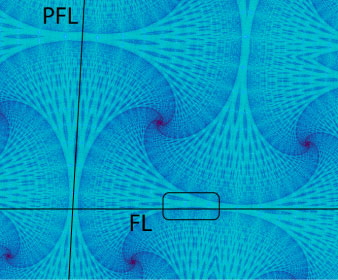}
	\caption[The creation of ``islands'']{The coded-coloring near the origin of $K$ when $c=-0.05+1.05i$.  Comparing this to the left image in Figure~\ref{fig:axes},  we see that islands are created when a structure is introduced into $\mathbb{H}^+$ from below.  The source of the islands is indicated.}
	\label{fig:branchingOut}
\end{figure}

Another boundary crisis occurs when $c$ is perturbed until $K(c^4)$ touches $\PFL$.  If $K(c^4)$ crosses $\PFL$, then there is a loss of the fixed point $\ell_0$ of $f_{c^4}$.  This in turn leads to a loss of structure.  An example of this is given in Figure~\ref{fig:washed}.  

\begin{figure}[htbp]
	\centering
		\includegraphics[width=1.0\textwidth]{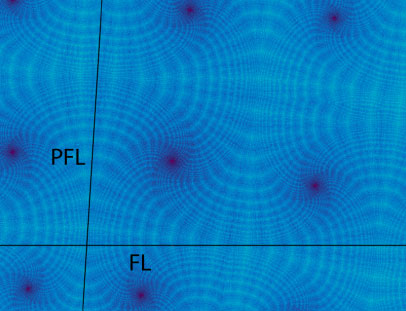}
	\caption[Loss of a periodic point]{The coded-coloring near the origin of $K$ when $c=-0.06+1.05i$.  Compared to Figure~\ref{fig:axes} you can only see faint (or transient) structure.  This is because $K(c^4)$ is no longer invariant under $f^4$.}
	\label{fig:washed}
\end{figure}

	\subsection{Hungry Set with Multiple Components in an Annulus}

	In Figure~\ref{fig:lossOfStructure} shows the coded-colorings of Figures~\ref{fig:branchingOut} and~\ref{fig:washed} with 10,000 points in the orbit of $0.1-i$ overlayed.  This suggests that the boundary crisis can cause a hungry set with multiple components to diffuse into a larger forward invariant set.  Often, the larger forward invariant set is a topological annulus.
	
\begin{figure}[htbp]
	\centering
		\includegraphics[width=1.00\textwidth]{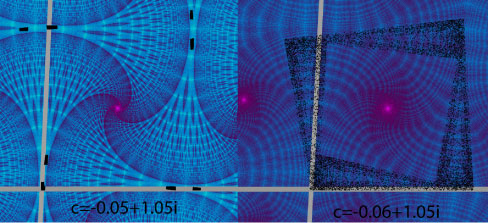}
	\caption[A change in parameter causing a hungry set to leak]{A hungry set with periodic components becoming unstable and  filling out a topological annulus  bounded inside and out by images of $ \FL$.}
	\label{fig:lossOfStructure}
\end{figure}

\begin{figure}[htbp]
	\centering
		\includegraphics[width=1.0\textwidth]{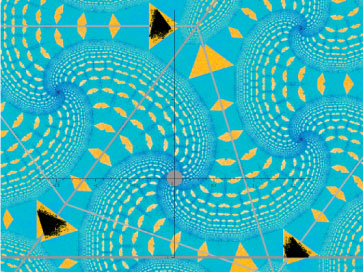}
	\caption[A hungry set with multiple components]{An example where a hungry set has multiple components.  The orbit of a point in $A$ is marked for the a few thousand iterations.}
	\label{fig:multiComponentAnnulus}
\end{figure}

\subsection{Swelling the Components of an Hungry Set}

Figure~\ref{fig:Swell} shows multiple examples of where a hungry set $A$ consists of multiple components bounded by a topological annulus.  Each of these components is contained in a periodic region of period 5.  Let $A_1$ be one of these components.  Then $f^5(A_1)=A_1$ and $f^5(B_1)=B_1$ where $B_1$ is the basin of attraction of $A_1$.  The boundary of $B_1$ consists of images of $ \FL$ and embedded sub-K's.  As the modulus of $c$ is increases, the size of each of these components increases until the boundary of one component (actually, each simultaneously) crosses into the basin of attraction of another component of $A$.  This results in the orbit of a point under $f^5$ staying in one section for a long time before moving into an adjacent section where the process is repeated.	Once the orbit of points begin to hop components, then $A$ is a topological annulus.  Further increasing the modulus of $c$ will eventually lead to the annulus swelling until it contains the origin, at which point $A$ is a topological disk.
	
	\begin{figure}[htbp]	
	\centering
		\includegraphics[width=1.00\textwidth]{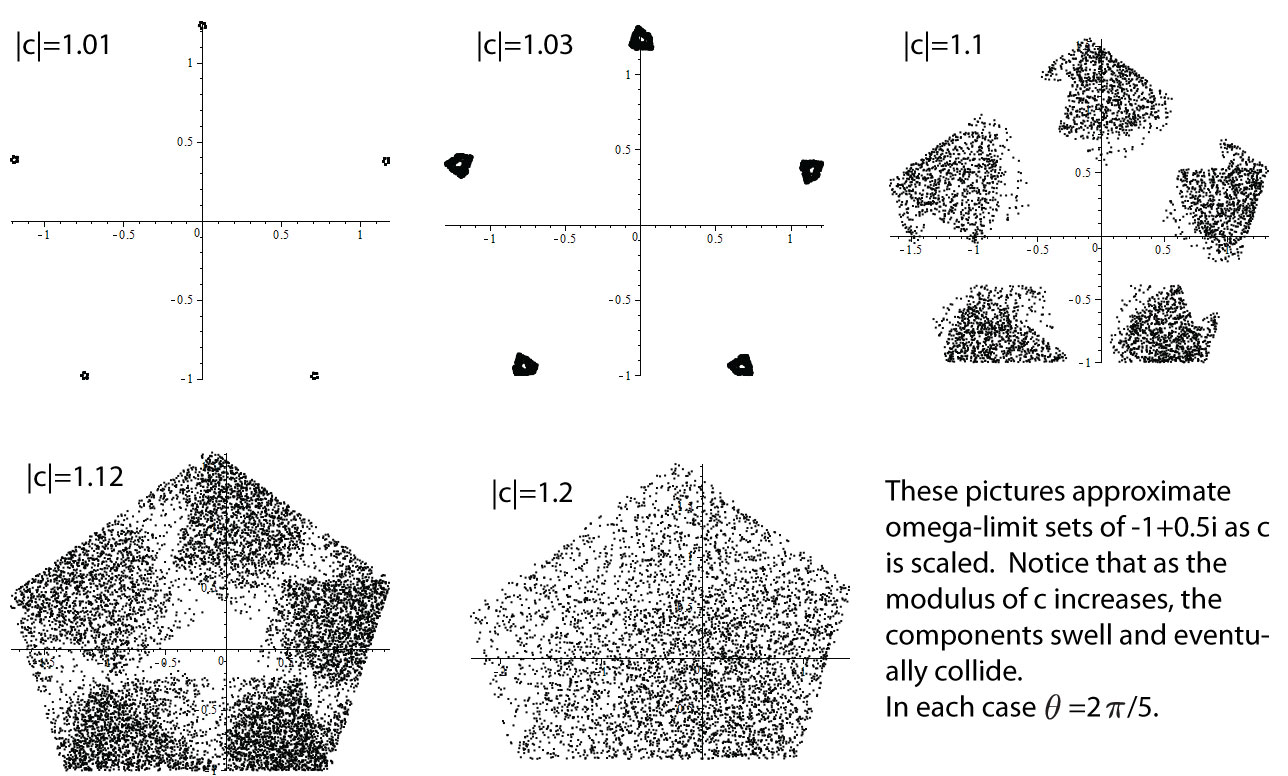}
	\caption[Each hungry set must intersect $ \FL$]{Each hungry set is a subset of $K$ and therefore is bounded.  Thus, each hungry set must intersect $ \FL$.  Also, as the modulus of $c$ increases, then $|cz|$ increases for every $z \in \FL$.  This is why the components of the hungry set swell as the modulus of $c$ increases.}
	\label{fig:Swell}
\end{figure}

	\subsection{Coexisting Disjoint Hungry Sets}

By Lemma~\ref{KisHungry}, any sub-K with nonempty interior is a hungry set.  Figure~\ref{fig:multipleHungrySets} shows an example of a $K$ where there is a hungry set near the origin and also a period 4 sub-K near the boundary.

\begin{figure}[htbp]
	\centering
		\includegraphics[width=1.00\textwidth]{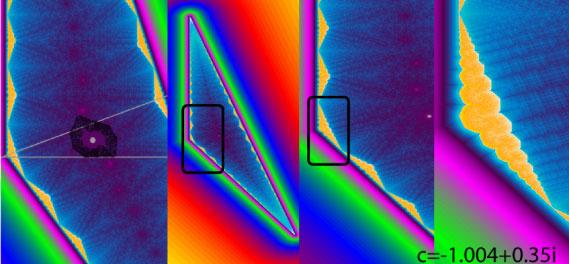}
	\caption[Using the omega-limit set of a point to find a hungry set]{The first image shows 8,000 points in the orbit of a point near the origin revealing a hungry set there.  The rest of the images show 3 progressively closer views of a component of a period 4 hungry set on the boundary of $K(-1.004+0.35i)$.}
	\label{fig:multipleHungrySets}
\end{figure}
\end{section}

\begin{section}{Partitioning the Parameter Space}\label{Partitioning the Parameter Space}

There are several different ways to draw pictures of the parameter space.  Each way has advantages and disadvantages.  We begin with the definition and picture of the polygonal locus.

\begin{Def}
The \textbf{polygonal locus} is the set of points in the parameter space for which $K$ is a polygon.
\end{Def}

By Theorems~\ref{thm:noZeta},~\ref{thm:zetaLow}, and~\ref{thm:zetaHigh}, $K$ is a polygon exactly when $\gamma_0 \in K$.  Thus, we can make an escape-time picture of the polygonal locus by using $\gamma_0$ as our test point.  Figure~\ref{fig:polygonalLocus} is the result.  Notice the vertical line at $\alpha=-1$.

\begin{figure}[htbp]
	\centering
		\includegraphics[width=0.70\textwidth]{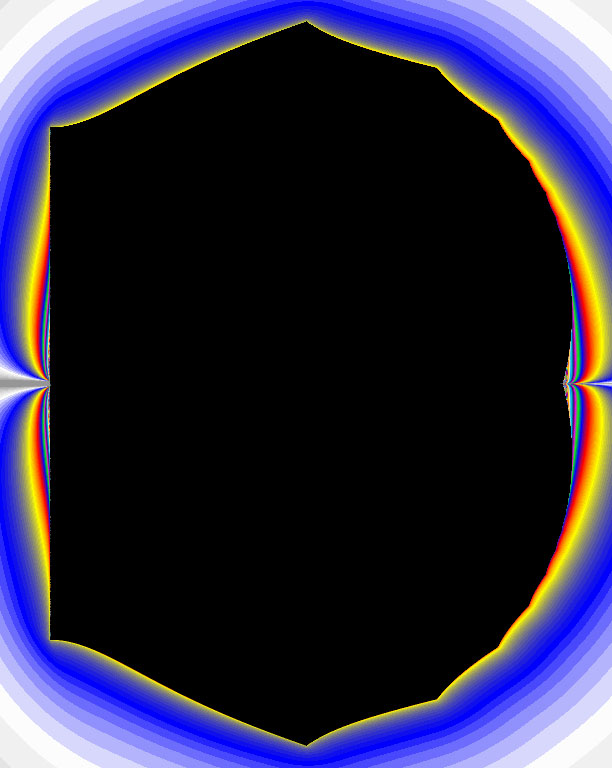}
	\caption{The polygonal locus.}
	\label{fig:polygonalLocus}
\end{figure}

The boundary of the polygonal locus consists of the union of many curves.  Every parameter on a curve has an associated $K(c)$ and for all $c$ on that curve, $K(c)$ will have the same number of sides.  The equations of these curves are implicitly defined by $c^n~\gamma_0=(1-t)\ell_0+t\ell_0/c$, that is, $f^n(\gamma_0)\in L_1$.

Figure~\ref{fig:bubbles} shows the unit disk centered at the origin and parameters where there is a smallest positive integer $n$ with $\ell_n \in \mathbb{H}^-.$  If $n$ is any integer, then the corresponding picture is Figure \ref{fig:ComplementOfCantorLocus2}.

\begin{figure}[htbp]
	\centering
	\includegraphics[width=0.80\textwidth]{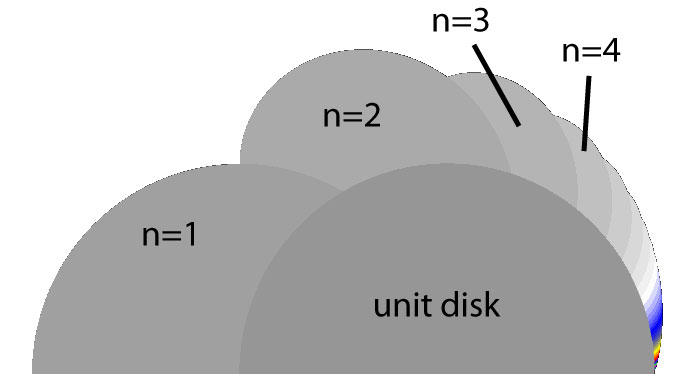}
	\caption[Points in the parameter plane where $\ell_n \in \mathbb{H}^-$]{Points in the parameter plane where $\ell_n$ is the first to be in the lower half plane.}
	\label{fig:bubbles}
\end{figure}

Another useful picture of the parameter plane is given in black and white in Figure~\ref{fig:Partitioning_the_parameter_space_BW} and in color in Figure~\ref{fig:Layered_Fractal}.  Each of these pictures mark the circles of radii $\sqrt{2}$ and 1, the polygonal locus, and the regions marked in Figure~\ref{fig:bubbles}.

\begin{figure}[htbp]
	\centering	\includegraphics[width=0.70\textwidth]{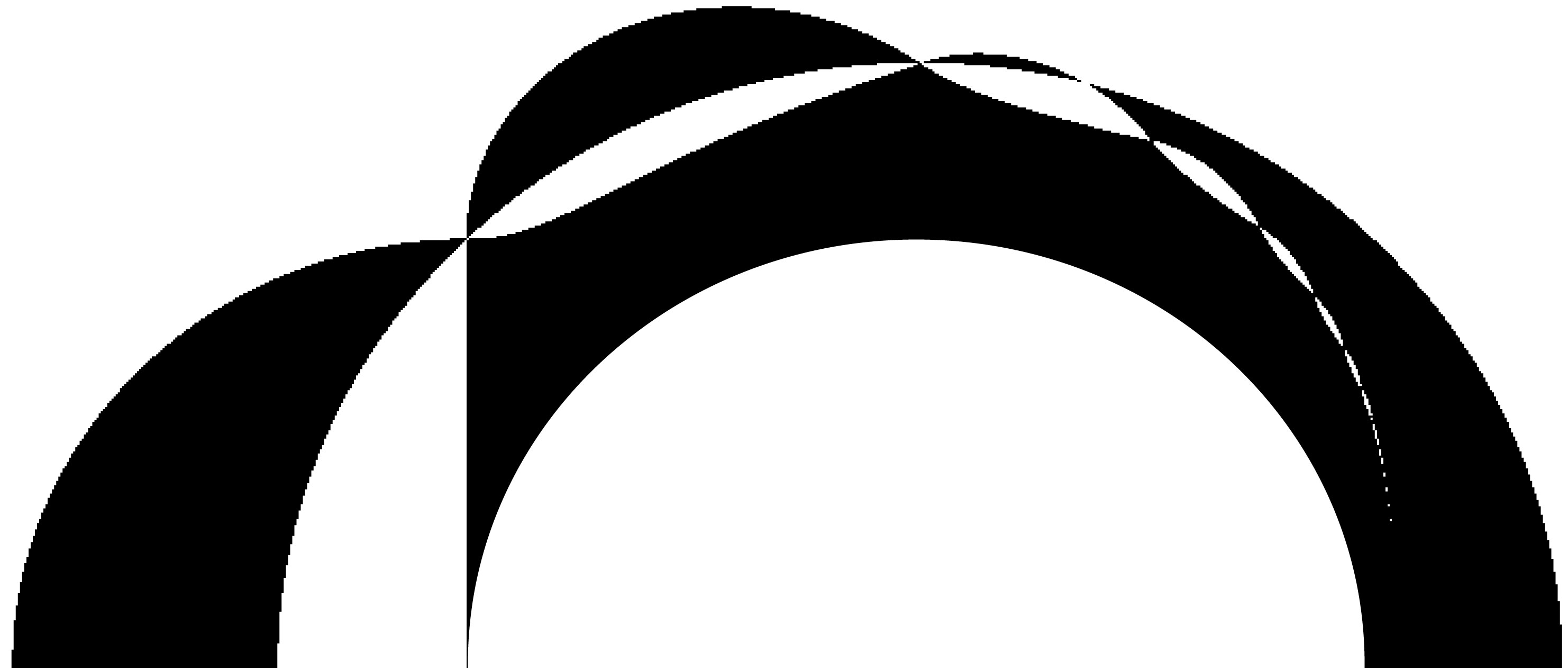}
	\caption[A black and white partition of the parameter plane]{A pretty partition of the portion of the parameter plane with positive imaginary part.}
	\label{fig:Partitioning_the_parameter_space_BW}
\end{figure}

\begin{figure}[htbp]
	\centering	\includegraphics[width=0.70\textwidth]{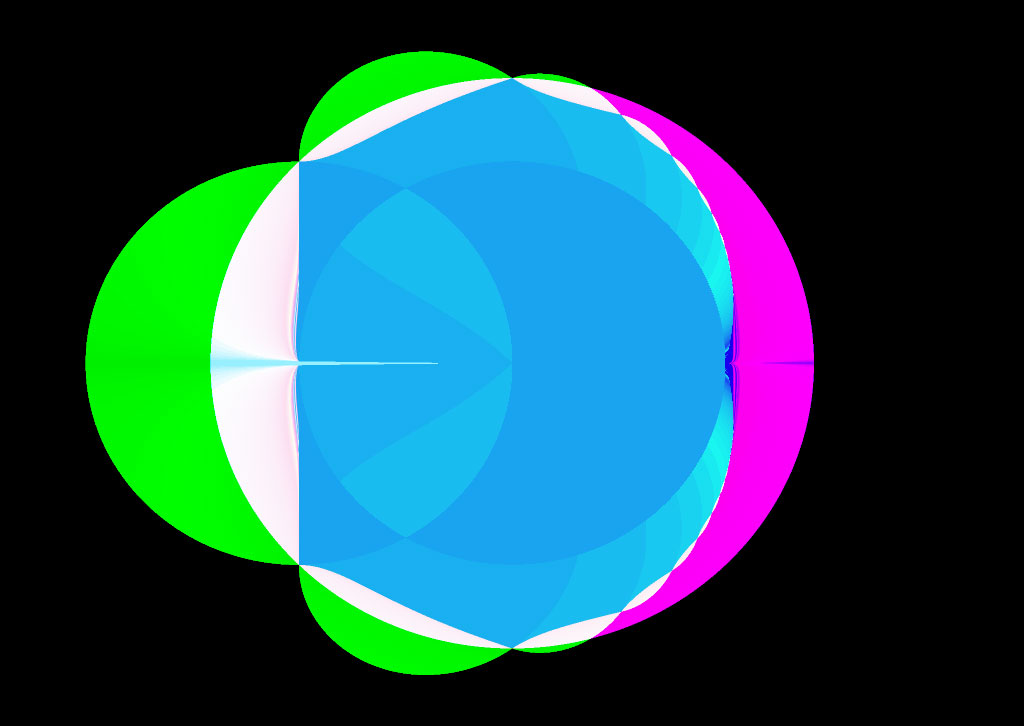}
	\caption{A colorful partition of the parameter plane.}
	\label{fig:Layered_Fractal}
\end{figure}

\begin{figure}
	\centering	\includegraphics[width=1.00\textwidth]{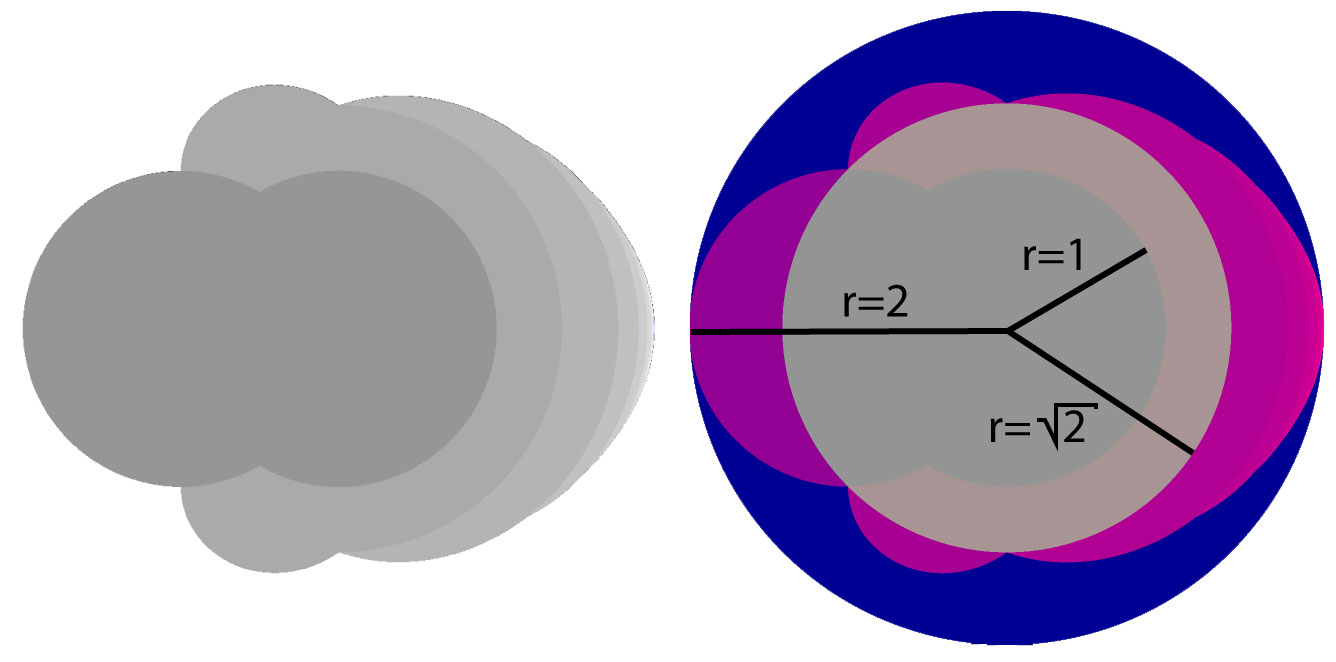}
	\caption[The conjectured complement of the Cantor Locus]{On the left is the set of parameters $c$ in the parameter plane such that $K(c)$ has at least one point in $\mathbb{H}^-$.  On the right is the same set along with 3 disks centered at the origin with labeled radii for reference.}
	\label{fig:ComplementOfCantorLocus2}
\end{figure}

The escape-time method of coloring the parameter plane provides visual representations of the external structures of $K(c)$ near a test point when $K(c)$ has empty interior.
Close inspection of Figure~\ref{fig:MandelbrotOutside} leads to Conjecture~\ref{K_Atlas_in_parameter_plane}.  Experiments suggest that the closer the test point is to the origin, the more ``little filled-in Julia sets'' appear along the outside of the polygonal locus.

\begin{conj}\label{K_Atlas_in_parameter_plane}
Drawing an escape-time picture of the parameter plane using a test point very close to 0 produces an atlas of $K$'s with empty interior all around the outside of the polygonal locus.  
\end{conj}

Using the coded-coloring instead of just the escape-time algorithm we can get a preview of the coded-coloring inside of $K(c)$ near the test point.  Figure \ref{fig:MandelbrotInside} shows that the coded-coloring of the parameter plane can help to find parameters $c$ such that $K(c)$ contains sub-K's.  

\begin{conj}
In a neighborhood of $c$ in the parameter plane, the coded-coloring using a test point very close to 0 produces an atlas of sub-K's which can be found in the coded-coloring of $K(c)$ near the test point.
\end{conj}

Figure~\ref{fig:SimilarSubKs} inspired the following question.
\begin{openq}
For which parameters are all the internal structures of $K$ near 0 the same?
\end{openq}

\begin{figure}
	\centering
		\includegraphics[width=1.00\textwidth]{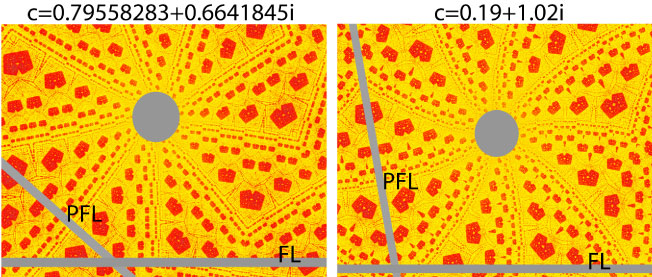}
	\caption{Very different $c$'s with similar structures inside of $K$.}
	\label{fig:SimilarSubKs}
\end{figure}

\begin{conj}
When making coded-colorings of the parameter plane, the smaller the modulus of the non-zero test point, the closer the similarity between the structure in $K(c)$ near the test point and the structure seen in the coded-coloring of the parameter plane near $c$.  Any sub-K's found in this way have sides that approach uniform length as $|c|$ goes to 1.
\end{conj}

The coded-coloring of the parameter plane in Figure~\ref{fig:paramPlane} reveals what resemble sun flares coming out of the unit disk.  We will refer to these as \textbf{flares}.  

\begin{figure}[htbp]
	\centering	\includegraphics[width=1.00\textwidth]{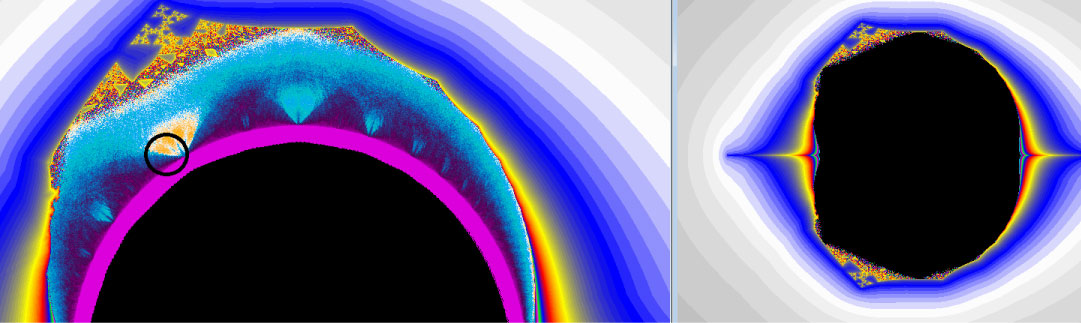}
	\caption[Escape-time and coded-colorings of the parameter plane]{Two pictures of the escape-time coloring of the parameter plane using $-i$ as the test point.  The picture on the left also uses the coded-coloring.}
	\label{fig:paramPlane}
\end{figure}

In Figure~\ref{fig:paramPlane} one of the flares is circled and Figure~\ref{fig:closeUpParamPlane} shows a close-up of this flare along with the coded-colorings of various $K(c)$ for six different choices of $c$ near this flare.  Experiments suggest that a flare is a set of parameters where there is an $n$ such that $K(c^n)$ is embedded in the dynamical structure of $K(c)$.  The larger the flare, and the sharper the image in the flare, the more ``stable'' the embedding.

\begin{figure}[htbp]
	\centering	\includegraphics[width=1.00\textwidth]{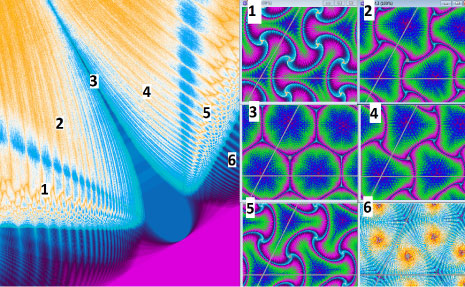}
	\caption[Around a flare in 80 days]{On the left is a close up of a flare in the coded-coloring of the parameter plane.  There are six locations marked on the left corresponding the the approximate parameters used in coded-colorings of the six $K(c)$'s on the right.}
	\label{fig:closeUpParamPlane}
\end{figure}

Now each flare has a central curve.  Figure~\ref{fig:closeUpParamPlane} shows that the closer a parameter is to this central curve (near location 3), the straighter $K(c^n)$ tends to be.  This suggest that the central curves of flares meet the unit circle at points whose argument is a rational multiple of $\pi$.  

\begin{conj}
Flares in the parameter plane meet the centered unit disk at points whose argument is a rational multiple of $\pi$.
\end{conj}
\end{section}

\begin{section}{Entropy}\label{Entropy}
In this chapter, we will calculate topological entropy as defined by Adler, Konheim and McAndrew.  We will follow closely what is done in \cite[page 188]{ALM}.

\begin{Def}
We will denote the cardinality of a set $A$ by $\#(A)$.
\end{Def}

The following Lemma is Corollary 2.2 in \cite[pg 829]{M}.

\begin{lem}\label{Entropy and Hausdorff dimension}
Let $X$ be a nonempty compact metric space and $f:X\to X$ a Lipschitz continuous map with Lipschitz constant $L$.  Then the Hausdorff dimension of $X$ is larger than or equal to $\frac{h(f)}{\log L}$.
\end{lem}

\begin{cor}\label{cor Entropy and Hausdorff dimension}
We have $h(f|_K)\leq \log(|c|^2).$
\end{cor}
\begin{proof}
The map $f$ is a Lipschitz continuous map with Lipschitz constant $|c|$.  Since $K$ is a polygon, then $K\subset \mathbb{C}$ is a nonempty compact metric space.  Since $K \subset \mathbb{C}$, then clearly the Hausdorff dimension of $K$ is less than or equal to $2$.  We then get our result by applying Lemma~\ref{Entropy and Hausdorff dimension}.
\end{proof}

\begin{Def}
Let $\mathcal{A},$ $\mathcal{B}$ be open covers of a space $X$ and let $f:X\rightarrow X$ be a continuous map.  The common refinement of $\mathcal{A}$ and $\mathcal{B}$ is $\mathcal{A}~\vee~\mathcal{B}=\{A\cap~B:A\in~\mathcal{A},B\in~\mathcal{B}\}.$  Let $f^{-n}(\mathcal{A})=\{f^{-n}(A):A\in \mathcal{A}\}$.  For every positive integer $n$, we define the $n^{th}$ common refinement of $\mathcal{A}$ by $\mathcal{A}^n=\mathcal{A}\vee~f^{-1}(\mathcal{A})\vee~...~\vee~f^{n-1}(\mathcal{A})$.
\end{Def}

It will be important to note that partitions are covers.

\begin{lem}\label{bound from below for h}
Assume that $K$ is a polygon and let $\mathcal{A}=\{P\mathbb{H}^+\cap K, (P\mathbb{H}^- \setminus \PFL)\cap K\}$.  Then $h(f|_K,\mathcal{A})\geq \log(|c|^2)$.
\end{lem}
\begin{proof}
Let $A\in \mathcal{A}^n$ and let $\lambda$ be the 2-dimensional Lebesgue measure.  Then $f^n$ is one-to-one on $A$ and so $\lambda(f^n(A))=(|c|^2)^n \lambda(A)$.  Also, $f^n(A)\subset f(K) \subset K$.  Thus $\lambda(K) \geq \lambda(f^n(A))=(|c|^2)^n \lambda(A)$.  This gives an upper estimate for $\lambda(A)$, namely:
\begin{equation}
\begin{aligned}
\lambda(A) \leq \frac{\lambda(K)}{(|c|^2)^n}.\\
\end{aligned}
\end{equation}
  
Also, $\underset{A \in \mathcal{A}^n}{\bigcup}A=K$.  And so $\underset{A\in \mathcal{A}^n}{\Sigma}\lambda(A)=\lambda(K).$  Using our upper estimate for $\lambda(A)$ we get $\#(\mathcal{A}^n)\frac{\lambda(K)}{(|c|^2)^n}\geq \lambda(K)$.  Since $\lambda(K)>0$ then dividing both sides by $\lambda(K)$ results in:
\begin{equation}
\begin{aligned}
\#(\mathcal{A}^n)\geq (|c|^2)^n. \\
\end{aligned}
\end{equation}

Taking the logarithm of both sides and dividing by $n$ we get: 
\begin{equation}
\begin{aligned}
\frac{1}{n}\log(\#(\mathcal{A}^n))\geq \frac{1}{n}\log((|c|^2)^n).\\
\end{aligned}
\end{equation}   

Taking the limit as $n$ goes to infinity we get $h(f,\mathcal{A}) \geq \log(|c|^2)$.  That is, the topological entropy of $f$ with respect to the partition $\mathcal{A}$ is greater than or equal to $\log(|c|^2).$
\end{proof}

\begin{lem}\label{valency of a partition}
Assume that $K$ is a polygon.  Let $\mathcal{A}=\{P\mathbb{H}^+\cap K, (P\mathbb{H}^- \setminus \PFL)\cap K\}$.  Then for any $z \in K$, $\#(\{A\in \mathcal{A}^n: z \in \Bd(A)\})\leq 4n+6$.
\end{lem}
\begin{proof}
Let $z \in K$ be fixed and let $N(z,n)=\#(\{A\in \mathcal{A}^n: z \in \Cl(A)\})$.  If $z$ is in the interior of $A \in \mathcal{A}^n$, then $N(z,n)=1$ and we are done.  It is clear that $\mathcal{A}^n$ is a finite set.  Thus, for a small enough neighborhood $U$ of $z$, if $A \in \mathcal{A}^n$ and $A \cap U\neq \emptyset$ then $z\in \Cl(A)$.

If $z \notin \Bd(K)$ then $N(z,n)$ is equal to twice the number of the first $n$ preimages of $\FL$ that are not collinear and meet at $z$.  Obviously, the number of preimages of $\FL$ that can meet at a point are bounded from above by the number of distinct angles a preimage of $ \FL$ can take.  We now show that there cannot be more than $2n+1$ distinct angles a preimage of $ \FL$ can take.

Recall that $c=|c|e^{i\theta}$.  \textbf{To simplify the calculations, in this proof we will measure angles relative to the negative real axis in a clockwise manner.}  Note that under this temporary convention, the angle of $\PFL$ relative to the real axis is $\theta$, and division by $c$ adds $\theta$ to the argument of every point.

Now $f^{-2}( \FL)=f^{-1}(\PFL)$ consists of two rays which are constructed by removing the open lower half-plane, unfolding a copy of the upper half-plane onto the lower half-plane, and then dividing by $c$.  Thus, if $W^n=\{w_1,w_2,...,w_k\}$ is the set of all angles achieved in the first $n$ preimages of $\PFL$, then $W^{n+1}=((-W^n)\cup W^n)+\theta=(\theta-W^n)\cup (\theta+W^n)\cup W^n$.  We now list the first few angle sets.
\begin{equation}
\begin{aligned}
W^0 &=\{\theta\}\\
W^1 &=\{0,\theta,2\theta\}\\
W^2 &=\{-\theta,0,\theta,2\theta,3\theta\}\\
W^3 &=\{-2\theta,-\theta,0,\theta,2\theta,3\theta,4\theta\}\\
.&\\
.&\\
.&\\
W^{n+1} &=(\theta-W^n)\cup (\theta+W^n)\\
\end{aligned}
\end{equation}

Note that $\#(W^n)=2(n)+1$ for these first few sets.  We now show that this growth is true for all $n\geq 0$.  Now suppose that $w=k\theta,k \in \mathbb{Z}$ and $W^n=\{(x-k)\theta:x=0,1...,n+1+k\}$.  Letting $v=w+\theta$ we have:

\begin{equation}
\begin{aligned}
W^{n+1}&=(\theta-W^n)\cup (\theta+W^n)\\
&=\{\theta-(w+2\theta),\theta-w,...,\theta+w-2\theta,\theta+w\}\\
&~~~\cup \{\theta-w,\theta-w+2\theta,...,\theta+w,\theta+w+2\theta\}\\
&=\{-w-\theta,-w+\theta,-w+3\theta,...,w+\theta,w+3\theta\}\\
&=\{-v,-v+2\theta,v+4\theta,...,v,v+2\theta\}.\\
\end{aligned}
\end{equation}
Thus by induction, $\#(W^{n+1})=\#(W^n)+2$ and we have $\#(W^n)=2n+1.$  Therefore, $N(z,n)\leq 2(\#(W^n))=2(2n+1)=4n+2.$

Lastly, if $z\in \Bd(K)$ then at most two $A \in \mathcal{A}^n$ have a common boundary with $K$ in the small neighborhood $U$ of $z$.  Since the sides of the polygon $K$ do not need to be parallel to a preimage of $ \FL$, then this may add at most two more angles to consider.  Thus, for all $z \in K$ we have $N(z,n) \leq 2(2n+3)=4n+6$.
\end{proof}

\begin{lem}\label{open cover for h}
There exists an open cover $\mathcal{B}$ of $\mathcal{A}^n$ such that each element of $\mathcal{B}$ intersects at most $k(n)$ elements of $\mathcal{A}^n$ where $k(n)$ grows linearly.
\end{lem}
\begin{proof}
Let $B_1=\{\Int(A):A\in \mathcal{A}^n\}$ and let $B_2=\{\Int(\Cl(A_j \cup A_k)):A_j,A_k\in~\mathcal{A}^n\}$.  Then $B_1 \cup B_2$ is a cover of $K$ everywhere but at the locations where more than 2 elements of $\mathcal{A}^n$ meet.  Since $\mathcal{A}^n$ is a finite set, then these locations are isolated and so can be covered by disjoint open sets.  By Lemma~\ref{valency of a partition} each of these open set intersects at most $4n+6$ elements of $\mathcal{A}^n$ which grows linearly.
\end{proof}

\begin{thm}
If $K$ is a polygon then $h(f|_K)=\log(|c|^2)$.
\end{thm}
\begin{proof}
Assume that $K$ is a polygon.  
We define a partition of $K$ by $\mathcal{A}=\{P\mathbb{H}^+\cap K, (P\mathbb{H}^- \setminus \PFL)\cap K\}$.  Then $\mathcal{A}_f^n=\mathcal{A} \vee f^{-1}(\mathcal{A})\vee...\vee f^{-(n-1)}(\mathcal{A})$ is the nth-common refinement of $\mathcal{A}$ with respect to $f$.  By Lemma~\ref{open cover for h} there exists an open cover $\mathcal{B}$ such that each element of $\mathcal{B}$ intersects at most $k(n)$ elements of $\mathcal{A}^n$.  Let $\mathcal{C}$ be a minimal subcover of $\mathcal{B}^m.$  Each element of $\mathcal{C}$ intersects at most $(k(n))^m$ elements of $(\mathcal{A}_f^n)_{f^n}^m$.  The total number of elements of $(\mathcal{A}_f^n)_{f^{n}}^m$ is less than or equal to $(\#\mathcal{C})(k(n))^m$.  We get:
\begin{equation}
\begin{aligned}
\mathcal{N}((\mathcal{A}_{f^n}^n)_{f^n} ^m)&=(\#\mathcal{A}_{f^n})_{f^n}^m\\
& \leq \mathcal{N}(\mathcal{B}_{f^n}^m)(k(n))^m\\
\end{aligned}
\end{equation}
Taking the logarithm of both sides and the limit as $n$ goes to infinity we get:
\begin{equation*}
h(f^n,\mathcal{A}^n) \leq h(f^n, \mathcal{B})+\log k(n)\leq h(f^n)+\log k(n).\\
\end{equation*}

Using the identity that for any partition $\mathcal{A}$ we have $h(f^n, \mathcal{A}^n)=nh(f,\mathcal{A})$ and dividing through by $n$ we get $h(f,\mathcal{A}) \leq h(f)+\frac{1}{n}\log k(n).$  As $n$ goes to infinity, $\frac{1}{n}\log k(n)$ goes to 0 and so $h(f,\mathcal{A}) \leq h(f)$.  Then by Lemma~\ref{bound from below for h} $\log(|c|^2)\leq h(f,\mathcal{A}) \leq h(f)$.  By Corollary~\ref{cor Entropy and Hausdorff dimension} $h(f)\leq \log(|c|^2).$  Thus $h(f)=\log(|c|^2).$
\end{proof}

\begin{prop}\label{entropy upper bound}
We have $h(f|_K)\leq \log2$.
\end{prop}
\begin{proof}
Let $\mathcal{A}=\{P\mathbb{H}^+\cap K, (P\mathbb{H}^- \setminus \PFL)\cap K\}$ and let $\mu$ be any $f$-invariant probability measure on $K$.  Then $\mathcal{A}$ is a one sided generator with two elements.  Thus, $h_\mu (f) =h(f,\mathcal{A})\leq \log2$.  Now the variational principle states that $h(f)\leq \sup\{h_\mu (f)\}$ where the supremum is taken over all $f$-invariant probability measures on $K$.  We have $h(f)\leq \log2$.
\end{proof}

\begin{thm}\label{strict entropy}
We have $h(f|_K) \leq \log \min(2,|c|^2)$.  Furthermore, the inequality is sometimes strict.
\end{thm}
\begin{proof}
Let $c \in \mathbb{R}$ where $1<c<2$.  Then by Theorem~\ref{K for real c}, $K$ is a compact subset of the imaginary axis.  By Lemma~\ref{conjugate to real tent map}, $f$ restricted to $K$ is conjugate to a real tent map $\tau$\label{symbol:tau} from $[0,1]$ to itself, where the absolute value of the slope is $|c|$.  Thus, $h(f)=h(\tau)$.  It is well known that if $1<|c|<2$, then $h(\tau)= \log|c|$.  Since $|c|< \min(2,|c|^2)$, then $h(f|_K) \leq h(f)=h(\tau)=\log|c| < \log \min(2,|c|^2)$ and we are done.
\end{proof}
  
\end{section}

\end{document}